\documentclass[12pt]{amsart}
\usepackage{amssymb,amscd}
\usepackage{verbatim}

\usepackage{amsmath,amssymb,graphicx,mathrsfs,bbm}   
\usepackage{enumerate}
\usepackage[colorlinks=true,allcolors = blue]{hyperref} 

\usepackage{tikz}
\usetikzlibrary{matrix}
\usetikzlibrary{patterns,arrows,decorations.pathreplacing}

\usepackage[all]{xy}

\let\frak\mathfrak
\let\Bbb\mathbb

\def\>{\relax\ifmmode\mskip.666667\thinmuskip\relax\else\kern.111111em\fi}
\def\<{\relax\ifmmode\mskip-.333333\thinmuskip\relax\else\kern-.0555556em\fi}
\def\vsk#1>{\vskip#1\baselineskip}
\def\vv#1>{\vadjust{\vsk#1>}\ignorespaces}
\def\vvn#1>{\vadjust{\nobreak\vsk#1>\nobreak}\ignorespaces}

  \let\ssize\scriptstyle
\let\sssize\scriptscriptstyle

\let\Medskip\medskip
\def\medskip{\par\Medskip}
\let\Bigskip\bigskip
\def\bigskip{\par\Bigskip}

\let\Maketitle\maketitle
\def\maketitle{\Maketitle\thispagestyle{empty}\let\maketitle\empty}

\newtheorem{thm}{Theorem}[section]
\newtheorem{cor}[thm]{Corollary}
\newtheorem{lem}[thm]{Lemma}
\newtheorem{prop}[thm]{Proposition}

\theoremstyle{definition}                                  
\numberwithin{equation}{section}

\theoremstyle{definition}
\newtheorem*{rem}{Remark}

\newcommand{\nc}{\newcommand}
\nc{\Z}{\mathbb{Z}}
\nc{\Q}{\mathbb{Q}}
\nc{\C}{\mathbb{C}}
\nc{\R}{\mathbb{R}}

\nc{\G}{\mathbb{G}}
\nc{\N}{\mathbb{N}}
\nc{\V}{\mathbb{V}}
\nc{\D}{\mathbb{D}}
\nc{\dc}{\Delta}

\nc{\ms}{\mathscr}
\nc{\mr}{\mathrm}
\nc{\mc}{\mathcal}
\nc{\mt}{\mathtt}
\nc{\msf}{\mathsf}
\nc{\mf}{\mathfrak}
\nc{\mb}{\mathbb} 
\newcommand{\ml}{\mathcal{L}}

\nc{\wtid}{\widetilde{\mathcal{D}}}
\nc{\wtic}{\widetilde{\mathcal{C}}}
\nc{\wh}{\widehat}
\nc{\ti}{\widetilde}

\nc{\rat}{\mt{rat}}

\nc{\la}{\lambda}

\nc{\msp}{\ms{P}}

\nc{\ve}{\varepsilon}
\nc{\vpi}{\varpi}

\nc{\sm}{\setminus}

\nc{\dbc}{\mr{D}^b_c}
\nc{\dpl}{\mr{D}^+}
\nc{\mrd}{\mr{D}}
\nc{\db}{\mr{D}^b}

\nc{\msm}{\ms{M}}

\nc{\msk}{\msf{k}}

\nc{\h}{\mbox{-}}

\nc{\bs}{\backslash}
\nc{\ov}{\overline}

\nc{\rank}{\mr{rank}}
\newcommand\im{\operatorname{im}}
\newtheorem{remark}[equation]{Remark}
\newtheorem{proposition}[equation]{Proposition}
\newtheorem*{claimn}{Claim}
\newtheorem{lemma}[equation]{Lemma}

\let\mc\mathcal
\let\nc\newcommand

\let\al\alpha

\let\dl\delta

\let\ka\kappa
\let\la\lambda

\let\phi\varphi
\let\si\sigma

\let\om\omega
\let\Om\Omega

\let\der\partial

\let\ox\otimes

\let\geq\geqslant

\let\leq\leqslant

\let\on\operatorname
\let\bi\bibitem
\let\bs\boldsymbol

\def\C{{\mathbb C}}
\def\Z{{\mathbb Z}}
\def\R{{\mathbb R}}

\def\F{{\mathbb F}}   

\def\V{{\mc V}}

\def\+#1{^{\{#1\}}}

\def\Gr{\on{Gr}}

\def\rank{\on{rank}}

\def\beq{\begin{equation}}
\def\eeq{\end{equation}}
\def\be{\begin{equation*}}
\def\ee{\end{equation*}}

\nc{\bea}{\begin{eqnarray*}}
\nc{\eea}{\end{eqnarray*}}
\nc{\bean}{\begin{eqnarray}}
\nc{\eean}{\end{eqnarray}}

\def\h{{\mathfrak h}}

\let\ga\gamma

\nc{\Il}{{\mc I_{\bs\la}}}
\nc{\bla}{{\bs\la}}
\nc{\Fla}{\F_\bla}
\nc{\tfl}{{T^*\Fla}}
\nc{\GL}{{GL_n(\C)}}
\nc{\GLC}{{GL_n(\C)\times\C^*}}

\let\sd s 

\def\ddk_#1{\kk_{#1}\<\>\frac\der{\der\<\>\kk_{#1}}}

\def\bul{\mathbin{\raise.2ex\hbox{$\sssize\bullet$}}}
\def\intt{\mathchoice
{\mathop{\raise.2ex\rlap{$\,\,\ssize\backslash$}{\intop}}\nolimits}
{\mathop{\raise.3ex\rlap{$\,\sssize\backslash$}{\intop}}\nolimits}
{\mathop{\raise.1ex\rlap{$\sssize\>\backslash$}{\intop}}\nolimits}
{\mathop{\rlap{$\sssize\<\>\backslash$}{\intop}}\nolimits}}

\let\kk q 
\let\cc c

\let\Ko K

\def\GZ/{Gelfand-Zetlin}
\def\KZ/{{\slshape KZ\/}}
\def\qKZ/{{\slshape qKZ\/}}
\def\XXX/{{\slshape XXX\/}}

\def\Sym{\on{Sym}}

\nc{\A}{{\mc A}}

\def\Sing{{\on{Sing}}}
\def\sll{{\frak{sl}}}

\def\slt{{\frak{sl}_2}}

\def\Ik{{\mc I_r}}
\def\Q{{\mathbb Q}}
\def\K{{\mathbb K}}

\nc{\hsl}{\widehat{{\frak{sl}_2}}}

\nc{\BC}{{ \mathbb C}}
\nc{\lra}{\longrightarrow}
\nc{\CO}{{\mathcal{O}}}
\nc{\BZ}{{ \mathbb Z}}
\nc{\hfn}{\hat{\frak{n}}}
\nc\Zs{{\Z/p^s\Z}}
\nc\Zo{{\Zs[z]^0}}
\nc\gr{{\on{gr}}}

\nc\fD{{\frak D}}

\let\wh\widehat

\nc\kka{{\tilde \ka}}

\nc\Tt{{\bf T}}

\begin{document}

\hrule width0pt
\vsk->

\title[$p$-curvature operators]
{$p$-curvature operators and Satake-type phenomenon
\\
for $\frak{sl}_2$ \KZ/ equations with $\ka=\pm 2$}

\author[P.~Belkale, E.~Mukhin, A.~Varchenko ]
{P.~Belkale$^*$, E.~Mukhin$^\circ$, A.~Varchenko$^{\star}$}

\maketitle

\begin{center}

{\it $^{* \star}$Department of Mathematics, University
of North Carolina at Chapel Hill\\ Chapel Hill, NC 27599-3250, USA\/}

\vsk.5>

{\it $\kern-.4em^\circ\<$Department of Mathematical Sciences,
Indiana University  Indianapolis\kern-.4em\\
402 North Blackford St, Indianapolis, IN 46202-3216, USA\/}

\end{center}

\vsk>
{\it Key words\/}:  
\KZ/ equations,  $p$-hypergeometric solutions,  $p$-curvature
\vsk.2>

2020 Mathematics Subject Classification: 11D79, 12H25 (32G34, 33C05, 33E30, 33E50)



{\let\thefootnote\relax
\footnotetext{\vsk-.8>\noindent
$^*\<${\sl E\>-mail}:\enspace belkale@email.unc.edu\>,
supported in part by NSF grant DMS - 2302288
\\
$^\circ\<${\sl E\>-mail}:\enspace emukhin@iupui.edu\>,
supported in part by Simons Foundation grant \rlap{709444}
\\
$^\star\<${\sl E\>-mail}:\enspace anv@email.unc.edu\>}}

\begin{abstract}

The $\frak{sl}_2$ \KZ/ differential equations with values in the tensor power of the fundamental representation 
with parameter $\ka=\pm 2$ are considered. 
A Satake-type correspondence  is established over complex 
numbers and subsequently  reduced to finite characteristic. 
This correspondence enables the study of the \KZ/ equations on the lower weight subspaces of the tensor power
in terms of the wedge powers of the weight subspace of the weight  just below the highest weight. 

We apply this approach to analyze the $p$-curvature operators associated with our \KZ/ equations, evaluate
 the dimension  of the solution space in characteristic $p$, 
 and  determine whether all solutions are generated by the so-called $p$-hypergeometric solutions. In particular, we 
 show that not all solutions  of the \KZ/   equations with $\ka=2$ in characteristic $p$ 
 are generated by $p$-hypergeometric solutions.  Previously, no such examples were known.

 \end{abstract}

{\small\tableofcontents\par}

\setcounter{footnote}{0}
\renewcommand{\thefootnote}{\arabic{footnote}}

\section{Introduction}

The \KZ/ equations form a system of linear partial differential equations
 that exhibit profound connections with theoretical physics, representation theory, algebraic geometry,
and integrable systems. In 1984, Vladimir Knizhnik and Alexander Zamolodchikov first introduced
these equations within the context of conformal field theory to depict the behavior of correlation
functions of vertex operators in two-dimensional quantum field theories, \cite{KZ}. Since then, the
equations have emerged
as a central object in both mathematical physics and pure mathematics, bridging seemingly disparate
fields such as Lie algebras, braid groups, hypergeometric integrals, congruences in number theory, 
and geometric topology.

In this paper we consider the example of the \KZ/ differential equations associated with the tensor power
$L^{\ox(2g+1)}$ of the two-dimensional irreducible representation $L$ of the Lie algebra $\frak{sl}_2$ 
with standard generators $e,f,h$.
The \KZ/ equations  are the system of differential equations on an $L^{\ox(2g+1)}$-valued function
$N(z)$, $z=(z_1,\dots, z_{2g+1})$,
\bean
\label{KZ int}
\qquad
\left(\frac{\der}{\der z_i} - \frac 1\ka\sum_{j\ne i}\frac{ P^{(i,j)}-1}{z_i-z_j}\right) \!
N(z) =0, 
\qquad  i = 1, \dots , 2g+1, 
\eean
where  $P^{(i,j)}$ is the permutation of the $i$-th and $j$-th factors of $L^{\ox(2g+1)}$ 
and $\ka$ is a scalar parameter of the equations.  

The differential operators $\nabla ^{\on{KZ},\ka}_i = 
\der_i - \frac 1\ka\sum_{j\ne i}\frac{ P^{(i,j)}-1}{z_i-z_j}$ define the \KZ/ connection on the trivial
bundle  $L^{\ox(2g+1)}\times \mathbb A^{2g+1} \to \mathbb A^{2g+1}$, where 
$\mathbb A^{2g+1}$ is the affine space with coordinates $z_1,\dots,z_{2g+1}$. The connection is flat
for any value of $\ka$ and has singularities over the union of all diagonal hyperplanes in $\mathbb A^{2g+1}$.

As a rule one considers the \KZ/  equations over the field of complex numbers.
Then the \KZ/ differential  equations have  multidimensional hypergeometric
solutions,  \cite{SV1}.

\vsk.2>
 
More precisely, the
 differential operators $\nabla ^{\on{KZ},\ka}_i$ commute with the $\frak{sl}_2$-action on 
 $L^{\ox(2g+1)}$. Hence the \KZ/ connection 
 has invariant subbundles
 $\Sing L^{\ox(2g+1)}[2g+1-2r]\times \mathbb A^{2g+1} \to \mathbb A^{2g+1}$ where
$\Sing L^{\ox(2g+1)}[2g+1-2r] \subset L^{\ox(2g+1)}$ is the subspace of vectors $v$ such that 
$ev=0$ and $hv= (2g+1-2r)v$.  Here $r=0,\dots, g$.  All $L^{\ox(2g+1)}$-valued solutions of 
the \KZ/ equations are generated by solutions with values in these subbundles.

\vsk.4>

We have
$\dim \Sing L^{\ox(2g+1)}[2g+1-2r] = \binom{2g}{r} - \binom{2g}{r-2}$.
For $r=0$,  $\dim \Sing L^{\ox(2g+1)}[2g+1]=1$, and the \KZ/ equations on a
$\Sing L^{\ox(2g+1)}[2g+1]$-valued (scalar) function $N(z)$ take the form
$\frac{\der}{\der z_i}N(z) = 0$, $i=1,\dots, 2g+1$, with obvious solutions.
 (In this case $P^{(i,j)}-1=0$.)

\vsk.4>

We have $\dim \Sing L^{\ox(2g+1)}[2g-1]=2g$ if  $r=1$. 
In a natural basis of $L^{\ox(2g+1)}$, the 
vectors  of $\Sing L^{\ox(2g+1)}[2g-1]$ are identified with 
$2g+1$-vectors $(N_1,\dots,N_{2g+1})$ with  $\sum_i N_i$=0, and the $\Sing L^{\ox(2g+1)}[2g-1]$-valued
hypergeometric solutions   are given by the vectors of one-dimensional hyperelliptic
integrals,
\bean
N^\ga(z) = \int_{\ga(z)} \prod_{i=a}^{2g+1}(t_1-z_a)^{-1/\ka}
\left(\frac 1{t_1-z_1}, \dots, \frac 1{t_1-z_{2g+1}}\right) dt_1
\eean
where $\ga(z)$ is a flat family of  Pochhammer cycles. The solutions depend on the choice of Pochhammer cycles. All
$\Sing L^{\ox(2g+1)}[2g-1]$-valued solutions are linear combinations of such vectors of integrals.

For $r>1$, the $\Sing L^{\ox(2g+1)}[2g+1-2r]$-valued hypergeometric solutions solutions are given by analogous 
$r$-dimensional hypergeometric  integrals,    \cite {CF, DJMM, SV1, V2, V3},  but considerably
more complicated. 

\vsk.2>

 Notice that  solutions of \KZ/ equations over complex numbers have remarkable properties. 
 For example, it was shown  in  \cite {L2,BBM,BF,BFM} recently,
  that when $\kappa$ is an arbitrary nonzero rational number, the spaces of solutions of \KZ/ equations 
carry the structure of variations of mixed Hodge structures with coefficients in a cyclotomic field.

\vsk.2>
In this paper we study the \KZ/ differential equations with parameters $\ka=2$ and $-2$.

\vsk.2>

Our first result, Theorem \ref{cor T(z)} and formula \eqref{T(z)}, describes
 the  $\Sing L^{\ox(2g+1)}[2g+1-2r]$-valued solutions of the \KZ/ equations with $\ka=2$ 
in terms of 
the $r$-fold wedge products of the 
$\Sing L^{\ox(2g+1)}[2g-1]$-valued solutions with $\kappa=2$.
The transformation of  solutions  is given by a  linear map
$T(z) : \wedge^r\Sing L^{\ox(2g+1)}[2g-1] \to \Sing L^{\ox(2g+1)}[2g+1-2r]$
 which is a homogeneous polynomial in $z_1,\dots,z_{2g+1}$ of degree $\binom{r}{2}$.

We
construct a natural direct decomposition of the space 
$\wedge^r\Sing L^{\ox(2g+1)}[2g-1] $ with the first direct summand denoted by
 $ \mc P_r^{\on{KZ}}(z)$ and called the primitive part. We show that
$T(z) \vert_{ \mc P_r^{\on{KZ}}(z) } :  \mc P_r^{\on{KZ}}(z)  \to 
 \Sing L^{\ox(2g+1)}[2g+1-2r]$ is an isomorphism of the corresponding \KZ/ connections with $\ka=2$ and show that 
 $T(z)$ restricted to other direct summands equals zero.

\vsk.2>

Theorem \ref{thm 5.4} describes an analogous statement for
 $\ka=-2$. In that case, we construct 
a map $\bar T(z) :   \Sing L^{\ox(2g+1)}[2g+1-2r] \to \wedge^r \Sing L^{\ox(2g+1)}[2g-1]$ which defines an embedding
of the corresponding \KZ/ connections with $\ka=-2$. 
The map $\bar T(z)$ is a homogeneous polynomial in $z_1,\dots,z_{2g+1}$ of degree $\binom{r}{2}$.

Formula \eqref{bar TTT}  presents a useful map
 $\tilde T(z) :   \Sing L^{\ox(2g+1)}[2g+1-2r] \to \wedge^r \Sing L^{\ox(2g+1)}[2g-1]$.
 We show that $\tilde T(z)$ equals $\bar T(z)$ 
up to a nonzero multiplicative constant,  Theorem \ref{thm tilde T}.

\begin{rem}

The 
 quantum Satake principle in \cite{GM, GGI} says that
 the quantum connection on the Grassmannian $\Gr(r,n)$ is
  the $r$-th wedge of the quantum connection on the projective space
$\on{P}^{n-1}$.  
The quantum connection on $\Gr(r,n)$ is a version of the KZ connection on
the weight space $L^{\ox n}[n-2r]$, while the
quantum connection on $\on{P}^{n-1}$ is a version of the KZ connection on
the weight space $L^{\ox n}[n-2]$, see, for example,
\cite{MO, GRTV, CDG, CV}.  Therefore our Theorems \ref{cor T(z)} and \ref{thm 5.4} 
exhibit  a flavor of the quantum Satake phenomenon.

\end{rem}

\begin{rem}
The development of the Satake-type operators \( T(z) \) and \( \bar{T}(z) \) was inspired by \cite[Proposition 12.1]{BFM}, which demonstrated that the \KZ/ connection on \(\Sing L^{\otimes (2g+2)}[0]\) with \(\kappa=2\) is isomorphic to the local system of the \(g\)-th primitive cohomology groups (tensored with \(\mathbb{C}\)) of the hyperelliptic genus \(g\) curve obtained as a double cover of \(\on{P}^1\) branched over the \(2g+2\) points (not over infinity).

Our construction of the Satake-type operators \( T(z) \) and \( \bar{T}(z) \) was motivated by applications in characteristic \( p \).

\end{rem}

In \cite{SV2}, the differential \KZ/ equations were considered modulo a   prime $p$. It turned
out that modulo $p$ the \KZ/ equations have a family of polynomial solutions. The construction
of these solutions is analogous to the construction of the hypergeometric
solutions over complex numbers, and these polynomial solutions are called the $p$-hypergeometric solutions.
\vsk.2>

The $\Sing L^{\ox(2g+1)}[2g-1]$-valued $p$-hypergeometric solutions of the \KZ/ equations with $\ka=2$
are defined as follows.
For any positive integer $\ell_1$, denote by $N^{\ell_1}(z)$ the coefficient of the monomial
$t_1^{\ell_1 p-1}$
in the vector of polynomials
\bean
 \prod_{i=a}^{2g+1}(t_1-z_a)^{(p-1)/2}
\left(\frac 1{t_1-z_1}, \dots, \frac 1{t_1-z_{2g+1}}\right)
\eean
in the variables $t_1,z_1,\dots, z_{2g+1}$.
For any positive integer $\ell_1$, the vector of polynomials $N^{\ell_1}(z)$ is a 
$\Sing L^{\ox(2g+1)}[2g-1]$-valued solution modulo $p$ of the \KZ/ equations with $\ka=2$, \cite{SV2}.
The solutions $N^{\ell_1}(z)$, $1\leq \ell_1\leq g$, are nonzero, linearly independent, and called 
the $\Sing L^{\ox(2g+1)}[2g-1]$-valued $p$-hypergeometric solutions with $\ka=2$.
For other positive integers $\ell_1$ the solutions  $N^{\ell_1}(z)$ equal zero,  \cite{V4}.

\vsk.2>

The $\Sing L^{\ox(2g+1)}[2g-1]$-valued $p$-hypergeometric solutions of the \KZ/ equations with $\ka=-2$
are defined similarly.
For any positive integer $\ell_1$, denote by $\bar N^{\ell_1}(z)$ the coefficient of
the monomial  $t_1^{\ell_1 p-1}$
in the vector of polynomials
\bean
 \prod_{i=a}^{2g+1}(t_1-z_a)^{(p+1)/2}
\left(\frac 1{t_1-z_1}, \dots, \frac 1{t_1-z_{2g+1}}\right)
\eean
in the variables $t_1,z_1,\dots, z_{2g+1}$.
For any positive integer $\ell_1$, the vector of polynomials $\bar N^{\ell_1}(z)$ is a 
$\Sing L^{\ox(2g+1)}[2g-1]$-valued solution modulo $p$ of the \KZ/ equations with $\ka=-2$,  \cite{SV2}.
The solutions $\bar N^{\ell_1}(z)$, $1\leq \ell_1\leq g$, are nonzero, linearly independent, and called 
the $\Sing L^{\ox(2g+1)}[2g-1]$-valued $p$-hypergeometric solutions  with $\ka=-2$.
For other positive integers $\ell_1$ the solutions  $\bar N^{\ell_1}(z)$ equal zero.

\vsk.2>

The $\Sing L^{\ox(2g+1)}[2g+1-2r]$-valued solutions modulo $p$ of the \KZ/ equations with $\ka=2$
are defined analogously.
For any positive integers $\ell_1, \dots,\ell_r$, one considers the coefficient
$N^{\ell_1, \dots,\ell_r}(z)$  of the monomial $t_1^{\ell_1 p-1}\dots t_r^{\ell_r p-1}$ in a suitably defined vector of polynomials
$\Psi_2(t_1,\dots,t_r, z_1,\dots$, $z_{2g+1})$. Then 
$N^{\ell_1, \dots,\ell_r}(z)$ is a $\Sing L^{\ox(2g+1)}[2g+1-2r]$-valued solution modulo $p$ of the \KZ/ equations with 
$\ka=2$.
  
\vsk.2>

The $\Sing L^{\ox(2g+1)}[2g+1-2r]$-valued solutions modulo $p$ of the \KZ/ equations with $\ka=-2$
are defined as  coefficients
$\bar N^{\ell_1, \dots,\ell_r}(z)$  of the monomials $t_1^{\ell_1 p-1}\dots t_r^{\ell_r p-1}$ in a suitably defined 
another vector of polynomials  $\Psi_{-2}(t_1,\dots,t_r, z_1,\dots$, $z_{2g+1})$. 

\vsk.2> 

The solutions \( N^{\ell_1, \dots, \ell_r}(z) \) and \( \bar{N}^{\ell_1, \dots, \ell_r}(z) \) are quite intricate. Determining which of these solutions are nonzero, understanding the linear relations among them, and establishing the dimensions of the solution spaces they generate pose significant challenges. 
Additionally, an important question is whether the \(\mathrm{KZ}\) equations admit more solutions in characteristic \( p \) than the \( p \)-hypergeometric solutions. We address these issues by applying our Satake-type operators \( T(z) \) and \( \bar{T}(z) \) in  characteristic \( p \).

\vsk.2>

We show that for all prime $p$ with finitely many exceptions, the maps 
$T(z) \vert_{ \mc P_r^{\on{KZ}}(z) } :  \mc P_r^{\on{KZ}}(z)  \to 
 \Sing L^{\ox(2g+1)}[2g+1-2r]$ and
$\bar T(z) :   \Sing L^{\ox(2g+1)}[2g+1-2r] \to \wedge^r \Sing L^{\ox(2g+1)}[2g-1]$  can be reduced modulo $p$ and
the reduced maps preserve the properties of these maps over complex numbers,  
see Sections \ref{sec 5.6}, \ref{sec 7.3} and Corollaries \ref{cor main T}, \ref{cor bar red}.
 These statements in particular show that all isomorphism-type characteristics of the \KZ/ connections with $\ka=2$ on
$\mc P_r^{\on{KZ}}(z)$  and
on $\Sing L^{\ox(2g+1)}[2g+1-2r]$ are identified by $T(z)$. 
See for example, Corollary \ref{cor main T}, Theorem \ref{thm r=2 all sol}, and Section \ref{sec 8.6}.

\vsk.3>

In Theorem \ref{thm p-map}, we show that $T(z)$ transforms wedges of the
$\Sing L^{\ox(2g+1)}[2g-1]$-valued $p$-hypergeometric 
solutions to the $\Sing L^{\ox(2g+1)}[2g+1-2r]$-valued 
$p$-hypergeometric solutions, $T(z)(N^{\ell_1}(z) \wedge \dots\wedge N^{\ell_r}(z)) = 
N^{\ell_1, \dots, \ell_r}(z)$.
Similarly, in Theorem \ref{thm p-map -2}, for $\ka=-2$, 
the map $\bar T(z)$  expresses the $r$-fold wedges 
$\bar N^{\ell_1}(z) \wedge \dots \wedge \bar N^{\ell_r}(z)$  in terms of the
$p$-hypergeometric  solutions $\bar N^{m_1,\dots, m_r}(z)$.

The statements of Theorems \ref{thm p-map} and \ref{thm p-map -2} were unexpected. 
Initially, it  was not 
clear why an isomorphism between two systems of differential equations would preserve the property of solutions being $p$-hypergeometric.

\vsk.2>

Using  Theorem 
\ref{thm p-map -2},  we 
show that  set of all  $p$-hypergeometric solutions
$\bar N^{\ell_1,\dots, \ell_r}(z)$
spans a  space of dimension $\binom{g}{r}$, see Corollary \ref{cor t vs b N}.

\vsk.2>

In Theorem \ref{thm ort r=2}, we prove  orthogonality relations
$S\big(N^{\ell_1,\dots, \ell_r}(z), \bar N^{m_1,\dots, m_r}(z)\big)=0$
 where 
$S$ is the Shapovalov form on
$\Sing L^{\ox(2g+1)}[2g+1-2r]$ and 
$\ell_1,\dots, \ell_r$ and $m_1,\dots, m_r$
are arbitrary positive integers.
 These 
relations are generalizations of the orthogonal relations
in \cite{VV1}. Cf. the orthogonal relations in \cite{MV} for simplest \qKZ/ difference equations.

\vsk.>

The \( p \)-curvature operators \( C_i := \left(\nabla^{\mathrm{KZ}, \kappa}_i\right)^p \), for \( i = 1, \dots, 2g+1 \), serve as a tool to analyze the solution space of the \(\mathrm{KZ}\) equations in characteristic \( p \). These operators act on sections of the \(\mathrm{KZ}\) connection and commute with multiplication of sections by functions on the base. In other words, the \( C_i \) are linear operators on the fibers of the \(\mathrm{KZ}\) connection. 
It is known that the intersection of their kernels,
$\bigcap_{i=1}^{2g+1} \ker C_i$,
coincides with the space generated by flat local sections of the \(\mathrm{KZ}\)-connection \(\nabla^{\mathrm{KZ}, \kappa}\), see  the Cartier descent theorem \cite[Theorem 5.1]{K1}.

\vsk.2>

Determining the \( p \)-curvature operators is quite nontrivial, as it requires computing the \( p \)-th powers of differential operators. Cf. \cite{EV}, where the spectrum of the \( p \)-curvature operators was explicitly determined.

\vsk.2>

An explicit formula for the \( p \)-curvature operators of the \(\mathrm{KZ}\) connection with fiber
\\
 \(\Sing L^{\otimes (2g+1)}[2g - 1]\) was obtained in \cite{VV1}. This formula naturally extends to a formula for the \( p \)-curvature operators of the \(\mathrm{KZ}\) connection with fiber \(\wedge^r \Sing L^{\otimes (2g+1)}[2g - 1]\).

\vsk.2>

In this paper, we analyze the case $r=2$ and $\ka=\pm 2$, and then 
using the operators $T(z)$ and 
$\bar T(z)$, demonstrate that the intersection of the kernels of the $p$-curvature operators 
on $\Sing L^{\ox(2g+1)}[2g-3]$ coincides with the space spanned by the $p$-hypergeometric 
solutions. As a consequence, we conclude that for
$\ka=\pm 2$,  all $\Sing L^{\ox(2g+1)}[2g-3]$-valued solutions in characteristic
 $p$ are $p$-hypergeometric, and the dimension of the solution space in these cases equals $\binom{g}{2}$,
 see Theorems \ref{thm r=2 all GM sol} and \ref{conj -2}.

\vsk.2>

We also analyze the $p$-curvature operators on $\Sing L^{\ox(2g+1)}[2g-5]$ for $\ka=2$. 
We observe that the intersection of
kernels of the $p$-curvature operators become larger than the space generated by the $p$-hypergeometric solutions, 
see Corollary \ref{cor a6}.
Hence the corresponding \KZ/ equations have solutions which are not $p$-hypergeometric.
These are the first examples where not all solutions in characteristic $p$ 
are $p$-hypergeometric. It would be interesting to find an effective construction of these new solutions.

\subsection*{Plan of the paper}

In Section \ref{sec 2}, we define the \KZ/ equations and  the  $\Sing L^{\ox(2g+1)}[2g+1-2r]$-valued 
hypergeometric solutions over complex numbers. In Section \ref{sec 3}, we define the \KZ/ equations with values in
$\wedge^r \Sing L^{\ox(2g+1)}[2g-1]$ and their hypergeometric solutions. We also introduce the primitive part of 
$\wedge^r H_1(C)$ and formulate Theorem \ref{thm Iso}.
Theorem \ref{thm Iso} is proved in Section \ref{sec 4}. Explicit formulas for $T(z)$ and $\bar T(z)$
 are developed in Section \ref{sec 5}.
We describe the $p$-hypergeometric solutions for
 $\ka=2$ in Section \ref{sec 6} and
  the $p$-hypergeometric solutions for $\ka=-2$ in Section \ref{sec 7}. 
  The orthogonal relations are presented in Section \ref{sec ort}.  
  The $p$-curvature operators for $r=2$ are studied in Section \ref{sec 8}. 
  The $p$-curvature operators for  $r=3$  are
  analyzed  in Section \ref{sec 9}.

In Appendix \ref{sec 10} we identify the action of the $p$-curvature operators with the action of the corresponding 
Kodaira-Spencer maps
and apply  our results on the $p$-curvature maps to  the Kodaira-Spencer maps.

\subsection*{Acknowledgments}  The authors thank G. Cotti, R. Rim\'anyi, and V.~Vologodsky for useful discussions. The third author thanks IHES for hospitality in July-August 2025.

\section{$\sll_2$ \KZ/ equations}
\label{sec 2}

\subsection{Definition of KZ equations}

Consider the  complex Lie algebra $\slt$ with generators $e,f,h$ and relations $[e,f]=h,\, [h,e]=2e, \,[h,f]=-2f$.
Consider the complex vector space $L$ with basis $w_1,w_2$ and the $\slt$-action,
\bea
e=\begin{pmatrix}
0 & 1
\\
0 & 0 & 
\end{pmatrix}, \qquad
f=\begin{pmatrix}
0 & 0
\\
1 & 0 & 
\end{pmatrix},
\qquad
h=\begin{pmatrix}
1 & 0
\\
0 & -1 & 
\end{pmatrix}.
\eea 
The $\slt$-module $L^{\ox n}$ has a basis labeled by subsets $J\subset\{1,\dots,n\}$,
\bea
w_J= w_{j_1}\ox \dots\ox w_{j_n},
\eea
where $j_i =1$ if $j_i\notin J$  and
$j_i =2$ if $j_i\in J$.

\vsk.2>

Consider the weight decomposition of $L^{\ox n}$ 
into eigenspaces of $h$,
  $L^{\ox n} = \sum_{r=0}^n L^{\ox n}[n-2r]$. 
The vectors $w_J$ with $|J|=r$ form a  basis of $L^{\ox n}[n-2r]$.
Denote by $\Ik$ the set of all $r$-element subsets of $\{1,\dots,n\}$.

Define the space of singular vectors of weight $n-2r$,
\bea
\Sing L^{\ox n}[n-2r] = \{ w\in L^{\ox n}[n-2r]\mid ew=0\}.
\eea
This space is nonempty if and only if $n\geq 2r$. We have
\bea
\dim \Sing L^{\ox n}[n-2r] = \dim L^{\ox n}[n-2r] - \dim L^{\ox n}[n-2r+2]
=\binom{n}{r} - \binom{n}{r-1}.
\eea
We shall also use the formula
\bean
\label{dim sing}
\dim \Sing L^{\ox n}[n-2r] =\binom{n-1}{r} - \binom{n-1}{r-2}.
\eean
Define the Casimir element 
$\Om =  \frac12 h\ox h + e\ox f+f\ox e\, \in\, \slt\ox\slt$. It acts on 
$L\ox L$ as $P-\frac12$ where $P$ is the permutation of factors.

\vsk.2>

Define  the linear operators on $L^{\ox n}$ depending on parameters $z=(z_1,\dots,z_n)$,
\bea
\tilde H_i(z) = \sum_{j\ne i}\frac{ \Om^{(i,j)}}{z_i-z_j}\,,
\qquad i=1,\dots,n,
\eea
where 
$ \Om^{(i,j)}:L^{\ox n}\to L^{\ox n}$ is the Casimir operator acting in the $i$-th and $j$-th tensor factors.
Denote
\bea
\der_i =\frac{\der}{\der z_i}\,,\qquad \tilde \nabla^{\on{KZ}, \ka}_i = \der_i - \frac1\ka \tilde H_i(z),
\qquad
i=1,\dots,n, \qquad \ka\in\C^\times.
\eea
The operators $\der_i - \frac1\ka \tilde H_i(z)$
define the \KZ/ connection on the trivial bundle 
over $\C^n$ with fiber $L^{\ox n}$. The connection has singularities over the diagonal hyperplanes in $\C^n$. 
The connection is flat.
For any $x\in\slt$ and $i=1,\dots,n,$    we have
\bean
\label{h inv}
[ \tilde H_i(z_1,\dots,z_n), x\ox 1\ox\dots\ox 1+\dots + 1\ox\dots\ox 1\ox x] =0.
\eean
The system of differential equations 
\bean
\label{kz tilde}
\left(\der_i - \frac1\ka \tilde H_i(z) \right)\tilde N(z_1,\dots,z_n)=0, 
\qquad  i = 1, \dots , n,
\eean
on  an $L^{\ox n}$-valued function
$\tilde N(z_1, \dots, z_n)$ is called the system of \KZ/ equations with parameter $\ka$, see \cite{KZ, EFK}.
By property \eqref{h inv},   the system of  \KZ/ equations can be considered with values in
any particular space   $\Sing L^{\ox n}[n-2r]$.

\subsection{Gauge transformation} 

If $\tilde N(z)$ satisfies the \KZ/ equations \eqref{kz tilde}, then the function $N(z)$, defined by 
\bean
\label{tilde I}
\tilde N(z) = N(z) \prod_{1\leq i<j\leq n}(z_i-z_j)^{1/2\ka}\,,
\eean
satisfies the equations
\bean
\label{KZ}
\qquad
\left(\der_i - \frac 1\ka\sum_{j\ne i}\frac{ P^{(i,j)}-1}{z_i-z_j}\right) 
N(z_1,\dots,z_n) =0, 
\qquad  i = 1, \dots , n,
\eean
which we also call the system of \KZ/ equations. Denote
\bean
\label{KZ op}
\nabla ^{\on{KZ},\ka}_i = 
\der_i - \frac 1\ka\sum_{j\ne i}\frac{ P^{(i,j)}-1}{z_i-z_j}.
\eean
We have $[ \nabla ^{\on{KZ},\ka}_i , \nabla ^{\on{KZ},\ka}_j]=0$ for all $i,j$.

\vsk.2>
Define
the nondegenerate symmetric bilinear form $S$ on $L^{\ox n}$ by the formula
\bean
\label{Sha}
S(w_I,w_J) = \delta_{I,J}\,.
\eean
The form is called the (tensor) Shapovalov form. 
The KZ connections $\nabla^{\on{KZ}, \ka}$ and $\nabla^{\on{KZ}, -\ka}$ are dual with respect to the Shapovalov form, 
\bean
\label{do}
\der_i S(x(z), y(z)) = S(\nabla^{\on{KZ}, \ka}_i x(z), y(z))+ S(x(z), \nabla^{\on{KZ}, -\ka}_i y(z))
\eean
for any $x(z), y(z)$. 

The Shapovalov form restricted to every subspace $\Sing L^{\ox n}[n-2r]$ is nondegenerate.

\subsection{Solutions of \KZ/ equations over $\C$}
  \label{Sol in C}

Denote $t=(t_1,\dots,t_r)$.
Define the   master function
\bean
\label{Master}
\Phi(t,z)
= 
\prod_{1 \leq i < j \leq r}\!\!  (t_i-t_j)^{2}
\prod_{a=1}^{n} \prod_{i=1}^{r} (t_i-z_a)^{-1}.
\eean
 Define the symmetrization and anti-symmetrization of a given differential form $f(t_1,\dots,t_k)$ by the formulas:
\bea
\Sym_{t_1,\dots,t_k}f =\sum_{\si\in S_k} f(t_{\si(1)}, \dots, t_{\si(k)}),\qquad
\on{Ant}_{t_1,\dots,t_k}f =\sum_{\si\in S_k} (-1)^{|\si|}f(t_{\si(1)}, \dots, t_{\si(k)}),\qquad
\eea
For $J  \in \Ik$ define the differential form
\bea
\nu_J(t,z)  
&=&
   \Phi(t,z;\ka)^{1/\ka}         
\Sym_t \left( \prod_{i=1}^{r}  \frac{1}{t_{i} - z_{j_i}} \right) dt_1\wedge\dots\wedge dt_r\,.
\eea
Given a flat family $ \gamma(z)$ in $\{z\} \times \C^r_t$ of
$r$-dimensional cycles of the twisted homology defined by the multivalued
function $\Phi(t,z;\ka)^{1/\ka}$, denote 
\bea
N^\ga_J(z) = \int_{\ga(z)} \nu_J\,.
\eea
Consider the  $L^{\otimes n}[n -2r]$-valued
function
\bean
\label{intrep}
N^{\gamma}(z) =  \sum_{J\in\Ik} N_J^\ga(z)\, w_J\,.
\eean

\begin{thm}
\label{thm s}
The function  $N^{\gamma} (z)$ takes values in
 $ \Sing L^{\otimes n}[n-2r]$
and satisfies the \KZ/ equations \eqref{KZ}.
\end{thm}

This theorem and  its generalizations can be found, for example,  in \cite {CF, DJMM, SV1, V2, V3}.
The solutions in Theorem \ref{thm s}  are called the  
hypergeometric solutions of the \KZ/ equations.

\begin{thm}
[{\cite[Theorem 12.5.5]{V2}}]
\label{thm alls}
If $\ka\in\C^\times$ is generic, then all solutions of the \KZ/ equations \eqref{KZ} have this form.

\end{thm}

For special values of the parameter $\ka$,  
the  hypergeometric solutions of the \KZ/ equations may span only 
a proper subspace of the space of all solutions.  For a discussion of 
the relations of this subspace of solutions and conformal blocks in conformal field theory
 see  \cite{FSV1, FSV2, BF}.

\subsection{\KZ/ equations for $\ka =2$}  

In this paper, the integer  $n$ is odd, $n=2g+1,$ and $\ka=\pm 2$. For $\ka=2$, the \KZ/ equations are
\bean
\label{KZ +2}
\qquad
\left(\der_i \, -\, \frac 1 2\,\sum_{j\ne i}\frac{ P^{(i,j)}-1}{z_i-z_j}\right) 
I(z_1,\dots,z_{2g+1}) =0, 
\qquad  i = 1, \dots , 2g+1.
\eean
Denote
\bean
\label{Peta}
\Psi(x,z) = \prod_{a=1}^{2g+1}(x-z_a), 
\qquad
\eta_{i}^{(k)}(x,z) 
=\Psi(x,z)^{-1/2}\frac{x^k dx}{x-z_i}.
\eean
Then
\bean
\label{Phi +2}
\Phi(t,z)^{1/2}
&=& 
\prod_{1 \leq i < j \leq r} \!\! (t_i-t_j) \prod_{i=1}^r \Psi(t_i,z)^{-1/2},
\eean
For $1\leq i_1,\dots,i_r\leq 2g+1$, we have
\bean
\label{nu}
\nu_{ i_1,\dots,i_r}
&=& \Phi(t,z)^{1/2}
\Sym_t \left( \prod_{j=1}^{r}  \frac{1}{t_{j} - z_{i_j}} \right) dt_1\wedge\dots\wedge dt_r\,
\\
\notag
&=&
\left(\prod_{1\leq i<j\leq r}(t_i-t_j)\right)
\on{Sym}_{t}
 \left(\prod_{j=1}^r \frac{\Psi(t_j,z)^{-1/2}}{t_j-z_{i_j}}
 \right) dt_1\wedge\dots\wedge dt_r
\\
\notag
&=&
\on{Sym}_{t}\left(
\left(\prod_{1\leq i<j\leq r}(t_i-t_j)\right)
 \frac{\Psi(t_1, z)^{-1/2}dt_1}{t_1-z_{i_1}}
 \wedge \dots \wedge \frac{\Psi(t_r, z)^{-1/2}dt_r}{t_r-z_{i_r}}\right)
\\
\notag
&=&
(-1)^{\frac{n(n-1)}2}
 \sum_{\si\in S_r} (-1)^{|\si|} 
\on{Sym}_{t}\left(
\eta^{(\si(1)-1)}_{i_1}(t_1,z) \wedge\dots\wedge
\eta^{(\si(r)-1)}_{i_r}(t_r,z) \right).
\eean
Notice that the differential form $\nu_{i_1,\dots,i_r}$ is skew-symmetric with respect to permutation of indices $i_1,\dots,i_r$.

Given a flat family $ \gamma(z)$ in $\{z\} \times \C^r_t$ of
$r$-dimensional cycles of the twisted homology defined by the multivalued
function $\Phi(t,z;\ka)^{1/2}$, define 
\bean
\label{Ng r}
N^\ga_{ i_1,\dots,i_r}(z) = \int_{\ga(z)} \nu_{ i_1,\dots,i_r}\,.
\eean
By Theorem \ref{thm s}, the  $L^{\otimes (2g+1)}[2g+1 -2r]$-valued
function
\bean
\label{int 2 r}
N^{\gamma}(z) =  \sum_{1\leq i_1<\dots<i_r\leq 2g+1} N_{ i_1,\dots,i_r}^\ga(z)\, w_{\{ i_1,\dots,i_r\}}\,
\eean
 takes values in
 $ \Sing L^{\otimes n}[2g+1-2r]$
and satisfies the \KZ/ equations \eqref{KZ} over $\C$ with $\ka=2$.

\subsection{\KZ/ equations for $\ka =2$ and $r=1$}  In that case, previous formulas take the following form.
For $i_1=1,\dots, 2g+1$, we have
\bean
\label{nu r=1}
\nu_{ i_1}(t_1)
=
\frac{\Psi(t_1, z)^{-1/2}dt_1}{t_1-z_{i_1}}
=
\eta^{(0)}_{i_1}(t_1) .
\eean
Consider the family of projective hyperelliptic curves $C(z)$ depending on $z$ and defined by the affine equation
$y^2 = \prod_{a=1}^{2g+1}(t_1-z_a).$
For distinct $z_1,\dots,z_{2g+1}$, the curve is a nonsingular curve of genus $g$.
Given a flat family $ \gamma(z)\in H_1(C(z))$, we define
\bea
N^\ga_{ i_1}(z) = \int_{\ga(z)} \nu_{ i_1}(t_1)\,.
\eea
By Theorem \ref{thm s},  the  $L^{\otimes n}[2g-1]$-valued
function
\bean
\label{int 2 r=1}
N^{\gamma}(z) =  \sum_{i_1=1}^{2g+1} N_{ i_1}^\ga(z)\, w_{\{ i_1\}}\,
\eean
 takes values in
 $ \Sing L^{\otimes (2g+1)}[2g-1]$
and satisfies the \KZ/ equations \eqref{KZ} with $\ka=2$. 

\vsk.2>
Recall that 
$\Sing L^{\ox(2g+1)}[2g-1]$ consists of the vectors
$\sum_{i_1=1}^{2g+1} c_{i_1} w_{\{ i_1\}} $ with  $ \sum_{i_1=1}^{2g+1} c_{i_1}=0$.

\begin{thm} [{\cite[Formula (1.3)]{V1}}]

\label{thm r=1 all sol}
All solutions of the \KZ/ equations with $\ka=2$ and values in $\Sing L^{\ox(2g+1)}[2g-1]$ have this form. Namely, the complex vector space of solutions of the form $N^{\gamma}(z)$ is $2g$-dimensional.

\end{thm}

This theorem follows from the determinant formula for multidimensional hypergeometric
integrals in \cite{V1}, in particular, from \cite[Formula (1.3)]{V1}. 

The $\Sing L^{\ox(2g+1)}[2g-1]$-valued \KZ/ equations with $\ka=2$  were analyzed in \cite{VV3} over $p$-adic fields.

\subsection{\KZ/ equations for $\ka =-2$}  

For $\ka=-2$, the \KZ/ equations are
\bean
\label{KZ -2}
\qquad
\left(\der_i \,+\, \frac 1 2\,\sum_{j\ne i}\frac{ P^{(i,j)}-1}{z_i-z_j}\right) 
I(z_1,\dots,z_{2g+1}) =0, 
\qquad  i = 1, \dots , 2g+1.
\eean
Then
\bean
\label{Phi -2}
\Phi(t,z)^{-1/2}
&=& 
\prod_{1 \leq i < j \leq r} \!\! (t_i-t_j) ^{-1}\prod_{i=1}^r \Psi(t_i,z)^{1/2},
\eean
For $1\leq i_1,\dots,i_r\leq 2g+1$, we have
\bean
\label{nu -}
\bar \nu_{ i_1,\dots,i_r}
&=& \Phi(t,z)^{-1/2}
\Sym_t \left( \prod_{j=1}^{r}  \frac{1}{t_{j} - z_{i_j}} \right) dt_1\wedge\dots\wedge dt_r\,
\eean
Given a flat family $ \gamma(z)$ in $\{z\} \times \C^r_t$ of
$r$-dimensional cycles of the twisted homology defined by the multivalued
function $\Phi(t,z;\ka)^{-1/2}$, define 
\bean
\label{Ng r}
\bar N^\ga_{ i_1,\dots,i_r}(z) = \int_{\ga(z)} \bar\nu_{ i_1,\dots,i_r}\,.
\eean
By Theorem \ref{thm s}, the  $L^{\otimes n}[2g+1 -2r]$-valued
function
\bean
\label{int -2 r}
\bar N^{\gamma}(z) =  \sum_{1\leq i_1<\dots<i_r\leq 2g+1} \bar N_{ i_1,\dots,i_r}^\ga(z)\, w_{\{ i_1,\dots,i_r\}}\,
\eean
 takes values in
 $ \Sing L^{\otimes n}[2g+1-2r]$
and satisfies the \KZ/ equations \eqref{KZ} over $\C$ with $\ka=-2$.

\subsection{\KZ/ equations for $\ka =-2$ and $r=1$}  In that case, 
for $i_1=1,\dots, 2g+1$, we have
\bean
\label{nu r=1}
\bar \nu_{ i_1}(t_1)
=
\frac{\Psi(t_1, z)^{1/2}dt_1}{t_1-z_{i_1}}\,.
\eean
Recall our family of projective hyperelliptic curves $C(z)$ depending on $z$ and defined by the affine equation
$y^2 = \prod_{a=1}^{2g+1}(t_1-z_a)$.
Given a flat family $ \gamma(z)\in H_1(C(z))$, we define
\bea
\bar N^\ga_{ i_1}(z) = \int_{\ga(z)} \bar\nu_{ i_1}(t_1)\,.
\eea
By Theorem \ref{thm s},  the  $L^{\otimes n}[2g-1]$-valued
function
\bean
\label{int 2 r=1}
\bar N^{\gamma}(z) =  \sum_{i_1=1}^{2g+1} \bar N_{ i_1}^\ga(z)\, w_{\{ i_1\}}\,
\eean
 takes values in
 $ \Sing L^{\otimes (2g+1)}[2g-1]$
and satisfies the \KZ/ equations \eqref{KZ} with $\ka=-2$. 

\begin{thm} [{\cite[Formula (1.3)]{V1}}]
\label{thm r=1 all sol -2}
All solutions of the \KZ/ equations with $\ka=-2$ and values in $\Sing L^{\ox(2g+1)}[2g-1]$ have this form. Namely, the complex vector space of solutions of the form $\bar N^{\gamma}(z)$ is $2g$-dimensional.

\end{thm}

\section{Wedge-power construction}
\label{sec 3}

\subsection{Wedge-power of \KZ/ equations on $\Sing L^{\ox(2g+1)}[2g-1]$ with $\ka=2$}
\label{sec 3.1}

The \KZ/ equations on $L^{\ox(2g+1)}[2g-1]$ with $\ka=2$ induce a system of differential equations
with values in $\wedge^r L^{\ox(2g+1)}[2g-1]$ with the invariant subspace
\bean
\label{subs}
\wedge^r \Sing L^{\ox(2g+1)}[2g-1] \subset \wedge^r L^{\ox(2g+1)}[2g-1].
\eean
The integral representations for solutions look as follows.

\vsk.2>

For $1\leq i_1,\dots ,i_r \leq 2g+1$, define
\bean
\label{mu}
\mu_{i_1,\dots,i_r} 
&=&
\on{Ant}_{t_1,\dots,t_r}
 \left(\prod_{j=1}^r \frac{\Psi(t_j,z)^{-1/2}}{t_j-z_{i_j}}
 \right) dt_1\wedge\dots\wedge dt_r\,,
\\
\notag
&=&
\on{Sym}_{t_1,\dots,t_r}
 \left( \eta^{(0)}_{i_1}(t_1)\wedge\dots\wedge  \eta^{(0)}_{i_r}(t_r)  \right).
 \eean
For any flat families $\ga_1(z),\dots,\ga_r(z) \in H_1(C(z))$, denote
\bea
M_{i_1,\dots,i_r}^{\ga_1\times\dots\times\ga_r} (z)
=
 \int_{\ga_1(z)\times\dots\times\ga_r(z)} \mu_{i_1,\dots,i_r}\,.
\eea

Notice that $\mu_{i_1,\dots,i_r}$ 
is skew-symmetric with respect to permutations of
the indices $i_1,\dots,i_r$.
Hence $M_{i_1,\dots,i_r}^{\ga_1\times\dots\times\ga_r}$ 
is skew-symmetric with respect to permutations of
the indices $i_1,\dots,i_r$.

Define a basis in $\wedge^r L^{\ox(2g+1)}[2g-1]$ by the formulas
\bea
w^{\{i_1,\dots,i_r\}} =w_{\{i_1\}}\wedge \dots\wedge w_{\{i_r\}}, \qquad 1\leq i_1<\dots<i_r\leq 2g+1.
\eea
For any flat families $\ga_1(z),\dots,\ga_r(z)\in H_1(C(z))$, the vector 
\bean
\label{sM}
M^{\ga_1\times\dots\times\ga_r}(z) =    \sum_{1\leq i_1<\dots <i_r\leq 2g+1} M^{\ga_1\times\dots\times\ga_r}_{i_1,\dots,i_r} (z) w^{\{i_1,\dots,i_r\}}
\eean
is a solution of the KZ equations over the field $\C$ with $\ka=2$ and values  in
\\
 $\wedge^r \Sing L^{\ox(2g+1)}[2g-1]$.
 By Theorem \ref{thm r=1 all sol},  all solutions over the field $\C$ with values in 
 \\
$\wedge^r \Sing L^{\ox(2g+1)}[2g-1]$ have this form. Thus we obtain an isomorphism
\bean
\label{mc I}
\mc I : \wedge^rH_1(C) \to \wedge^r \Sing L^{\ox(2g+1)}[2g-1],
\quad
\ga_1(z)\wedge \dots\wedge \ga_r(z) \mapsto M^{\ga_1\times\dots\times\ga_r}(z)
\eean
of the  Gauss-Manin connection on $\wedge^rH_1(C)$ and the \KZ/ connection with $\ka=2$ on
\\
$\wedge^r \Sing L^{\ox(2g+1)}[2g-1]$.

\subsection{Curves and their primitive cohomology}
\label{curves}

For a smooth connected genus $g$ curve $C$ over the field of complex numbers, 
we may identify the homology and cohomology of the 
curve. This identification arises by  the Poincare pairing and is given by 
an element $\delta\in \wedge^2 H^1(C)$ and also by an element
$\Delta\in \wedge^2H_1(C)$ which we can think of as symplectic forms. 

For any integer $r$ with
$0 \leq r \leq g$, we have the primitive cohomology ${\mc P}^r(C)$
of $\wedge^r {H}^1(C)$ given as the kernel of the map
$\wedge^r {H}^1(C) \to \wedge^{2g-r+2}{H}^1(C)$
induced by ``wedging with $\delta$'' $g-r+1$ times. 
 We may also
view ${\mc P}^r(C)$ as the quotient of $\wedge^r {H}^1(C)$ by
$\on{Image} \big(\delta \wedge\,:\, \wedge^{r-2} {H}^1(C)$ 
$\to \wedge^{r} {H}^1(C)\big)$.
The rank of
${\mc P}^r(C)$ is $ \binom{2g}{r} - \binom{2g}{r-2}$.

\vsk.2>

These constructions can also be carried out in homology groups and we get primitive parts  ${\mc P}_r(C)$ in 
$\wedge^r H_1(C)$. When the curve $C$ varies in a family, the homology and cohomology groups
carry the Gauss-Manin connection, and the Poincare elements are flat. Therefore we get local systems
of primitive homology and cohomology groups $\mc{P}_r(C)$ and $\mc{P}^r(C)$ for $0\leq r\leq g$.

For the family of curves $C(z)$ considered before, the local systems of primitive cohomology and homology groups 
are irreducible, and pairwise distinct. We will review this setting in detail in Section \ref{sec 4}.

It is known that we get a direct decomposition,
\bean
\label{dec}
\wedge^r H_1(C(z)) 
&=&
 \mc P_r(C(z)) \oplus \left(\Delta(z)\wedge \mc P_{r-2}(C(z))\right)\oplus
\\
\notag
&\oplus& 
\Big(\Delta(z)\wedge \Delta(z)
\wedge
\mc P_{r-4}(C(z))\Big)\oplus \dots.
\eean

\subsection{Primitive part and KZ equations}

\label{sec 3.3}

Notice that  $\dim \mc P^r = \dim \Sing L^{\ox(2g+1)}[2g+1-2r]$. This observation extends to the following theorem.

\begin{thm}
\label{thm Iso}
For $0\leq r\leq g$, there is a nonzero homomorphism
\bean
\label{T iso}
\mc T : \wedge^r H_1(C) \to \Sing L^{\ox(2g+1)}[2g+1-2r]
\eean
of the Gauss-Manin connection  on $\wedge^r H_1(C)$ to  the \KZ/  connection on 
$\Sing L^{\ox(2g+1)}[2g+1-2r] $ with $\ka=2$.  The nonzero homomorphism is unique up to a
 multiplicative constant.

The map $\mc T$ restricted to 
$\mc P_r(C(z))$ is an isomorphism, and restricted to other summands in \eqref{dec} equals zero.

The map $\mc T$ is given by the formula,
\bean
\label{sM new}
\mc T \ :\  \ga_1(z)\wedge\dots\wedge\ga_r(z) \to  \sum_{1\leq i_1<\dots <i_r\leq 2g+1} 
 N^{\ga_1\times\dots\times\ga_r}_{i_1,\dots,i_r} (z) \,w_{\{i_1,\dots,i_r\}}
\eean
where $\ga_1(z),\dots,\ga_r(z)\in H_1(C(z))$ are flat families, and
$N^{\ga_1\times\dots\times\ga_r}_{ i_1,\dots,i_r}(z) =
 \int_{\ga_1(z)\times\dots\times\ga_r(z) } \nu_{ i_1,\dots,i_r}$.
\end{thm}

Theorem \ref{thm Iso} is proved in Section \ref{sec 4.6}.

\vsk.2>

A  polarized variation of Hodge structures is an integrally defined  local system with the extra data of a Hodge filtration (tensored with 
$\C$)  and has a non-degenerate bilinear pairing called the polarization (plus Griffiths transversality and other data). These local systems tensored with $\C$ have a Hodge decomposition. The indices $p,q$ in $H^{p,q}$ that appear here add up to the weight of the Hodge structure.

The local system of primitive cohomology  groups $\mc P_r(C(z))$ with $\C$ coefficients is the local system of primitive cohomology groups with integer coefficients (which is known to be a polarized variation of the Hodge structures) tensored with $\C$. Therefore, Theorem \ref{thm Iso} implies that the \KZ/ local system, which is over the  complex numbers, is isomorphic to a polarized variation of pure Hodge structures of weight $r$ tensored with complex numbers:
\begin{cor}
The monodromy of the \KZ/ equations on $\Sing L^{\ox(2g+1)}[2g+1-2r] $ with $\ka=2$ is irreducible and isomorphic to a polarized variation of pure Hodge structures of weight $r$ tensored with $\C$. 
\end{cor}

\begin{rem}

It was shown in \cite[Proposition 12.1]{BFM} that the \KZ/ connection on \(\Sing L^{\otimes (2g+2)}[0]\) with \(\kappa=2\) is isomorphic to the local system of the \(g\)-th primitive cohomology groups  of the hyperelliptic genus \(g\) curve obtained as a double cover of \(\mathbb{P}^1\) branched over the \(2g+2\) points (not over infinity).  Note that there is a Tate twist in \cite{BFM} but that can be ignored for 
the statement on local systems.  Some farther expectations were formulated in   \cite[Section 12.2.2]{BFM} when not all 
tensor factors are $L$, but still only singular vectors of weight zero were considered, see also  \cite[Remark 12.14 ]{BFM}.

\end{rem}

\subsection{Wedge-power of \KZ/ equations on $\Sing L^{\ox(2g+1)}[2g-1]$ with $\ka=-2$}
\label{sec 3.4}

The \KZ/ equations on $L^{\ox(2g+1)}[2g-1]$ with $\ka=-2$ induce a system of differential equations
with values in $\wedge^r L^{\ox(2g+1)}[2g-1]$ with the invariant subspace
\bean
\label{subs bar}
\wedge^r \Sing L^{\ox(2g+1)}[2g-1] \subset \wedge^r L^{\ox(2g+1)}[2g-1].
\eean
The integral representations for solutions look as follows.  
For $1\leq i_1,\dots ,i_r \leq 2g+1$, define
\bean
\label{mu bar}
\bar \mu_{i_1,\dots,i_r} 
&=&
\on{Ant}_{t_1,\dots,t_r}
\!
 \left(\prod_{j=1}^r \frac{\Psi(t_j,z)^{1/2}}{t_j-z_{i_j}}
 \right) dt_1\wedge\dots\wedge dt_r\,.
 \eean
For any flat families $\ga_1(z),\dots,\ga_r(z) \in H_1(C(z))$, denote
\bea
\bar M_{i_1,\dots,i_r}^{\ga_1\times\dots\times\ga_r} (z)
=
 \int_{\ga_1(z)\times\dots\times\ga_r(z)} \bar \mu_{i_1,\dots,i_r}\,.
\eea

Notice that $\bar \mu_{i_1,\dots,i_r}$ 
is skew-symmetric with respect to permutations of
the indices $i_1,\dots,i_r$.
Hence $\bar M_{i_1,\dots,i_r}^{\ga_1\times\dots\times\ga_r}$ 
is skew-symmetric with respect to permutations of
the indices $i_1,\dots,i_r$.

For any flat families $\ga_1(z),\dots,\ga_r(z)\in H_1(C(z))$, the vector 
\bean
\label{sM bar}
\bar M^{\ga_1\times\dots\times\ga_r}(z) =    \sum_{1\leq i_1<\dots <i_r\leq 2g+1} 
\bar M^{\ga_1\times\dots\times\ga_r}_{i_1,\dots,i_r} (z) \,w^{\{i_1,\dots,i_r\}}
\eean
is a solution of the KZ equations over the field $\C$ with values  in
 $\wedge^r \Sing L^{\ox(2g+1)}[2g-1]$ and $\ka=-2$.
 By Theorem \ref{thm r=1 all sol -2},  all solutions over the field $\C$ with values in 
$\wedge^r \Sing L^{\ox(2g+1)}[2g-1]$ have this form.
Thus we obtain an isomorphism
\bean
\label{mc bar I}
\phantom{aaa}
\bar{\mc I} : \wedge^rH_1(C) \to \wedge^r \Sing L^{\ox(2g+1)}[2g-1],
\quad
\ga_1(z)\wedge \dots\wedge \ga_r(z) \mapsto \bar M^{\ga_1\times\dots\times\ga_r}(z)
\eean
of the  Gauss-Manin connection on $\wedge^rH_1(C)$ and the \KZ/ connection with $\ka=-2$ on
\\
$\wedge^r \Sing L^{\ox(2g+1)}[2g-1]$.

\subsection{Shapovalov form}
\label{sec 3.5}

For any $2\leq r \leq g$, the Shapovalov form on $\Sing L^{\ox(2g+1)}[2g-1]$ induces a nondegenerate 
symmetric bilinear form $S^{\wedge^r}$ on
$\wedge^r\Sing L^{\ox(2g+1)}[2g-1]$, also called the Shapovalov form.
The \KZ/ connections on $\wedge^r\Sing L^{\ox(2g+1)}[2g-1]$ with $\ka=2$ 
and $-2$ are dual with respect to $S^{\wedge^r}$.

\section{Proof of Theorem \ref{thm Iso}}
\label{sec 4}

\subsection{Hyperelliptic curves}

Let $\mc{C}_{2g+1}$ be the configuration space of $2g+1$ distinct points on $\Bbb{A}^1_{\Bbb{C}}$.
Recall that $n=2g+1$ and we have a family of hyperelliptic curves $C(z)$ over $\mc{C}_{2g+1}$. 
There are corresponding local systems of homology and cohomology groups over $\mc{C}_{2g+1}$.

We introduce some formal notation to facilitate some of the arguments in this section. Let $\pi:C \to \mc{C}_{2g+1}$ be the family of smooth hyperelliptic
curves given by the normalisation of the closure of
$y^2=\prod_{i=1}^{2g+1}(z_i-t) \subset \mc{C}_{2g+1} \times \Bbb{A}^2$ in
$\mc{C}_{2g+1} \times \Bbb{P}^2$.   The fiber of $\pi$ over $(z_1,\dots,z_{2g+1})$ is the curve
$C(z)$ defined earlier.

We denote by $\mc{H}^1(C)$ the local system of rank $2g$ on
$S$ given by $H^1(\pi_*\Q_C)$. We have a canonical injection
$\delta:\Q_S \to \wedge^2 {\mc H}^1(C)$ given by the Poincar\'e pairing. For any integer $r$ with
$0 \leq r \leq g$, we have a canonical sub-local system $\mc{P}^r(C)$
of $\wedge^r (\mc{H}^1(C))$ given as the kernel of the map
$\wedge^r \mc{H}^1(C) \to \wedge^{2g-r+2}\mc{H}^1(C)$
induced by ``wedging with $\delta$'' $g-r+1$ times. Here $\delta$ is the 
Poincar\'e form. Note that $\mc{P}^r(C)$ is a variation of pure Hodge structures and 
its $H^{r,0}$ equals $\wedge^r \mc{H}^{1,0}(C)$ and is hence nonzero. Fiber by fiber this is local system of 
primitive cohomology groups considered in Section \ref{curves}.

\begin{remark}
For a smooth projective connected curve $X$ over the complex numbers, let $J(X)= H^0(X,\Omega^1)^*/H_1(X,\Bbb{Z})$ denote the  
Jacobian variety of $X$. Recall also that $J(X)$ is the $\operatorname{Pic}^0(X)$ of the curve $X$ (group of divisors of degree $0$) by the Abel-Jacobi theorem, and thus can be defined in any characteristic. Topologically $J(X)$ is a torus $(S^1)^{2g}$, and therefore $H^r(J(X))=\wedge^r H^1(J(X))$, further the Abel-Jacobi map $X\to J(X)=\operatorname{Pic}^0(X)$ given by $p\mapsto [p]-[p_0]$ gives an isomorphism $H^1(J(X))= H^1(X)$. Therefore the $r$-th cohomology of $J(X)$ equals $\wedge^r H^1(X)$.

Let $\pi': J(C)\to \mc{C}_{2g+1}$ be family of Jacobian varieties of $C$ corresponding to the family of curves $\pi:C\to\mc{C}_{2g+1}$. The cohomology
groups $H^r(\pi'_*\Bbb{Q}_{J(C)})$ are identified with $\wedge^r \mc{H}^1(C)$.
It is known that $J(C)$ is a family of principally polarized abelian varieties, with 
polarization given by  $\delta\in  \wedge^2 \mc{H}^1(C)= H^2(\pi'_*\Bbb{Q}_{J(C)})$.
The local system $\mc{P}^r(C)$ is identified with the  local system of $r$-th primitive 
cohomology groups of the family of polarized varieties $J(C)$.
\end{remark}

The monodromy group of the topological local
system $\mc{H}^1(C)$ on $\mc{C}_{2g+1}(\C)$ is Zariski dense in
the group $\mr{Sp}_{2g}$ (see, e.g., \cite[Theorem 1 (2)]{ACN}). Since $\mc{P}^r(C)$ corresponds to the $r$-th
fundamental representation of $\mr{Sp}_{2g}$ (see, e.g., \cite[Theorem
17.5]{fulton-harris}) the monodromy representation of
$\mc{P}^r(C)$ is absolutely irreducible.  We also obtain that the local systems
$\mc{P}^r(C)$ for $r=0,\dots,g$ are pairwise distinct irreducible monodromy representations.

Let $X_r = C\times_SC\times_S \dots \times_SC$, where there are $r$
factors and let $\pi_r:X_r \to S$ be the structure map. By the Kunneth
formula there is a natural injection
$\otimes^r \mc{H}^1(C) \to H^r(\pi_{r,*}\Q_{X_r})$ and the image of
$\wedge^r \mc{H}^1(C)$ lies   in the space of $\Sigma_r$-invariants
$H^r(\pi_{r,*}\Q_{X_r})^{\Sigma_r}$. Consequently, we have an
injection $\mc{P}^r(C) \to H^r(\pi_{r,*}\Q_{X_r})^{\Sigma_r}$.

\subsection{KZ equations}

Let $\mf{g}=\mf{sl}_2$ and $\kappa=2$. Let  $\vec{\vpi}^{2g+1} := (\vpi,\dots,\vpi)$ with $\vpi$ appearing
$2g+1$ times, where $\vpi:= \vpi_1$ is the fundamental dominant weight of $\mf{sl}_2$.
 Let $r\in\Bbb{Z}$ with $1\leq r\leq g$. 
 
 Let $\mc{KZ}_2$  be the local system on $\mc{C}_{2g+1}$ corresponding to the  
 dual of the KZ connection on $\Sing L^{\ox(2g+1)}[2g+1-2r]$  
 with $\ka =2$. It has the same rank as that of $\mc{P}^r(C)$, see \eqref{dim sing}.

\begin{proposition}\label{primitiveKZ}
  The  local system $\mc{KZ}_{2}$ on $\mc{C}_{2g+1}$ is isomorphic to
  the local system $\mc{P}^r(C)$, up to tensoring with a rank one local system on  $\mc{C}_{2g+1}$.
\end{proposition}

\subsection{Proof of Proposition \ref{primitiveKZ}} 
\label{sec proof}

 The proof is an adaptation of the methods and proof in Section 12 of \cite{BFM} where  a family of hyperelliptic curves which are unramified over $\infty\in\Bbb{P}^1$ (and only a weight space of $0$ is considered at infinity) is considered in the basic geometric setting. There are some additional ingredients which we will point out at appropriate places.

Let
$\widehat{U}\to U\subseteq \mc{C}_{2g+1} \times
\Bbb{A}_{\msk}^{r}$ be the covering given by the tuples
$({z}, t_1,\dots,t_{r}, u)$ such that $u^2=\Phi(t,z)$ (see \eqref{Master}). This
covering has Galois group $\Sigma_{r}\times \mu_2$.

Let $C^0\subset C$ be the open subset of points on the affine part
(i.e., $t\neq \infty$) such that $t\neq z_i$ for all $i$ and, for any
$r>0$, let $X^0_r = C^0 \times_SC^0\times_S\dots \times_SC^0$ with $r$
factors and let $\Delta_r$ be the union of the diagonals $t_i = t_j$,
$1 \leq i < j \leq r$ in $X^0_r$.

There is an unramified Galois
covering $X^0_{r}\setminus \Delta_{r}\to\widehat{U}$ given by
$$\prod_{i=1}^{r} (y_i,t_i)\mapsto (t_1,\dots,t_{r},y_1\dots
y_{r}\prod_{1 \leq b < c \leq r} (t_b - t_c)^{-1}).$$ The
composition $X_0^{r} \setminus \Delta_{r}\to{U}$ is an unramified Galois
covering with Galois group $\Sigma_{r}\times \mu_2^{r}$. For any $r>0$,
let $\pi_{r}^0: X^0_{r} \setminus \Delta_r \to \mc{C}_{2g+1}$ be the structure map.
Keeping track of isotypical components, we get a map (not a priori an
injection)

\begin{equation}\label{e:mapfromkz2}
  \mc{KZ}_{2}\to  H^{r}(\pi^0_{r,*}\Q_{X^0_{r}\setminus\Delta_{r}})^{\tau}
\end{equation}
where $\tau: \Sigma_{r} \times \mu_2^{r}\to\mu_2$ is trivial on
$\Sigma_{r}$ and the product on the second (the triviality on the
$\Sigma_{r}$ factor is because of the factor $\prod (t_i-t_j)$ in $\Phi(t,z)$). Let $I$ be the image
of this map.

We have also a map of local systems
$\mc{P}^r(C)  \to
H^{r}(\pi^0_{r,*}\Q_{X^0_{r}-\Delta})^{\tau'} $ 
and we let its image be $J$; the invariance under
$\tau'$ uses that the hyperelliptic involution acts by multiplication
by $-1$ on $H^1(\pi_*\Q_C)$.
\begin{claimn}\label{Claime}
  The local system $I \cap J$ is nonzero.
\end{claimn}

We first show that the claim implies the proposition and then prove
the claim. Since $J$ is nonzero by the claim and the monodromy
representation of $\mc{P}^r(C)$ is irreducible, we deduce that
$J \simeq \mc{P}^r(C)$ and then the claim further implies that
$I \cap J = J$. The rank equality of 
$\mc{KZ}_{2}$  and $ \mc{P}^r(C)$
implies that $I = J$ and the map \eqref{e:mapfromkz2} is
injective. The proposition follows.

\subsection{Proof of Claim \ref{Claime}}
So far modulo the rank verification the argument was entirely following the lines of Section 12 of \cite{BFM}. At this point instead of a residue computation, we follow a different method still involving global holomorphic forms of the top dimension.

Consider the space of top degree holomorphic differentials $$H^{r,0}(\pi_{r,*}\Q_{X_{r}})^{\tau'}= H^{r,0}(\pi_{r,*}\Q_{X_{r}})^{\Sigma_r}=H^{r,0}(\mc{P}^r(C))$$ where $X_r$ is as before. By results of Deligne \cite{hodge2}, $H^{r,0}(\pi_{r,*}\Q_{X_{r}})^{\tau'}$ maps injectively into $H^{g}(\pi^0_{r,*}\Q_{X^0_{g}-\Delta})^{\tau'}$. Since $r\leq g$, $H^{r,0}(\mc{P}^r(C))$ is nonzero.

To prove the claim it suffices  to show that $H^{r,0}(\pi_{r,*}\Q_{X_{r}})^{\tau'}\subseteq H^{g}(\pi^0_{r,*}\Q_{X^0_{g}-\Delta})^{\tau'}$ is in the image 
$I$ of the map \eqref{e:mapfromkz2}.  We can prove this statement fiberwise on $\mc{C}_{2g+1}$.

This follows from \cite{BF} diagram (4.5): To make the transition to the notation in that paper, we first note that there is a $\Sigma_r$ equivariant rank one local system $\mathcal{L}$ on $U$ such that there is a map \cite{SV1} (see Remark \ref{notation} for further explanation of the notation):
\begin{equation}\label{kzmap}
\mc{KZ}_{2}\to H^{r}(U,\ml)^{\chi}
\end{equation}
where $\chi$ is the sign character. Here we note that $H^{r}(U,\ml)^{\chi}$ is the same as the target $H^{r}(\pi^0_{r,*}\Q_{X^0_{r}\setminus\Delta_{r}})^{\tau}$ of the map \ref{e:mapfromkz2}, i.e., $H^{r}(\pi^0_{r,*}\Q_{X^0_{r}\setminus\Delta_{r}})^{\tau}$. We want the image of 
\eqref{kzmap} to contain $H^{r,0}(\pi_{r,*}\Q_{X_{r}})^{\tau'}$. We need the following consequence of \cite{L2} (as noted implicitly in \cite{BF}) which is valid in a more general context:
\begin{lemma}\label{compact}
The image of \eqref{kzmap} contains the image of the compactly supported cohomology $H^{r}_c(U,\ml)^{\chi}\to H^{r}(U,\ml)^{\chi}$.
\end{lemma}
Given this, the desired equality now follows from the fact that global top degree holomorphic forms on a smooth proper variety lift to compactly supported cohomology classes on Zariski open subsets. Hence the image of $H^{r,0}(\pi_{r,*}\Q_{X_{r}})^{\tau'}$ in $H^{r}(\pi^0_{r,*}\Q_{X^0_{r}\setminus\Delta_{r}})^{\tau}$ lifts to compactly supported cohomology
of the fibers which equals $H^{r}_c(U,\ml)^{\chi}$. This tells us that  $H^{r,0}(\pi_{r,*}\Q_{X_{r}})^{\tau'}\subseteq J\subseteq H^{r}(U,\ml)^{\chi}$ is in the image of $H^{r}_c(U,\ml)^{\chi}\to H^{r}(U,\ml)^{\chi}$ (identified with $I$) as desired.

\subsection{Proof of  Lemma \ref{compact}}
By Looijenga's refinement \cite{L2} of the results of Schechtman-Varchenko \cite{SV2} (see Proposition 4.7 in \cite{BF}), the 
\KZ/ space $\mc{KZ}_{2}$ is identified with the image of a map
\begin{equation}\label{Looijenga}
H^M(V', q_!\mathcal{L})^{\chi}\to H^M(V, j_!\mathcal{L})^{\chi}.
\end{equation}
\begin{remark}\label{notation}
We explain some of the notation used above: Let $\mathcal{L}$ be the rank one $\Sigma_r$ equivariant  local system local system on $U$ given by the connection $d+d \log \Phi(t,z)$ on the trivial bundle, and $\chi:\Sigma_r\to\mu_2$ the sign character.

Let $P$ be a smooth projective compactification of $U$ such that the complement $P-U=\cup_{\alpha} E_{\alpha}$ is a divisor with simple normal crossings. Let $a_{\alpha}\in \Bbb{Q}$ be the residue of $d \log \Phi(t,z)$ over $E_{\alpha}$.

Let $V'=P-\cup' E_{\alpha}$ where  the union is over $\alpha$  such that $a_{\alpha}$ is not in $\Bbb{Z}$ or is an integer $>0$, and let $q:U\to V'$ be the inclusion. Similarly $V\subseteq V'$ is the open subset of $P$ given by $V=P-\cup' E_{\alpha}$ where the union is over $\alpha$ such that $a_{\alpha}$ is not in $\Bbb{Z}$ or is an integer $\geq 0$. Let $j:U\to V$ be the inclusion.  

\end{remark}

This image has a subquotient given by the image of $H^{r}_c(U,\ml)^{\chi}\to H^{r}(U,\ml)^{\chi}$ (see the bottom square in equation (4.7) of \cite{BF}),
\begin{equation}
\label{diagrammo}
\xymatrix{
 H^{r}_c(U,\ml)^{\chi}\ar[r]\ar[d] & H^M(V', q_!\mathcal{L}(a))^{\chi}\ar[d] 
 \\
 H^{r}(U,\ml)^{\chi} & H^M(V, j_!\mathcal{L}(a)^{\chi}\ar[l] 
 \\}
\end{equation}
The KZ space is identified with the image of the right vertical arrow. The image of this space  in $H^{r}(U,\ml)^{\chi}$ is the image of \eqref{kzmap}. It contains the  image of the left vertical arrow because of the above diagram.

\begin{remark} Proposition \ref{primitiveKZ} is valid also for $\kappa=-2$. One can see this by the same argument as above with $\Phi(t,z)$ replaced by its inverse. We could also appeal to duality because the KZ system for $\kappa=-2$ is dual to the system for $\kappa=2$, and the primitive cohomology groups are self dual since $\mc H^1(C)$ is self dual. 
\end{remark} 

\subsection{Proof of Theorem \ref{thm Iso}} 
\label{sec 4.6}

The local system $\Sing L^{\ox(2g+1)}[2g+1-2r]$ with $\ka=2$ is dual to  $\mc{KZ}_2$ in Proposition \ref{primitiveKZ}. Hence Proposition \ref{primitiveKZ} implies that the local system $\Sing L^{\ox(2g+1)}[2g+1-2r]$ 
 with $\ka=2$  is isomorphic to  the  dual of $\mc{P}^r(C)$ which
is $\mc{P}^r(C)$ by the self-duality of representations of $\mr{Sp}_{2g}$. The local system of homology groups of $C$ is isomorphic to the local system of the first cohomology groups by Poincar\'e duality.

Therefore to prove the first part of Theorem \ref{thm Iso}, we are reduced to showing that there is a unique up to scalars map
$$
\wedge^r \mc{H}^1(C) \to \mc{P}^r(C).
$$
The desired statement  now follows by Schur's lemma since $\wedge^r \mc{H}^1(C)$ is a direct sum corresponding to the decomposition of the $r$-th wedge product of the vector representation of the symplectic group:
$$
\wedge^r \mc{H}^1(C)=\mc{P}^r(C)\oplus \mc{P}^{r-2}(C)\oplus\dots\oplus{P}^{r_0}(C)
$$
where $r_0=0$ if $r$ is even and $1$ if $r$ is odd,
and the $\mc{P}^{i}(C), i=0,\dots,g,$ are pairwise distinct irreducible representations of the monodromy group.

We finally prove the explicit formula for $\mc{T}$. When  $\kappa=2$, the image of \eqref{Looijenga} maps to $H^M(V, \mathcal{L})^{\chi}$ which is equal to 
 $H^{r}(\pi^0_{r,*}\Q_{X^0_{r}})^{\tau}$ (compare with \eqref{e:mapfromkz2}) so that we may ignore the singularities on the diagonals. These maps are also consistent with connections. Now $X^0_{r}$ is a product of affine curves, so that their second cohomology groups are zero. Therefore $H^{r}(\pi^0_{r,*}\Q_{X^0_{r}})^{\tau}$ is fiber by fiber the tensor product of the first cohomologies of the curves.
 
 This implies that flat families of elements in $\wedge^r H_1(C^0)$ where $C^0=C-\{z_1,\dots,z_n,\infty\}$ give rise to solutions (possibly zero) to the KZ system on $\Sing L^{\ox(2g+1)}[2g+1-2r]$ for $\kappa=2$ by well-defined integrals over cycles. Because of Proposition \ref{primitiveKZ} (using that $I$ is contained in the image of   $\wedge^r \mc{H}^1(C)$, these solutions only depend upon the image of $\wedge^r H_1(C^0)$ in $\wedge^r H_1(C)$, and generate all solutions to the KZ system on $\Sing L^{\ox(2g+1)}[2g+1-2r]$. This gives the surjective map in Theorem \ref{thm Iso}. 
 
 Since the other summands $\mc P_i(C(z))$ for $i\neq r$ in \eqref{dec} are irreducible representations of the monodromy group different from  $\mc P_r(C(z))\cong \Sing L^{\ox(2g+1)}[2g+1-2r]$, the map $\mc{T}$ restricted to these other summands in \eqref{dec} is zero. 
 
 This concludes the proof of Theorem \ref{thm Iso}.

 \begin{remark}
 We may think of $H_1(C)$ as  the space
  parametrizing solutions to the KZ equations on $\Sing L^{\ox(2g+1)}[2g-1]$
  with $\ka=2$. Therefore we may want to ask for an explicit surjection of local systems of the form
 $$
 \wedge^r\Sing L^{\ox(2g+1)}[2g-1]\to \Sing L^{\ox(2g+1)}[2g+1-2r]
 $$
with $\ka=2$.
This is carried out in Theorem \ref{cor T(z)}.
 \end{remark}

\section{Maps $T(z)$ and $\bar T(z)$}
\label{sec 5}

\subsection{Cohomological relations}

Denote by $K$ the $(2g+1)\times (2g+1)$-matrix with all entries equal to 1. 
Let $Z=\on{diag}(z_1,\dots,z_{2g+1})$.   For $g=1$,
\bea
Z = 
\begin{bmatrix}
  z_1&0&0
  \\
 0&z_2&0
 \\
 0&0&z_3
\end{bmatrix} , \quad
KZ =
\begin{bmatrix}
  z_1  & z_2&z_3
\\  z_1  & z_2&z_3
\\  z_1  & z_2&z_3
\end{bmatrix} . 
\eea
Denote
\bean
\label{Ak}
A_k(z) = Z + \frac1{2k-2g-1}KZ.
\eean
Recall the notation \eqref{Peta},
\bean
\label{Peta t}
\Psi(t_1,z)^{-1/2} = \prod_{a=1}^{2g+1}(t_1-z_a)^{-1/2}, 
\qquad
\eta_{i}^{(k)}(t_1,z) 
=\Psi(t_1,z)^{-1/2}\frac{t_1^k dt_1}{t_1-z_i}.
\eean
The function $\Psi(t_1,z)^{-1/2}$ and differential form $\eta_{i}^{(k)}(t_1,z)$ are well-defined
rational  function and rational differential form on the curve $C(z)$. 

The next theorem describes the cohomological relations in $H^1(C(z))$
between the cohomology classes of the differential forms $\eta_{i}^{(k)}(t_1,z)$.
Denote by $\eta^{(k)}(t_1,z)$ the column $2g+1$-vector 
$(\eta_{1}^{(k)}(t_1,z), \dots, \eta_{2g+1}^{(k)}(t_1,z))^\intercal$. 
 
\begin{thm}
\label{thm coh}
The cohomology classes $\left[\eta^{(k)}\right]$ of the vectors $\eta^{(k)}$ 
of differential forms  satisfy the relations,
\bean
\label{coh}
\phantom{aaa}
\left[\eta^{(k)}\right] 
&=& 
A_k(z)A_{k-1}(z)\dots A_1(z) \left[\eta^{(0)}\right] , \qquad k\geq 1.
\eean
\end{thm}

\begin{proof}
The theorem follows from the relation
\bean
\label{ind k}
\left[\eta^{(k)}\right] 
=  A_{k-1}(z) \left[\eta^{(k-1)}\right].
\eean
To prove \eqref{ind k}, notice that
\bean
\label{331}
\Psi(t_1,z)^{-1/2} \frac{t_1^k dt_1}{t_1-z_i} 
&=&
\Psi(t_1,z)^{-1/2}t_1^{k-1} dt_1 + z_i \Psi(t_1,z)^{-1/2}\frac{t_1^{k-1} dt_1}{t_1-z_i},
\\
\notag
d_{t_1} \!\left( \Psi(t_1,z)^{-1/2}t_1^{k}\right)
& =&
 k \Psi(t_1,z)^{-1/2}t_1^{k-1} dt_1 
 \\
 \notag
 &-&
 \frac12\Psi(t_1,z)^{-1/2}t_1^{k-1}\sum_{a=1}^{2g+1}(t_1-z_a+z_a)\frac{dt_1}{t_1-z_a}\,.
\eean
 Hence
\bean
\label{332}
\left(k-\frac {2g+1}2\right)\left[ \Psi(t_1,z)^{-1/2}t_1^{k-1} dt_1 \right] = 
\frac12 \sum_{a=1}^{2g+1} z_a \left[\Psi(t_1,z)^{-1/2}t_1^{k-1}\frac{dt_1}{t_1-z_a}\right]\,.
 \eean
 Combining \eqref{331} and \eqref{332} we obtain
\bean
\label{333}
\phantom{aaa}
\left[ \Psi(t_1,z)^{-1/2}\frac{t_1^k dt_1}{t_1-z_i}\right] 
&=&
z_i 
\left[ \Psi(t_1,z)^{-1/2}\frac{t_1^{k-1} dt_1}{t_1-z_i}\right]
 \\
 \notag
 & + &
 \frac1{2k-2g-1}\sum_{a=1}^n z_a\left[ \Psi(t_1,z)^{-1/2}\frac{t_1^{k-1} dt_1}{t_1-z_a} \right]\,
\eean
which is \eqref{ind k}.
\end{proof}

Denote 
\bean
\label{Tk}
T^{(k)}(z) = A_k(z)A_{k-1}(z)\dots A_1(z). 
 \eean
 This is a $(2g+1)\times (2g+1)$-matrix,  $T^{(k)}(z) = (T^{(k)}_{ij}(z))$. The entries $T^{(k)}_{ij}(z)$
 are homogeneous polynomials in $z$ of degree $k$. By Theorem  \ref{thm coh}, we have
 \bean
 \label{coh ij}
\left[\eta^{(k)}_i\right]=\sum_{a=1}^{2g+1} T^{(k)}_{ia}(z) \left[\eta^{(0)}_a\right],
\qquad i=1,\dots,2g+1.
\eean
Notice that for $k=0$, we have $T^{(0)}_{ia}(z)=\delta_{ia}$.
Applying these formulas, we obtain a formula, 
\bean
\label{TTT}
{}
\\
\notag
\left[\eta^{(b_1)}_{i_1}(t_1) \wedge \dots\wedge \eta^{(b_r)}_{i_r}(t_r)\right]
=\sum_{a_1,\dots,a_r=1}^{2g+1}
\prod_{j=1}^r T^{(b_j)}_{i_j,a_j}( z) 
\left[\eta^{(0)}_{a_1}(t_1) \wedge \dots\wedge \eta^{(0)}_{a_r}(t_r)\right].
\eean

\subsection{Map  $T(z)$}

For any flat families $\ga_1(z),\dots,\ga_r(z) \in H_1(C(z))$ and $1\leq i_1<$ \dots $<i_r\leq 2g+1$, denote
\bea
N_{i_1,\dots,i_r}^{\ga_1\times\dots\times\ga_r} (z)
=
 \int_{\ga_1(z)\times\dots\times\ga_r(z)} \nu_{i_1,\dots,i_r}\,,
\eea
cf. formula \eqref{Ng r}. 
These are coordinates of the hypergeometric solution 
\bean
\label{tn1}
\phantom{aaa}
&&
\\
\notag
 N^{\ga_1\times\dots\times\ga_r}(z) 
&=&
    \sum_{1\leq i_1<\dots <i_r\leq 2g+1} N^{\ga_1\times\dots\times\ga_r}_{i_1,\dots,i_r} (z) w_{\{i_1,\dots,i_r\}}\,
\eean
of the \KZ/ equations with values in $\Sing L^{\ox(2g+1)}[2g+1-2r]$ and $\ka=2$, see formula \eqref{int 2 r}.
By formula \eqref{nu}, we have
\bean
\label{nu n}
\nu_{ i_1,\dots,i_r}
&=& 
(-1)^{\frac{n(n-1)}2}
 \sum_{\si\in S_r} (-1)^{|\si|} 
\on{Sym}_{t}\left(
\eta^{(\si(1)-1)}_{i_1,z}(t_1) \wedge\dots\wedge
\eta^{(\si(r)-1)}_{i_r}(t_r,z) \right).
\eean
Applying formula \eqref{TTT} to the right-hand side  we obtain
\bean
\label{tn11}
\phantom{aaa}
&&
\\
\notag
N^{\ga_1\times\dots\times\ga_r}_{i_1,\dots,i_r} (z) 
&=& 
(-1)^{\frac{n(n-1)}2}
\sum_{\si\in S_r}(-1)^{|\si|}
\!\!
\sum_{a_1,\dots,a_r=1}^{2g+1}
\left(
\prod_{j=1}^r T^{(\si(j)-1)}_{i_j,a_j}(z)\right)
 M^{\ga_1\times\dots\times\ga_r}_{a_1,\dots,a_r} (z). 
\eean
Notice  that $M^{\ga_1\times\dots\times\ga_r}_{a_1,\dots,a_r} (z)$ are skew-symmetric functions of the indices
$a_1,\dots, a_r$.  Notice also that the products
$\prod_{j=1}^r T^{(\si(j)-1)}_{i_j,a_j}(z)$ are homogeneous polynomials in $z_1,\dots,z_{2g+1}$  of degree $\binom{r}{2}$.

For example, for $r=2$, we have
\bean
\label{tn r=2}
N^{\ga_1\times \ga_2}_{i_1,i_2} 
&:=&
 (z_{i_1}-z_{i_2}) M^{\ga_1\times\ga_2}_{i_1,i_2} (z)
 +\frac 1{1-2g}\sum_{a=1}^{2g+1} z_a\left(M^{\ga_1\times \ga_2}_{a,i_1} (z)
 + M^{\ga_1\times\ga_2}_{a,i_2} (z)\right).
 \eean

Formula \eqref{tn11} gives a linear expression for the solution $N^{\ga_1\times\dots\times\ga_r}(z)$ 
in terms of the solution 
 \bea
M^{\ga_1\times\dots\times\ga_r}(z) =  \sum_{1\leq i_1<\dots<i_r\leq 2g+1} M_{ i_1,\dots,i_r}^{\ga_1\times\dots\times\ga_r}(z)\, w^{ \{i_1,\dots,i_r\}}.
\eea

Define a map
\bean
\label{T(z)}
&&
{}
\\
\notag
T(z) 
&:&
 \wedge^r\Sing L^{\ox(2g+1)}[2g-1] \to \Sing L^{\ox(2g+1)}[2g+1-2r], \
(M_{ i_1,\dots,i_r}) \mapsto (N_{ i_1,\dots,i_r}),
\\
\notag
N_{ i_1,\dots,i_r} 
&=& 
(-1)^{\frac{n(n-1)}2}
\sum_{\si\in S_r}(-1)^{|\si|}
\!\!
\sum_{a_1,\dots,a_r=1}^{2g+1}
\left(
\prod_{j=1}^r T^{(\si(j)-1)}_{i_j,a_j}(z)\right)
 M_{a_1,\dots,a_r} . 
\eean
The map $T(z)$ is  a homomorphism of the \KZ/ connection on
$ \wedge^r\Sing L^{\ox(2g+1)}[2g-1]$ with $\ka=2$ to the \KZ/ connection on $\Sing L^{\ox(2g+1)}[2g+1-2r]$ with $\ka=2$, 
see  \eqref{tn11}.

Recall the isomorphism 
\bean
\label{mc I 1}
\phantom{aaa}
\mc I : \wedge^rH_1(C) \to \wedge^r \Sing L^{\ox(2g+1)}[2g-1],
\quad
\ga_1(z)\wedge \dots\wedge \ga_r(z) \mapsto M^{\ga_1\times\dots\times\ga_r}(z),
\eean
in \eqref{mc I}. 
Recall the homomorphism,
\bean
\label{T iso1}
\phantom{aaaaa}
\mc T : \wedge^r H_1(C) \to \Sing L^{\ox(2g+1)}[2g+1-2r],\quad
 \ga_1(z)\wedge\dots\wedge\ga_r(z) \to 
 N^{\ga_1\times\dots\times\ga_r}(z),
 \eean
 of Theorem \ref{thm Iso}. Clearly, we have
 \bean
 \label{TtT}
\mc T \,= \,T(z)\mc I.
\eean 
This formula and the uniqueness statement
in Theorem \ref{thm Iso} give us the following theorem.

\begin{thm}
\label{cor T(z)}

The map $T(z)$ defines a nonzero epimorphism
of the \KZ/ connection on
\\
$ \wedge^r\Sing L^{\ox(2g+1)}[2g-1]$ with $\ka=2$ to the \KZ/ connection on
 $\Sing L^{\ox(2g+1)}[2g+1-2r]$ with $\ka=2$.
A nonzero homomorphism 
of the \KZ/ connection on
$ \wedge^r\Sing L^{\ox(2g+1)}[2g-1]$ with $\ka=2$ to the \KZ/ connection on
 $\Sing L^{\ox(2g+1)}[2g+1-2r]$ with $\ka=2$ 
is unique up to a multiplicative constant.
\qed

\end{thm}

 \begin{cor}
 \label{cor asah}
 All solutions over $\C$ of the \KZ/equations with values in $\Sing L^{\ox(2g+1)}[2g+1-2r]$ and $\ka=2$
are hypergeometric, that is, given by formula \eqref{tn1}. 
\qed
  \end{cor}

 \begin{rem}
 Notice that the homomorphisms $\mc T$ and $\mc I$ in formula \eqref{TtT}
  are given by hypergeometric integrals, 
  which are quite nontrivial multi-valued complex analytic functions,  
while the map $T(z)$ is given by a matrix whose entries are homogeneous polynomials in $z$ of degree $\binom{r}{2}$ with coefficients in the field of rational numbers.

 \end{rem}
 
 \subsection{Poincar\'e element}
 \label{sec 5.3}

 Recall the Poincare element $\Delta(z) \in \wedge^2 H_1(C(z))$. 
 Denote $\mc D(z) = \mc I(\Delta(z))$ where
 $\mc I : \wedge^2 H_1(C(z)) \to \wedge^2\Sing L^{\ox(2g+1)}[2g-1]$ is the isomorphism in \eqref{mc  I 1}
 for $r=2$.

 \begin{lem}[{\cite[Theorem 4.11]{VV2}}]

 We have
 \bean
 \label{J(z)}
 \mc D(z) = \sum_{1\leq i_1<i_2\leq 2g+1} \frac{2}{z_{i_1}-z_{i_2}}
 \left(
 \frac1{D_{i_1}(z)} +  \frac1{D_{i_2}(z)}\right)
w^{\{i_1,i_2\}},
\eean
where $D_a(z) = \prod_{j\ne a}(z_a-z_j)$ and 
$\frac{2}{z_{i_1}-z_{i_2}}
 \left(
 \frac1{D_{i_1}(z)} +  \frac1{D_{i_2}(z)}\right)$
is the Poincar\'e pairing of the cohomological classes of
the  differential forms
$\nu_{i_1}$ and $\nu_{i_2}$ on the curve $C(z)$.

  \end{lem}

The image under $\mc I$ of the direct decomposition in \eqref{dec} becomes
\bean
\label{dec KZ}
\wedge^r \Sing L^{\ox(2g+1)} [2g-3]
&=&
\mc P^{\on{KZ}}_r(z) \oplus \left(\mc D(z)\wedge \mc P^{\on{KZ}}_{r-2}(z)\right)\oplus
\\
\notag
&\oplus& 
\Big(\mc D(z)\wedge \mc D(z)
\wedge
\mc P^{\on{KZ}}_{r-4}(z)\Big)\oplus \dots 
\eean
where $\mc P^{\on{KZ}}_a(z) := \mc I(\mc P_{a}(C(z))$.
We have the following corollary of Theorem \ref{thm Iso} and formula \eqref{TtT}.
 
 \begin{cor}
 \label{cor TP}
 
 The map
 \bean
 \label{T restr}
 T(z)\vert_{ \mc P_r^{\on{KZ}}(z) }\  : \   \mc P_r^{\on{KZ}}(z)   \to 
 \Sing L^{\ox(2g+1)}[2g+1-2r]
 \eean
 is an isomorphism of the \KZ/ connections  with $\ka=2$, and 
 $T(z)$ restricted to other direct summands in \eqref{dec KZ} equals zero.
 \qed

 \end{cor}

 \subsection{Map $\bar T(z)$}
 
 Recall the map
 \bean
\label{TT(z)}
T(z) \  :\   \wedge^r\Sing L^{\ox(2g+1)}[2g-1] \to \Sing L^{\ox(2g+1)}[2g+1-2r]
\eean
in formula \eqref{T(z)}. Consider the dual map
\bean
\label{T^*}
 T(z)^*  :  \left(\Sing L^{\ox(2g+1)}[2g+1-2r]\right)^* \to \left(\wedge^r\Sing L^{\ox(2g+1)}[2g-1]\right)^* .
\eean
Recall that the \KZ/ connection on 
$\Sing L^{\ox(2g+1)}[2g+1-2r]$ with $\ka=-2$
is dual to the \KZ/ connection on $\Sing L^{\ox(2g+1)}[2g+1-2r]$ with $\ka=2$ under the Shapovalov form $S$,
while  the \KZ/ connection on $\wedge^r\Sing L^{\ox(2g+1)}[2g-1]$ with $\ka=-2$
is dual to the \KZ/ connection on $\wedge^r \Sing L^{\ox(2g+1)}[2g-1]$ with $\ka=2$ under 
the Shapovalov form $S^{\wedge^r}$.
Using    these     two    dualities in \eqref{T^*}      we      obtain       a        map
\bean
\label{bar tt}
\bar T(z)  \ :\   \Sing L^{\ox(2g+1)}[2g+1-2r] \to \wedge^r\Sing L^{\ox(2g+1)}[2g-1]
\eean
defined by the formula
\bean
\label{def bT} 
 S^{\wedge^r}(v, \bar T(z)w) = S(T(z)v, w), 
 \eean
 for all $ v \in  \wedge^r\Sing L^{\ox(2g+1)}[2g-1]$ and 
 $w\in  \Sing L^{\ox(2g+1)}[2g+1-2r]$.
 Clearly, the map $\bar T(z)$ is given by a matrix whose entries are homogeneous 
 polynomials in $z$ of degree $\binom{r}{2}$ with coefficients in the field of rational numbers.

 Theorem \ref{cor T(z)} and formula \eqref{def bT} imply the following theorem.

\begin{thm}
\label{thm 5.4}

The map $\bar T(z)$ in \eqref{bar tt} defines an embedding of the \KZ/ connection   on
\\
$\Sing L^{\ox(2g+1)}[2g+1-2r]$ with $\ka=-2$ to the \KZ/ connection on 
\\
$\wedge^r\Sing L^{\ox(2g+1)}[2g-1]$ with $\ka=-2$.
A nonzero homomorphism
\bea
\Sing L^{\ox(2g+1)}[2g+1-2r] \to \wedge^r\Sing L^{\ox(2g+1)}[2g-1]
\eea
of these connections is unique up to a multiplicative constant.
\qed
\end{thm}
 
 \subsection{Formula for $\bar T(z)$ for $r=2$}

For $r=2$ and $\ka = -2$, we have 
\bea
\Phi(t_1,t_2,z)^{-1/2} 
&=&
 \frac1{t_1-t_2} \Psi(t_1,z)^{1/2}\Psi(t_2,z)^{1/2},
 \\
 \bar \nu_{i_1,i_2} 
 &=&
  \Phi^{-1/2}(t_1,t_2,z) \left(\frac 1{(t_1-z_{i_1})(t_2-z_{i_2})}
+\frac 1{(t_2-z_{i_1})(t_1-z_{i_2})}\right)dt_1\wedge dt_2\,,
\\
 \bar \mu_{i_1,i_2} 
 &=&
 \Psi(t_1,z)^{1/2}\Psi(t_2,z)^{1/2} \left(\frac 1{(t_1-z_{i_1})(t_2-z_{i_2})} -
\frac 1{(t_2-z_{i_1})(t_1-z_{i_2})}\right)dt_1\wedge dt_2\,.
\eea
The problem is to express the cohomological classes of the differential forms $ \bar \mu_{i_1,i_2} $ 
in terms of the cohomological classes of the differential forms $ \bar \nu_{i_1,i_2}$.

\begin{prop}
\label{prop 5.4}

The cohomology  class of the form $\bar \mu_{i_1,i_2}$, $1\leq i_1<i_2\leq 2g+1$,
 equals the cohomology class of  the differential form
\bean
\label{bar munu}
(z_{i_1}-z_{i_2})\bar\nu_{i_1,i_2} 
+\frac 1{2g+1}\sum_{a\ne i_1} z_a \bar \nu_{a, i_1}
-\frac 1{2g+1}\sum_{a\ne i_2} z_a \bar \nu_{a, i_2}\,.
\eean
\end{prop}

\begin{proof}
We have
\bean
\label{munu}
\phantom{aaa}
 \bar \mu_{i_1,i_2} 
&=&
 \Psi(t_1,z)^{1/2}\Psi(t_2,z)^{1/2}
 \left(\frac {1}{(t_1-z_{i_1})(t_2-z_{i_2})} -
\frac {1}{(t_2-z_{i_1})(t_1-z_{i_2})}\right)dt_1\wedge dt_2
\\
\notag
&=&
 \Phi(t_1,t_2,z)^{-1/2}
 \left(\frac {t_1-t_2}{(t_1-z_{i_1})(t_2-z_{i_2})} +
\frac {t_2-t_1}{(t_2-z_{i_1})(t_1-z_{i_2})}\right)dt_1\wedge dt_2
\\
\notag
&=&
(z_{i_1}-z_{i_2}) \,\bar \nu_{i_1,i_2} + 
\\
\notag
&+&
 \Phi(t_1,t_2,z)^{-1/2}
 \left(\frac {1}{t_1-z_{i_2}} + \frac {1}{t_2-z_{i_2}}
 \  -  \
\frac {1}{t_1-z_{i_1}} - \frac {1}{t_2-z_{i_1}}
\right) dt_1\wedge dt_2\,.
\eean
The proposition follows from the following lemma. 
Denote
\bea
\bar \om_b \,:=\,  \Phi(t_1,t_2,z)^{-1/2}
 \left(\frac {1}{t_1-z_{b}} + \frac {1}{t_2-z_{b}}
 \right)dt_1\wedge dt_2\,.
\eea

\begin{lem}
The cohomology class of the form $\bar\om_b$ equals the cohomology class of the form
\\
$-\frac 1{2g+1} \sum _{a\ne b} z_a\bar \nu_{a,b} \,.$

\end{lem}

\begin{proof} 
Denote
\bea
\bar \eta_b := \Phi(t_1,t_2,z)^{-1/2}\left(
 t_1 \frac{dt_2}{t_2-z_b} - 
 t_2\frac{dt_1}{t_1-z_b}\right).
\eea 
The differential of $\bar\eta_b$ is given by the formula
\bea
&&
d \bar \eta_b 
=
\bar \om_b  + \Phi(t_1,t_2,z)^{-1/2}
\left(-\frac{dt_1-dt_2}{t_1-t_2}
+ \frac 12 \sum_{a=1}^{2g+1}
\left(\frac{dt_1}{t_1-z_a} + \frac{dt_2}{t_2-z_a}\right)\right)
\\
&&
\phantom{aaaaaaaaqqqaaaaaaaaa}
\wedge
\left(t_1 \frac{dt_2}{t_2-z_b} - t_2\frac{dt_1}{t_1-z_b}\right).
\eea
Notice that 
\bea
&&
\Phi(t_1,t_2,z)^{-1/2}\left(\frac{dt_1}{t_1-z_a} + \frac{dt_2}{t_2-z_a}\right)
\wedge
\left(t_1 \frac{dt_2}{t_2-z_b} - t_2\frac{dt_1}{t_1-z_b}\right)
=
\bar \om_b +z_a\bar \nu_{a,b}\,,
\\
&&
\Phi(t_1,t_2,z)^{-1/2}\frac{dt_1-dt_2}{t_1-t_2}
\wedge
\left(t_1 \frac{dt_2}{t_2-z_b} - t_2\frac{dt_1}{t_1-z_b}\right) = \bar\om_b + \frac 1{(t_1-z_b)(t_2-z_b)}
= \bar\om_b + \frac{z_b}2\bar \nu_{b,b}\,.
\eea
These identities imply the lemma. 
\end{proof}
\end{proof}

Fix $z^0=(z_1^0,\dots,z_{2g+1}^0)$ with distinct coordinates.
Fix non-intersecting oriented loops $c_1$, $c_2$ on the curve $C(z^0)$ which define distinct nonzero homology classes
$\ga_1(z^0), \ga_2(z^0)\in H_1(C(z^0))$. Extend the classes $\ga_1(z^0), \ga_2(z^0)$ to families 
$\ga_1(z), \ga_2(z) \in H_1(C(z))$ flat under the Gauss-Manin connection.
Then the vector
\bean
\label{sM bar -2}
\bar M^{\ga_1\times\ga_2}(z) =    \sum_{1\leq i_1<i_2\leq 2g+1} 
\bar M^{\ga_1\times\ga_2}_{i_1,i_2} (z) \,w^{\{i_1,i_2\}}
\eean
is a nonzero multi-valued solution of the \KZ/ equations with values in   
$\wedge^2\Sing L^{\ox(2g+1)}[2g-1] $ and $\ka= -2$.

The product $c_1\times c_2$ of loops defines a two-dimensional homology class $\ga(z^0)$
on the graph of the two-valued function  $\Psi(t_1,t_2,z)^{-1/2}$. Extend $\ga(z^0)$ to a family
$\ga(z)$ of two-dimensional homology classes flat with respect to the Gauss-Manin connection. 
Then the vector
\bean
\label{sN bar -2}
\bar N^{\ga}(z) =    \sum_{1\leq i_1<i_2\leq 2g+1} 
\bar N^{\ga}_{i_1,i_2} (z) \,w_{\{i_1,i_r\}}
\eean
is a nonzero multi-valued solution of the \KZ/ equations with values in   
$\Sing L^{\ox(2g+1)}[2g-3] $ and $\ka= -2$.

By Proposition \ref{prop 5.4}, we have
\bean
\label{bar MTN}
&&
\\
\notag
&&
\bar M_{i_1,i_2}^{\ga_1\times\ga_2}(z)
=(z_{i_1}-z_{i_2})\bar N_{i_1,i_2}^{\ga}(z) 
+\frac 1{2g+1}\sum_{a\ne i_1} z_a \bar N_{a, i_1}^{\ga}(z)
-\frac 1{2g+1}\sum_{a\ne i_2} z_a \bar N_{a, i_2}^{\ga}(z)\,.
\eean
This is a linear expression of the solution 
$\bar M^{\ga_1\times\ga_2}(z)$ in terms of the solution
$\bar N^{\ga}(z)$.

\vsk.2>

The monodromy of the \KZ/ equations on  $\Sing L^{\ox(2g+1)}[2g-3] $ with $\ka=-2$ is irreducible
by Theorem \ref{thm Iso}.  Thus, analytic continuation of the solution 
$\bar N^{\ga}(z)$ generates the whole space  of 
all multi-valued solutions of the \KZ/ equations on  $\Sing L^{\ox(2g+1)}[2g-3]$ with $\ka=-2$.
Therefore, formula \eqref{bar MTN} well-defines a nonzero morphism
\bean
\label{bar TT}
\tilde T(z) 
&:&
  \Sing L^{\ox(2g+1)}[2g-3] \to \wedge^2 \Sing L^{\ox(2g+1)}[2g-1], \ 
(\bar N_{i_1,i_2}) \mapsto (\bar M_{i_1,i_2}),
\\
\notag
\bar M_{i_1,i_2}
&=&
(z_{i_1}-z_{i_2})\bar N_{i_1,i_2} 
+\frac 1{2g+1}\sum_{a\ne i_1} z_a \bar N_{a, i_1}
-\frac 1{2g+1}\sum_{a\ne i_2} z_a \bar N_{a, i_2}\,,
\eean
of the \KZ/ connection on $  \Sing L^{\ox(2g+1)}[2g-3]$ with $\ka=-2$ to the \KZ/ connection on 
\\
 $ \wedge^2 \Sing L^{\ox(2g+1)}[2g-1]$ with $\ka=-2$. Here $(\bar N_{i_1,i_2})$ are coordinates of vectors in 
 \\
$\Sing L^{\ox(2g+1)}[2g-3]$ and 
$(\bar M_{i_1,i_2})$ are coordinates of vectors in 
$\wedge^2 \Sing L^{\ox(2g+1)}[2g-1]$.

\begin{thm}
\label{thm  tilde T}

For $r=2$, the maps $\bar T(z)$ and
 $\tilde T(z)$ are equal  up to a nonzero multiplicative constant.
\end{thm}

\begin{proof}
This statement is a corollary of Theorem \ref{thm 5.4}.
\end{proof}

 \subsection{Formula for $\bar T(z)$ for arbitrary $r$}
 
 In this section,  we define a map $\tilde T(z)$ for any $r$. 
 Recall
\bea
\Phi(t,z)
= 
\prod_{1 \leq i < j \leq r}\!\!  (t_i-t_j)^{2}
\prod_{a=1}^{n} \prod_{i=1}^{r} (t_i-z_a)^{-1}.
\eea
Let $q$ be a nonnegative integer such that
 $0\leq q\leq r$.  Let $a=(a_1\geq a_2\geq \dots \geq a_q\geq 0)$ be a non-increasing sequence
 of nonnegative integers.
Let  $J=\{1\leq j_{q+1}<i_{q+2}<\dots<j_{r}\leq 2g+1\}$.
 
 Denote 
 \bea
 \beta_{q, a,J} = \Phi(t,z)^{-1/2}\on{Ant}_{t_1,\dots,t_r}
 \left(t_1^{a_1}\dots t_q^{a_q} \prod_{i=q+1}^r (t_i-z_{j_i})^{-1} 
 dt_1\wedge \dots \wedge dt_r\right).
  \eea
 Define  the degree of $\beta_{q, a,J}$ by the formula
 \bea 
 \deg \beta_{q,a,J} = a_1+\dots + a_q + q.
 \eea
 The forms of minimal degree 0 are exactly the forms
 $\bar \nu_{i_1,\dots, i_r}$ defined in \eqref{nu -}.
 If $\deg \beta_{q,a,J}>0$, then $q>0$ and $a_1\geq 0$.

 \vsk.2>
 
 Consider two differential forms:
 $\beta_{q,a,J}$ and  $\beta_{q',a',J'}$.
 We say that the form
 $\beta_{q',a',J'}$ is lexicographically smaller than the form
  $\beta_{q,a,J}$ if $\deg  \beta_{q',a',J'} <
   \deg \beta_{q,a,J}$ or    $\deg  \beta_{q',a',J'} =
   \deg \beta_{q,a,J}$ and the sequence $a'$ is lexicographically smaller than
   the sequence $a$.
   
   \vsk.2>

 If $f(z_1, \dots,z_{2g+1})$ is a homogeneous polynomial in 
 $z_1,\dots, z_{2g+1}$ of degree $m$, then we say that the product
 $f(z_1, \dots,z_{2g+1}) \beta_{q, a,J}$ is a homogeneous differential $r$-form of
 degree $m+ \deg \beta_{q, a,J} $.

\vsk.2>

For a differential form $\beta_{q, a,J}$ with $\deg \beta_{q, a,J}>0$, we define
 \bea
 \alpha_{q, a,J} = \Phi(t,z)^{-1/2}\on{Ant}_{t_1,\dots,t_r}
 \left(t_1^{a_1+1}t_2^{a_2}\dots t_q^{a_q} \prod_{i=q+1}^r (t_i-z_{j_i})^{-1} dt_2\wedge \dots 
 \wedge dt_r\right).
  \eea

\begin{prop}
\label{prop rel}
Let $\deg \beta_{q,a,J}>0$. Then
\bean
\label{REL}
d \alpha_{q, a,J} =  
\frac{2g+5+2a_1-2r}2\beta_{q,a,J}
+ \sum_{(q',a',J')} 
f_{q',a',J'} (z) \beta_{q',a',J'}
  \eean
where the summation is over $(q',a',J')$ 
such that  $\beta_{q',a',J'} (z) $  is lexicographically smaller than
$\beta_{q,a,J} (z)$. Here  $f_{q',a',J'} (z) $ are homogeneous polynomials in $z$
with half-integer coefficients,
and for all nonzero summands $f_{q',a',J'} (z) \beta_{q',a',J'}$ we have
$\deg (f_{q',a',J'} (z) \beta_{q',a',J'}) = \deg \beta_{q,a,J}$.

\end{prop}

\begin{proof} We have
\bea
d \alpha_{q, a,J} =  (a_1+1) \beta_{q, a,J} - \frac12 \frac{d\Phi}{\Phi}\wedge
\alpha_{q, a,J}
\eea
where
\bea
\frac{d\Phi}{\Phi} = 2\sum_{1\leq i< j\leq r} \frac {dt_i-dt_j}{t_i-t_j}
-\sum_{i=1}^r\sum_{a=1}^{2g+1} \frac{dt_i}{t_i-z_a}.
\eea
Calculating the product $\frac{d\Phi}{\Phi}\wedge
\alpha_{q, a,J}$ we obtain the proposition. Here are the main elements of the calculation:
\bea 
\frac{t_1^{a_1+1}t_j^{a_j}}{t_1-t_j} 
&=&
t_1^{a_1} t_j^{a_j}
+ t_1^{a_1-1} t_j^{a_j+1}+\dots
+ t_1^{a_j} t_j^{a_1}+
\frac{t_j^{a_1+1}t_i^{a_j}}{t_1-t_j} \,,
\\
\frac{t_1^{a_1+1}}{(t_1-t_j)(t_j-z_{j_i})}
&=&
\frac{t_1^{a_1} 
+ t_1^{a_1-1} t_j+\dots
+  t_j^{a_1}}{t_j-z_{j_i}}+
\\
&+&
\frac{t_j^{a_1}+t_j^{a_1-1}z_{j_i}  +\dots + z_{j_i}^ {a_1} }{t_1-t_j}
+ \frac{z_{j_i}^ {a_1+1}}{(t_1-t_j)(t_j-z_{j_i})},
\eea
\bea
\frac{t_1^{a_1+1}}{t_1-z_b} 
&=&
t_1^{a_1} + z_b t_1^{a_1-1} +\dots + z_b^{a_1} + \frac{z_b^{a_1+1}}{t_1-z_b}\,.
\eea
\end{proof}

\begin{cor}
\label{cor REL}
The cohomological class $\left[\beta_{q,a,J}\right]$ of the form $\beta_{q,a,J}$ equals
\bean
\label{Rel 1}
-
\frac 2{2g+5+2a_1-2r} \sum_{(q',a',J')} 
f_{q',a',J'} (z) [\beta_{q',a',J'}].
\eean

\end{cor}

Repeatedly applying Corollary \ref{cor REL} we obtain the following statement.

\begin{cor}
\label{cor REL 2}
Let $\deg \,\beta_{q,a,J}>0$. 
Then the cohomological class $\left[\beta_{q,a,J}\right]$ of the form $\beta_{q,a,J}$ can be 
presented as a linear combination 
\bean
\label{Rel 2}
\sum_{1\leq i_1<\dots< i_r\leq 2g+1} 
h_{i_1,\dots,i_r} (z) [\bar \nu_{i_1,\dots,i_r}]
\eean
of the cohomological classes $[\bar \nu_{i_1,\dots,i_r}]$
of the forms $\bar \nu_{i_1,\dots,i_r}$ defined in 
\eqref{nu -}.
Here every $h_{i_1,\dots,i_r} (z) $ is a homogeneous polynomial in $z$ of degree 
$\deg \beta_{q,a,J}$ or equals zero. 
If $h_{i_1,\dots,i_r} (z) $ is nonzero, then its coefficients are rational numbers.
 In reduced form, the denominators of the coefficients are not divisible by any prime integer $p> 2g+5+a_1-2r$.

\end{cor}

 For $1\leq i_1,\dots ,i_r \leq 2g+1$, recall 
\bean
\label{mu bar 2}
\bar \mu_{i_1,\dots,i_r} 
&=&
\on{Sym}_{t_1,\dots,t_r}
\!
 \left(\prod_{j=1}^r \frac{\Psi(t_j,z)^{1/2}}{t_j-z_{i_j}}
  dt_1\wedge\dots\wedge dt_r\right).
 \eean
 Similarly to \eqref{munu}, we may write
\bea
\bar \mu_{i_1,\dots,i_r} 
=
\Phi(t,z)^{-1/2} 
\on{Ant}_{t_1,\dots,t_r}
\!
 \left( \frac{\prod_{a<b}(t_a-t_b)}{\prod_{j=1}^r(t_j-z_{i_j})}
  dt_1\wedge\dots\wedge dt_r\right).
 \eea
Applying Corollary \ref{cor REL 2} to this differential form 
we obtain the following statement.

\begin{prop}
\label{prop sum +2}
The cohomological class of the form $\bar \mu_{i_1,\dots,i_r}$
can be written as a sum 
\bean
\label{sum +2}
\sum_{1\leq a_1<\dots < a_r\leq 2g+1} h_{a_1,\dots,a_r}^{i_1,\dots,i_r}(z)
[\bar \nu_{a_1,\dots,a_r}]
\eean
of the cohomological classes of the forms $\bar \nu_{a_1,\dots,a_r}$.
Here every $ h_{a_1,\dots,a_r}^{i_1,\dots,i_r}(z)$ is a homogeneous polynomial in $z$ of degree 
$\binom{r}{2}$ or equals zero. If $ h_{a_1,\dots,a_r}^{i_1,\dots,i_r}(z)$
 is nonzero, then its coefficients are rational numbers.
 In reduced form the denominators of the coefficients are not divisible by any prime integer 
 $p> 2g+4-r$.
The differential forms $\bar \nu_{i_1,\dots,i_r}$ are defined in 
\eqref{nu -}.
\qed

\end{prop}

Similarly to \eqref{bar TT}, we obtain a  well-defined 
nonzero morphism
\bean
\label{bar TTT}
&&
\\
\notag
\tilde T(z) 
&:&
  \Sing L^{\ox(2g+1)}[2g+1-2r] \to \wedge^r \Sing L^{\ox(2g+1)}[2g-1], \ 
(\bar N_{i_1,\dots, i_r}) \mapsto (\bar M_{i_1,\dots, i_r}),
\\
\notag
\bar M_{i_1,\dots, i_r}
&=&
\sum_{1\leq a_1<\dots < a_r\leq 2g+1} h_{a_1,\dots,a_r}^{i_1,\dots,i_r}(z)
\,
\bar N_{a_1,\dots,a_r},
\eean
of the \KZ/ connection on $  \Sing L^{\ox(2g+1)}[2g+1-2r]$ with $\ka=-2$ to the \KZ/ connection on 
\\
 $ \wedge^r \Sing L^{\ox(2g+1)}[2g-1]$ with $\ka=-2$. Here $(\bar N_{i_1,\dots,i_r})$ 
 are coordinates of vectors in 
 \\
$\Sing L^{\ox(2g+1)}[2g+1-2r]$ and 
$(\bar M_{i_1,\dots, i_r})$ are coordinates of vectors in 
$\wedge^r \Sing L^{\ox(2g+1)}[2g-1]$.

\begin{thm}
\label{thm  tilde T}

For $r\geq 2$, the maps $\bar T(z)$ and
 $\tilde T(z)$ are equal  up to a nonzero multiplicative constant.
\end{thm}

\begin{proof}
This statement is a corollary of Theorem \ref{thm 5.4}.
\end{proof}

 \begin{cor}
 \label{cor asah -2}
 All solutions over $\C$ of the \KZ/equations with values in $\Sing L^{\ox(2g+1)}[2g+1-2r]$ and $\ka=-2$
are hypergeometric, that is, given by formula \eqref{tn1}. 
\qed
  \end{cor}

\section{$p$-hypergeometric solutions for $\ka=2$}
\label{sec 6}

\subsection{$p$-integrals}

Let $P(t_1,\dots, t_r)$ be a polynomial with coefficients in a field $\K$ of characteristic $p$. 
Given positive integers
$\ell_1,\dots,\ell_r$ denote by $\int_{\ell_1,\dots,\ell_r}P(t_1,\dots,t_r)$ 
the coefficient of $t_1^{p\ell_1-1}\dots t_r^{p\ell_r-1}$ in
$P(t_1,\dots,t_r)$, see \cite{V5}.  For $i=1,\dots,r$ we have
\bea
\int_{\ell_1,\dots,\ell_r}\frac{\der P}{\der t_i}(t_1,\dots,t_r) = 0.
\eea

In this paper,  we consider polynomials $P(t_1,\dots, t_r, z_1,\dots,z_{2g+1}) \in 
\K[t_1,\dots, t_r, z_1,$ \dots, $z_{2g+1}]$. In that case
$\int_{\ell_1,\dots,\ell_r} P(t_1,\dots, t_r, z_1,\dots,z_{2g+1})$ denotes the coefficient
of  $t_1^{p\ell_1-1}\dots t_r^{p\ell_r-1}$  in $P(t_1,\dots, t_r, z_1,\dots,z_{2g+1})$.

\subsection{$p$-hypergeometric solutions in $\Sing L^{\ox(2g+1)}[2g+1-2r]$ for $\ka=2$}

Let $p$ be an odd prime and $\K$ a field of characteristic $p$. 
In this section we remind the construction of $p$-hypergeometric solutions of the \KZ/ equations 
over the field $\K$.

Recall $\Psi(x,z)=\prod_{a=1}^{2g+1}(x-z_a)$. 
For $1\leq i_1,\dots ,i_r \leq 2g+1$, define 
\bean
\label{nu p}
\nu_{i_1,\dots,i_r} ^{(p)} 
&=&
 \on{Ant}_{t_1,\dots,t_r}
  \!
 \left(
\left(\prod_{1\leq i<j\leq r}(t_i-t_j)\right)
 \prod_{j=1}^r \frac{\Psi(t_j,z)^{(p-1)/2}}{t_j-z_{i_j}}
 \right) ,
\eean
cf. \eqref{nu}.
These are polynomials in $t,z$ with coefficients in $\K$.  
For positive integers $\ell_1,\dots, \ell_r$  denote
\bean
\label{N p}
N_{i_1,\dots,i_r}^{\ell_1,\dots, \ell_r} 
=
 \int_{\ell_1,\dots, \ell_r} \nu_{i_1,\dots,i_r}^{(p)}\,,
\eean
cf. \eqref{Ng r}.
These are polynomials in $z$ with coefficients in $\F_p\subset\K$.  

 Notice that  
$N_{i_1,\dots,i_r}^{\ell_1,\dots,\ell_r}$  are   skew-symmetric with respect to permutations of
the indices $\ell_1,\dots,\ell_r$.

 \begin{lem}
 We have $N_{i_1,\dots,i_r}^{\ell_1,\dots,\ell_r}=0$ 
 if at least one of $\ell_1, \dots,\ell_r$
is greater than $g$. 
\qed
 
 \end{lem}

Recall the standard basis 
$w_{\{i_1,\dots,i_r\}}$, $ 1\leq i_1<\dots<i_r\leq 2g+1$, of $L^{\ox(2g+1)}[2g+1-2r]$.

\begin{thm}
[{\cite[Theorem 2.4]{SV2}}]
\label{thm KZ p}

For any positive integers $\ell_1,\dots,\ell_r$, the vector 
\bea
N^{\ell_1,\dots,\ell_r}(z) =    \sum_{1\leq i_1<\dots<i_r\leq 2g+1} N^{\ell_1,\dots,\ell_r}_{i_1,\dots,i_r} (z)
\, w_{\{i_1,\dots,i_r\}}
\eea
is a solution in characteristic $p$ of the \KZ/ equations 
\eqref{KZ +2} with $\ka=2$ and values  in 
\\
$\Sing L^{\ox(2g+1)}[2g+1-2r]$.

\end{thm}

\begin{thm}
[{\cite[Theorem 3.5]{V6}}]
\label{thm lin ind}
The solutions 
$N^{\ell_1,\dots, \ell_r}(z)$, 
$1\leq \ell_1<\dots<\ell_r\leq g$,
are linearly independent over the field  $\K(z)$ if $p>2g-2r+3$.

\end{thm}

For $r=1$, this theorem is proved in \cite[Theorem 3.1]{V4}.

The solutions  $N^{\ell_1,\dots, \ell_r}(z)$, $1\leq \ell_1<\dots<\ell_r\leq g$, are called the $p$-hypergeometric solutions.

\subsection{$p$-hypergeometric solutions for $\ka =2$ and $r=1$}  

In that case, previous formulas take the following form.
For $i_1=1,\dots, 2g+1$ and $\ell_1=1,\dots,g$,  we have
\bean
\label{nu r=1 p}
\nu_{ i_1}^{(p)}(t_1)
=
\frac{\Psi(t_1, z)^{(p-1)/2}}{t_1-z_{i_1}}\, ,
\qquad
N^{\ell_1}_{i_1} = \int_{\ell_1} \nu_{ i_1}^{(p)}(t_1)\,.
\eean
By Theorem \ref{thm KZ p},  the  $L^{\otimes n}[2g-1]$-valued
function
\bean
\label{int 2 r=1 p}
N^{\ell_1}(z) =  \sum_{i_1=1}^{2g+1} N_{ i_1}^{\ell_1}(z)\, w_{ \{i_1\}}\,
\eean
 takes values in
 $ \Sing L^{\otimes (2g+1)}[2g-1]$
and satisfies the \KZ/ equations \eqref{KZ +2} with $\ka=2$.

\subsection{$p$-hypergeometric solutions for wedge-power and $\ka=2$}

The \KZ/ equations on $L^{\ox(2g+1)}[2g-1]$ with $\ka=2$ induce a system of differential equations
with values in $\wedge^r L^{\ox(2g+1)}[2g-1]$ with the invariant subspace
\bean
\label{subs}
\wedge^r \Sing L^{\ox(2g+1)}[2g-1] \subset \wedge^r L^{\ox(2g+1)}[2g-1].
\eean
The $p$-hypergeometric solutions  look as follows.  
  For $1\leq i_1,\dots ,i_r \leq 2g+1$, define the polynomials
\bean
\label{mu p}
\mu_{i_1,\dots,i_r} ^{(p)}
&=&
 \on{Ant}_{t_1,\dots,t_r}
 \left(\prod_{j=1}^r \frac{\Psi(t_j,z)^{(p-1)/2}}{t_j-z_{i_j}}
 \right) \,,
 \eean
cf. \eqref{mu}. For positive integers $\ell_1,\dots, \ell_r$  denote
\bea
M_{i_1,\dots,i_r}^{\ell_1,\dots, \ell_r} 
=
 \int_{\ell_1,\dots, \ell_r} \mu_{i_1,\dots,i_r}^{(p)}\,,
\eea
cf. \eqref{Ng r}.
These are polynomials in $z$ with coefficients in $\K$.  
\vsk.2>

Notice that $\mu_{i_1,\dots,i_r}^{(p)}$ 
is skew-symmetric with respect to permutations of
the indices $i_1,\dots,i_r$.
Hence $M_{i_1,\dots,i_r}^{\ell_1,\dots,\ell_r}$ 
is skew-symmetric with respect to permutations of $i_1,\dots,i_r$.
Notice also that  $M_{i_1,\dots,i_r}^{\ell_1,\dots,\ell_r}$ 
is    skew-symmetric with respect to permutations of
the indices $\ell_1,\dots,\ell_r$.

Recall the basis
$w^{\{i_1,\dots,i_r\}}$, $1\leq i_1<\dots < i_r \leq 2g+1$, in $\wedge^r L^{\ox(2g+1)}[2g-1]$.
By Theorem \ref{thm KZ p},  for any positive integers $\ell_1,\dots,\ell_r$, the vector 
\bean
\label{sM p}
M^{\ell_1,\dots,\ell_r}(z) =    \sum_{1\leq i_1,<\dots <i_r\leq 2g+1} M^{\ell_1,\dots,\ell_r}_{i_1,\dots, i_r} (z)\, 
w^{\{i_1,\dots, i_r\}}
\eean
is a solution of the \KZ/ equations in characteristic $p$ with $\ka=2$ and values  in
\\
 $\wedge^r \Sing L^{\ox(2g+1)}[2g-1]$.

\subsection{Comparison of solutions in characteristic $p$ for $\ka=2$}

The product $\prod_{j=1}^r T^{(\si(j)-1)}_{i_j,a_j}(z)$ in formula \eqref{tn1} can be reduced modulo $p$ if
 $p$ does not divide  $\prod_{i=1}^{r-1}(2g+1-2i)$, see the denominator in formula \eqref{Ak}.

\begin{thm}
\label{thm p-map}
Assume that  $p$ does not divide  $\prod_{i=1}^{r-1}(2g+1-2i)$. Then for 
 any positive integers $\ell_1,\dots,\ell_r$, we have
\bean
\label{tn2 r}
N^{\ell_1,\dots, \ell_r}_{i_1,\dots,i_r} (z)
=
(-1)^{\frac{n(n-1)}2}
 \sum_{\si\in S_r}(-1)^{|\si|}
\!\!
\sum_{a_1,\dots,a_r=1}^{2g+1}
\prod_{j=1}^r T^{(\si(j)-1)}_{i_j,a_j}(z) \, M^{\ell_1,\dots, \ell_r} _{a_1,\dots,a_r} (z). 
 \eean

\end{thm}

\begin{proof}
The proof follows from formula \eqref{TTT}.
\end{proof}

Theorem \ref{thm p-map} expresses  the $r$-dimensional $p$-hypergeometric solutions 
$N^{\ell_1,\dots, \ell_r}(z)$ in terms of $r$-fold wedge-products 
$M^{\ell_1,\dots, \ell_r} _{a_1,\dots,a_r} (z)$ of one-dimensional 
$p$-hypergeometric solutions.

\vsk.2>

For example, for $r=2$, we have
\bean
\label{tn r=2}
N^{\ell_1, \ell_2}_{i_1,i_2} 
&:=&
 (z_{i_1}-z_{i_2}) M^{\ell_1, \ell_2}_{i_1,i_2} (z)
 +\frac 1{1-2g}\sum_{a=1}^{2g+1} z_a(M^{\ell_1, \ell_2}_{a,i_1} (z)
 + M^{\ell_1, \ell_2}_{a,i_2} (z)).
 \eean

\subsection{Reduction of $T(z)$ modulo $p$}
\label{sec 5.6}

The map 
 \bean
 \label{T restr p}
 T(z) \vert_{ \mc P_r^{\on{KZ}}(z)}  \ : \  \mc P_r^{\on{KZ}}(z)  \to 
 \Sing L^{\ox(2g+1)}[2g+1-2r]
 \eean
 defines an isomorphism over the field $\C$ between the \KZ/ equations  with $\ka=2$ and values in
$\mc P_r^{\on{KZ}}(z)$ and  the KZ equations  with $\ka=2$ and values in
$\Sing L^{\ox(2g+1)}[2g+1-2r]$. 
This map can be reduced modulo $p$, if $p$ does not divide 
$\prod_{i=1}^{r-1}(2g+1-2i)$.
It is clear that the reduction modulo $p$ remains an isomorphism for all prime $p$ with finitely many exceptions.

\begin{cor}
\label{cor main T}  

For all prime $p$ with finitely many exceptions, all isomorphism-type characteristics of the \KZ/
 differential equations on $\mc P_r^{\on{KZ}}(z)$
and $ \Sing L^{\ox(2g+1)}[2g+1-2r]$ are identified by $ T(z)\vert_{ \mc P_r^{\on{KZ}}(z)}$. 

\end{cor}

For example, the  solution space of these two systems of differential equations have the same dimension, also,
the collections of the $p$-curvature operators of these systems of differential equations are isomorphic under
$ T(z)\vert_{ \mc P_r^{\on{KZ}}(z) }$, see Section \ref{sec 8} below. Notice that
in general,  it is easier to study the \KZ/
equations on  $\mc P_r^{\on{KZ}}(z)$ due to available explicit formulas.

\section{$p$-hypergeometric solutions for $\ka=-2$}
\label{sec 7}

\subsection{$p$-hypergeometric solutions in $\Sing L^{\ox(2g+1)}[2g+1-2r]$ for $\ka=-2$}

For $n=2g+1$ and $\ka=-2$, the \KZ/ equations are
\bean
\label{KZ -2}
\qquad
\left(\der_m \,+\, \frac 1 2\,\sum_{j\ne m}\frac{ P^{(m,j)}-1}{z_m-z_j}\right) 
I(z_1,\dots,z_{2g+1}) =0, 
\qquad  m = 1, \dots , 2g+1.
\eean
Let $p$ be an odd prime and $\K$ a field of characteristic $p$. 
In this section we remind the construction of $p$-hypergeometric solutions of the \KZ/ equations 
 over the field $\K$.

Recall $\Psi(x,z)=\prod_{a=1}^{2g+1}(x-z_a)$. 
For $1\leq i_1,\dots ,i_r \leq 2g+1$, define 
\bean
\label{bar nu p}
\bar\nu_{i_1,\dots,i_r} ^{(p)} 
&=&
 \on{Sym}_{t_1,\dots,t_r}
 \!
  \left(
\left(\prod_{1\leq i<j\leq r}(t_i-t_j)^{p-1}\right)
 \prod_{j=1}^r \frac{\Psi(t_j,z)^{(p+1)/2}}{t_j-z_{i_j}}
 \right) ,
\eean
cf. \eqref{nu p}.
These are polynomials in $t,z$ with coefficients in $\K$.  
For positive integers $\ell_1,\dots, \ell_r$  denote
\bea
\bar N_{i_1,\dots,i_r}^{\ell_1,\dots, \ell_r} 
=
 \int_{\ell_1,\dots, \ell_r} \bar \nu_{i_1,\dots,i_r}^{(p)}\,,
\eea
cf. \eqref{N p}.
These are polynomials in $z$ with coefficients in $\F_p\subset\K$.  

\begin{thm}
[{\cite[Theorem 2.4]{SV2}}]
\label{thm KZ p -2}

Let $1\leq r\leq g$. Then for any positive integers $\ell_1,\dots,\ell_r$, the vector 
\bean
\label{bar N}
\bar N^{\ell_1,\dots,\ell_r}(z) =    \sum_{1\leq i_1<\dots<i_r\leq 2g+1} \bar N^{\ell_1,\dots,\ell_r}_{i_1,\dots,i_r} (z)
\, w_{\{i_1,\dots,i_r\}}
\eean
is a solution in characteristic $p$ of the \KZ/ equations 
\eqref{KZ -2} with $\ka=-2$ and values  in 
\\
$\Sing L^{\ox(2g+1)}[2g+1-2r]$.

\end{thm}

Notice that these solutions are symmetric with respect to permutations of the indices $\ell_1,\dots, \ell_r$.

\smallskip

For $r=1$, these objects take the following form.
For $i_1=1,\dots, 2g+1$ and any positive integer $\ell_1$,  we have
\bean
\label{nu r=1 p}
\bar \nu_{ i_1}^{(p)}(t_1)
=
\frac{\Psi(t_1, z)^{(p+1)/2}}{t_1-z_{i_1}}\, ,
\qquad
\bar N^{\ell_1}_{i_1} = \int_{\ell_1} \bar \nu_{ i_1}^{(p)}(t_1)\,.
\eean
By Theorem \ref{thm KZ p -2},  the  $L^{\otimes (2g+1)}[2g-1]$-valued
function
\bean
\label{int 2 r=1 p}
\bar N^{\ell_1}(z) =  \sum_{i_1=1}^{2g+1} \bar N_{ i_1}^{\ell_1}(z)\, w_{ \{i_1\}}\,
\eean
 takes values in
 $ \Sing L^{\otimes (2g+1)}[2g-1]$
and satisfies the \KZ/ equations \eqref{KZ +2} with $\ka=-2$
 in characteristic $p$.

\vsk.2>

For degree reason, it is easy to see that $\bar N^{\ell_1}(z) $ is the zero solution if $\ell_1>g$.

\begin{thm}
\cite[Theorem 3.2]{SlV}
\label{thm lin ind -2}
If $p> 2g+1$, then the  solutions $\bar N^{\ell_1}(z) $, $\ell_1=1,\dots,g$, are 
linearly independent over the field
$\K(z)$.

\end{thm}

The solutions $\bar N^{\ell_1}(z) $, $\ell_1=1,\dots,g$, are called the $p$-hypergeometric solutions.

\begin{thm} [{Orthogonality, \cite[Corollary 1.10]{VV1}}]
\label{thm ort} 

If $p>2g+1$, then the  $p$-hypergeometric solutions 
$N^{\ell}(z)$, $\ell=1,\dots,g$, 
of the \KZ/ equations on $\Sing L^{\ox(2g+1)}[2g-1]$ with $\ka=2$
and the  $p$-hypergeometric solutions 
$\bar N^{m}(z)$, $m=1,\dots,g$, 
of the \KZ/ equations on $\Sing L^{\ox(2g+1)}[2g-1]$ with $\ka=-2$ are orthogonal
under the Shapovalov form,
\bean
\label{ort}
S\big(N^{\ell}(z), \bar N^{m}(z)\big) = 0,
\eean
for all positive integers $\ell,m$.

\end{thm}

\subsection{$p$-hypergeometric solutions for wedge-power and $\ka=-2$}

The \KZ/ equations on $L^{\ox(2g+1)}[2g-1]$ with $\ka=-2$ induce a system of differential equations
with values in $\wedge^r L^{\ox(2g+1)}[2g-1]$ with the invariant subspace
\bean
\label{subs}
\wedge^r \Sing L^{\ox(2g+1)}[2g-1] \subset \wedge^r L^{\ox(2g+1)}[2g-1].
\eean
We consider the case  $1\leq r\leq g$.
The $p$-hypergeometric solutions  look as follows.  
  For $1\leq i_1,\dots ,i_r \leq 2g+1$, define the polynomials
\bean
\label{bar mu p}
\bar \mu_{i_1,\dots,i_r} ^{(p)}
&=&
 \on{Ant}_{t_1,\dots,t_r}
 \left(\prod_{j=1}^r \frac{\Psi(t_j,z)^{(p+1)/2}}{t_j-z_{i_j}}
 \right) \,,
 \eean
cf. \eqref{mu}. For positive integers $\ell_1,\dots, \ell_r$  denote
\bea
\bar M_{i_1,\dots,i_r}^{\ell_1,\dots, \ell_r} 
=
\int_{\ell_1,\dots, \ell_r}  \bar  \mu_{i_1,\dots,i_r}^{(p)}\,,
\eea
cf. \eqref{Ng r}.
These are polynomials in $z$ with coefficients in $\K$.  
\vsk.2>

Notice that $\bar \mu_{i_1,\dots,i_r}^{(p)}$ 
is skew-symmetric with respect to permutations of the indices $i_1,\dots,i_r$.
Hence $\bar M_{i_1,\dots,i_r}^{\ell_1,\dots,\ell_r}$ 
is skew-symmetric with respect to permutations of $i_1,\dots,i_r$.
Notice also that  $\bar M_{i_1,\dots,i_r}^{\ell_1,\dots,\ell_r}$ 
is  skew-symmetric with respect to permutations of
the indices $\ell_1,\dots,\ell_r$.
 
\vsk.2>

Recall the basis
$w^{\{i_1,\dots,i_r\}}$, $1\leq i_1<\dots < i_r \leq 2g+1$, in $\wedge^r L^{\ox(2g+1)}[2g-1]$.
By Theorem \ref{thm KZ p -2},  for any positive integers $\ell_1,\dots,\ell_r$, the vector 
\bean
\label{sM p}
\bar M^{\ell_1,\dots,\ell_r}(z) =    \sum_{1\leq i_1,<\dots <i_r\leq 2g+1} \bar M^{\ell_1,\dots,\ell_r}_{i_1,\dots, i_r} (z)\, 
w^{\{i_1,\dots, i_r\}}
\eean
is a solution of the \KZ/ equations in characteristic $p$ with $\ka=-22$ and values  in
\\
 $\wedge^r \Sing L^{\ox(2g+1)}[2g-1]$.
 By Theorem \ref{thm lin ind -2}, the solutions 
$\bar M^{\ell_1,\dots,\ell_r}(z)$, $1\leq \ell_1 <\dots < \ell_r\leq g$, are linearly
 independent over the field $\K(z)$.

\subsection{Renumbering $p$-hypergeometric solutions $\bar M^{\ell_1,\ell_2}(z)$} 

Recall that 
\bean
\label{2 mu p}
\phantom{aaa}
\bar \mu_{i_1,i_2} ^{(p)}
&=&
\Psi(t_1,z)^{(p+1)/2}\Psi(t_2,z)^{(p+1)/2}\left(\frac 1{(t_1-z_{i_1})(t_2-z_{i_2})} 
-
\frac 1{(t_2-z_{i_1})(t_1-z_{i_2})}\right),
\\
\notag
\bar M^{\ell_1,\ell_2}(z)
&=&
\sum_{1\leq i_1<i_2\leq 2g+1} \bar M_{i_1,i_2} ^{\ell_1,\ell_2}(z) w^{\{i_1,i_2\}},
\qquad
\bar M_{i_1,i_2} ^{\ell_1,\ell_2}(z) 
=
 \int_{\ell_1,\ell_2} \bar \mu_{i_1,i_2} ^{(p)}\,.
 \eean
 We have $\bar M^{\ell_1,\ell_2}(z) = - \bar M^{\ell_2,\ell_1}(z)$.  
If $1\leq\ell_1<\ell_2$, then $\bar M^{\ell_1,\ell_2}(z) \ne 0$ if an only if $\ell_2\leq g$. The solutions
$\bar M^{\ell_1,\ell_2}(z)$, $1\leq \ell_1<\ell_2\leq g$ are linearly independent over $\K(z)$.
 
Denote
\bean
\label{t mu p}
{}
\\
\notag
\tilde \mu_{i_1,i_2} ^{(p)}
&=&
(t_1^p-t_2^p)
\Psi(t_1,z)^{(p+1)/2}\Psi(t_2,z)^{(p+1)/2}\left(\frac 1{(t_1-z_{i_1})(t_2-z_{i_2})} 
-
\frac 1{(t_2-z_{i_1})(t_1-z_{i_2})}\right),
\\
\notag
\tilde M^{\ell_1,\ell_2}(z)
&=&
\sum_{1\leq i_1<i_2\leq 2g+1} \tilde M_{i_1,i_2} ^{\ell_1,\ell_2}(z) w^{\{i_1,i_2\}},
\qquad
\tilde M_{i_1,i_2} ^{\ell_1,\ell_2}(z) 
=
 \int_{\ell_1,\ell_2} \tilde \mu_{i_1,i_2} ^{(p)}\,.
 \eean
Clearly, we have 
$\tilde M^{\ell_1,\ell_2}(z) =  \tilde M^{\ell_2,\ell_1}(z)$ and
\bean
\label{tb}
\tilde M^{\ell_1,\ell_2}(z)  = \bar M^{\ell_1-1,\ell_2}(z) - \bar M^{\ell_1,\ell_2-1}(z).
\eean

\begin{lem}
\label{lem tilde M}  For any positive integers $\ell_1,\ell_2$, the vector 
$\tilde M^{\ell_1,\ell_2}(z)$ is a solution in characteristic $p$ of the \KZ/ equations with values in 
$ \wedge^2\Sing L^{\ox(2g+1)}[2g-1]$ and $\ka=-2$. 
These solutions have the following properties:
\begin{enumerate}

\item[$\on{(i)}$]

$ \tilde M^{\ell_1,\ell_2+1}(z)+ \tilde  M^{\ell_1-1,\ell_2+2}(z)+ \dots + \tilde  M^{1, \ell_1+\ell_2}(z)
  =- \bar M^{\ell_1,\ell_2}(z).$

\item[$\on{(ii)}$]

For $1\leq \ell $, we have
\bean
\label{tb3}
 &&
 \tilde M^{\ell,\ell+1}(z)+ \tilde  M^{\ell-1,\ell+2}(z)+ \dots + \tilde  M^{1, 2 \ell}(z) = 0,
 \\
 \notag
 &&
 \tilde M^{\ell,\ell}(z)+ 2(\tilde  M^{\ell-1,\ell+2}(z)+ \dots + \tilde  M^{1, 2 \ell-1}(z)) = 0.
 \eean

\item[$\on{(iii)}$]

 We have $\tilde M^{\ell_1,\ell_2}(z)=0$ if $\ell_1> g+1$ or  $\ell_1> g+1$, also, 
 \bean
 \label{boundary}
  \tilde M^{1,1}(z)=\tilde M^{1,2}(z)=0,
  \qquad
\tilde M^{g,g+1}(z)=\tilde M^{g+1,g+1}(z)=0. 
\eean
\qed 

\end{enumerate}

\end{lem}

For example,   $\tilde M^{1,3}(z) = - \bar M^{1,2}(z)$,
$\tilde M^{2,2}(z)= 2\bar M^{1,2}(z)$, and $\tilde M^{2,2}(z) + 2\tilde M^{1,3}(z) = 0$.

\begin{cor}
\label{cor t vs b}

The vectors $\tilde M^{\ell_1, \ell_2}(z)$ with positive integers $\ell_1$ and $\ell_2$
generate a $\K(z)$-vector space of dimension $\binom{g}{2}$. The vectors
$ \tilde  M^{\ell_1,\ell_2+1}(z)$, $ 1\leq \ell_1<\ell_2\leq g$, 
form a basis of that space.
\qed
\end{cor}

\subsection{Renumbering $p$-hypergeometric solutions $\bar M^{\ell_1,\dots, \ell_r}(z)$}

Denote
\bean
\label{ti mu p}
{}
\\
\notag
\tilde \mu_{i_1,\dots, i_r} ^{(p)}
&=&
\prod_{1\leq i<j\leq r}(t_i^p-t_j^p)
\on{Ant}_{t_1,\dots, t_r}
\left(\frac {\Psi(t_i,z)^{(p+1)/2}}{t_i-z_{i_j}} 
\right),
\\
\notag
\tilde M^{\ell_1,\dots,\ell_r}(z)
&=&
\sum_{1\leq i_1<\dots<i_r\leq 2g+1} \tilde M_{i_1,\dots, i_r} ^{\ell_1,\dots,\ell_r}(z) \,
w^{\{i_1,\dots, i_r\}},
\qquad
\tilde M_{i_1,\dots, i_r} ^{\ell_1,\dots,\ell_r}(z) 
=
 \int_{\ell_1,\dots,\ell_r} \tilde \mu_{i_1,\dots,i_r} ^{(p)}\,.
 \eean
Clearly, the vectors $\tilde M^{\ell_1,\dots,\ell_r}(z)$ are symmetric with respect to the permulations
of $\ell_1, \dots,\ell_r$.

Let $\si\in S_r$ and $\si=(\si_1,\dots, \si_r)$. Then
\bean
\label{bar ti}
\tilde M^{\ell_1,\dots,\ell_r}(z) = (-1)^{r(r-1)/2}\sum_{\si\in S_r} (-1)^{|\si|} \bar 
M^{\ell_1+1-\si_1,\dots,\ell_r+1-\si_r}(z).
\eean
For example, for $r=2$ we have formula \eqref{tb}.

\begin{lem}
\label{lem tild M}  For any positive integers $\ell_1,\dots, \ell_r$, the vector 
$\tilde M^{\ell_1,\dots, \ell_r}(z)$ is a solution in characteristic $p$ of the \KZ/ equations with values in 
$ \wedge^r\Sing L^{\ox(2g+1)}[2g-1]$ and $\ka=-2$. 
\qed

\end{lem}

\begin{lem}
\label{cor ti vs b}

The vectors $\tilde M^{\ell_1, \dots, \ell_r}(z)$ with positive integers $\ell_1,\dots,\ell_r$
generate a $\K(z)$-vector space of dimension $\binom{g}{r}$. 
\end{lem}

\begin{proof}
For $r=2$, the lemma follows from Lemma \ref{lem tilde M}. We give a proof for $r=3$. 
For  $r>3$ the proof is similar.

The goal is  to express an arbitrary $\bar M^{\ell_1,\ell_2,\ell_3}(z)$, $1\leq \ell_1<\ell_2<\ell_3\leq g$ as a linear 
combination of $\tilde  M^{a_1,a_2,a_3}(z)$. We use  the lexicographical order on the set of
triples $(\ell_1,\ell_2,\ell_3)$. The lexicographically minimal triple is $(1,2,3)$ and 
$\bar M^{1,2,3}(z) = \tilde M^{1,3,5}(z)$. 

Assume that the set of all triples is partitioned into two subsets $X$ and $Y$ where for 
every $(\ell_1,\ell_2,\ell_3)\in X$ the polynomial $\bar M^{\ell_1,\ell_2,\ell_3}(z)$
is already expressed as a linear combination of the polynomials $\tilde M^{a_1,a_2,a_3}(z)$
and for
every $(\ell_1,\ell_2,\ell_3)\in Y$ the polynomial $\bar M^{\ell_1,\ell_2,\ell_3}(z)$
is not yet expressed as a linear combination of the polynomials $\tilde M^{a_1,a_2,a_3}(z)$.

Choose the lexicographically minimal triple 
 $(\ell_1,\ell_2,\ell_3)\in Y$  and consider
 the polynomial $\bar M^{\ell_1,\ell_2,\ell_3}(z)$. Then
 \bean
 \label{lexi}
 \tilde M^{\ell_1,\ell_2+1,\ell_3+2}(z)
&=&
 \bar M^{\ell_1,\ell_2,\ell_3}(z)
- \bar M^{\ell_1,\ell_2-1,\ell_3+1}(z)
-\bar M^{\ell_1-1,\ell_2+1,\ell_3}(z)
\\
\notag
&+&
 \bar M^{\ell_1-1,\ell_2-1,\ell_3}(z)
+ \bar M^{\ell_1-2,\ell_2+1,\ell_3+1}(z)
-\bar M^{\ell_1-2,\ell_2,\ell_3+2}(z).
\eean
The last five polynomials $\bar M^{\ell_1,\ell_2-1,\ell_3+1}(z)$,
$\bar M^{\ell_1-1,\ell_2+1,\ell_3}(z)$,
$ \bar M^{\ell_1-1,\ell_2-1,\ell_3}(z)$,
$\bar M^{\ell_1-2,\ell_2+1,\ell_3+1}(z)$,
$\bar M^{\ell_1-2,\ell_2,\ell_3+2}(z)$ are either zero or correspond to solutions with triples in
$X$. Hence formula \eqref{lexi} provides an expression for 
$ \bar M^{\ell_1,\ell_2,\ell_3}(z)$ in terms of $\tilde  M^{a_1,a_2,a_3}(z)$. This proves the lemma
for $r=3$.
\end{proof}

\subsection{Reduction of $\bar T(z)$ modulo $p$}
\label{sec 7.3}

Recall that the map  $\bar T(z)$ in \eqref{bar tt},
\bean
\label{bar ttt} 
\bar T(z)  :  \Sing L^{\ox(2g+1)}[2g+1-2r] \to \wedge^r\Sing L^{\ox(2g+1)}[2g-1]
\eean
defines an embedding
over the field $\C$ of the \KZ/ connection  with $\ka=-2$ and values in
  $\Sing L^{\ox(2g+1)}[2g+1-2r]$ to   the \KZ/ connection  with 
 $\ka=-2$ and   values in
$\wedge^r \Sing L^{\ox(2g+1)}[2g-1]$.

The map $\bar T(z)$ is defined by formula  \eqref{def bT} in terms of $T(z)$.
Since $T(z)$  can be reduced modulo $p$
for all $p$ with finitely many exceptions,  and the reduction 
for  all $p$ with finitely many exceptions keeps the properties of
  $T(z)$ over complex numbers, we obtain the following corollary.

\begin{cor}
\label{cor bar red}  

The reduction modulo $p$ of the map  $\bar T(z)$ is well-defined for 
all $p$ with finitely many exceptions, and the reduction is an embedding in characteristic $p$ 
of the \KZ/ connection  with $\ka=-2$ and values in
  $\Sing L^{\ox(2g+1)}[2g+1-2r]$ to   the \KZ/ connection  with 
 $\ka=-2$ and   values in
$\wedge^r \Sing L^{\ox(2g+1)}[2g-1]$. 

\end{cor}

Recall also that the map \eqref{bar TTT},
\bean
\label{bar TTTT}
&&
\\
\notag
\tilde T(z) 
&:&
  \Sing L^{\ox(2g+1)}[2g+1-2r] \to \wedge^r \Sing L^{\ox(2g+1)}[2g-1], \ 
(\bar N_{i_1,\dots, i_r}) \mapsto (\bar M_{i_1,\dots, i_r}),
\\
\notag
\bar M_{i_1,\dots, i_r}
&=&
\sum_{1\leq a_1<\dots < a_r\leq 2g+1} h_{a_1,\dots,a_r}^{i_1,\dots,i_r}(z)
[\bar N_{a_1,\dots,a_r}],
\eean
is proportional to $T(z)$ and 
defines an embedding  over the field $\C$ of the \KZ/ connection  with $\ka=-2$ and values in
  $\Sing L^{\ox(2g+1)}[2g+1-2r]$ to   the \KZ/ connection  with 
 $\ka=-2$ and   values in
$\wedge^r \Sing L^{\ox(2g+1)}[2g-1]$. 

It is clear that this map can be reduced modulo $p$ for all $p$ with finitely many exceptions.
It is also clear that the reduction modulo $p$ remains an embedding for all prime $p$ with finitely many exceptions.

\subsection{Comparison of solutions in characteristic $p$ for $\ka=-2$}

\begin{thm}
\label{thm p-map -2}
Assume that  $p$ is such that $\tilde T(z)$ can be reduced modulo $p$. Then for any positive
integers
 $\ell_1,\dots,\ell_r$,
we have
\bean
\label{p map -2}
\tilde M_{i_1,\dots, i_r}^{\ell_1,\dots,\ell_r}(z)
&=&
\sum_{1\leq a_1<\dots < a_r\leq 2g+1} h_{a_1,\dots,a_r}^{i_1,\dots,i_r}(z)\,
\bar N_{a_1,\dots,a_r}^{\ell_1,\dots,\ell_r}\,.
\eean

\end{thm}

\begin{proof}
In characteristic $p$, we have  
\bea
\tilde \mu_{i_1,\dots, i_r} ^{(p)}
&=&
\prod_{1\leq i<j\leq r}(t_i^p-t_j^p)
\on{Ant}_{t_1,\dots, t_r}
\left(\frac {\Psi(t_i,z)^{(p+1)/2}}{t_i-z_{i_j}} 
\right)
\\
&=&
\prod_{1\leq i<j\leq r}(t_i-t_j)^{p-1}
\on{Sym}_{t_1,\dots, t_r}
\left(\prod_{1\leq i<j\leq r}(t_i-t_j)\frac {\Psi(t_i,z)^{(p+1)/2}}{t_i-z_{i_j}} 
\right)
\eea
Then continuing like in the proof of Proposition \ref{prop sum +2} we conclude that the $p$-integral of 
$\tilde \mu_{i_1,\dots, i_r} ^{(p)}$ associated with positive integers $\ell_1, \dots, \ell_r$  is given by the right-hand side in formula 
\eqref{p map -2}.
\end{proof}

\begin{cor}
\label{cor t vs b N}
Assume that $p$ is such that $\bar T(z)$ is an embedding. Then 
the $p$-hypergeometric solutions  $\bar N^{\ell_1, \dots, \ell_r}(z)$ 
generate a $\K(z)$-vector space of dimension $\binom{g}{r}$.
\end{cor}

For example, for $g=2$, the vectors $\bar N^{\ell_1, \ell_2}(z)$ generate a one-dimensional space 
$\K(z)$-vector space with a basis vector $\bar N^{1,3}(z)$.  
We also have
\bea
  \bar N^{1,1}(z)=\bar N^{1,2}(z)=0,
  \quad
\bar N^{2,2 }(z)+2\bar N^{1, 3}(z)=0,
\quad
\bar N^{2,3}(z)=\bar N^{3,3}(z)=0.
\eea

\subsection{Orthogonal relations on $\Sing L^{\ox(2g+1)}[2g+1-2r]$}
\label{sec ort}

Recall the orthogonal relations of Theorem \ref{thm ort},
\bean
\label{ort n}
S\big(N^{\ell_1}(z), \bar N^{\ell_2}(z)\big) = 0, \qquad \ell_1,\ell_2 =1, \dots, g,
\eean
where $S$ is the Shapovalov form on $\Sing L^{\ox(2g+1)}[2g-1]$  and 
$N^{\ell_1}(z), \bar N^{\ell_2}(z)$ are
\\
$\Sing L^{\ox(2g+1)}[2g-1]$-valued  $p$-hypergeometric solutions of the \KZ/equations
with $\ka=2$ and $-2$, respectively.
This implies
\bea
S^{\wedge^r} (M^{\ell_1,\dots, \ell_r}(z), \bar M^{m_1,\dots, m_r}(z))=0
\eea
for all positive integers $\ell_1,\dots, \ell_r$ and $m_1,\dots,m_r$.

\begin{thm}
\label{thm ort r=2}

For all prime $p$ with finitely many exceptions, we have 
\bean
\label{ort n2}
S\big(N^{\ell_1,\dots, \ell_r}(z), \bar N^{m_1,\dots, m_r}(z)\big)=0,
\eean
for all positive integers $\ell_1,\dots, \ell_r$ and $m_1,\dots,m_r$.
 Here 
 $S$ is the Shapovalov form on
 \\
  $\Sing L^{\ox(2g+1)}[2g+1-2r]$  and 
$N^{\ell_1,\dots, \ell_r}(z), \bar N^{m_1,\dots, m_r}(z)$ are
$\Sing L^{\ox(2g+1)}[2g+1-2r]$-valued  $p$-hypergeometric solutions of the \KZ/equations
with $\ka=2$ and $-2$, respectively.

\end{thm}

\begin{proof}  

Recall that $\bar T(z) = \on{const}\tilde  T(z)$ by Theorem \ref{thm tilde T}. Hence
\bea
&&
S\big(N^{\ell_1, \dots, \ell_r}(z), \bar N^{m_1,\dots,m_r}(z)\big) =
S\big(T(z) M^{\ell_1, \dots, \ell_r}(z), \bar N^{m_1,\dots, m_r}(z)\big)
\\
&&
\phantom{aaa}
= S^{\wedge^r}\big(M^{\ell_1, \dots, \ell_r}(z), \bar T(z) \bar N^{m_1,\dots,m_r}(z)\big)
= 
\on{const} S^{\wedge^r}\big(M^{\ell_1, \dots, \ell_r}(z), \tilde T(z) \tilde N^{m_1,\dots,m_r}(z)\big)
\\
&&
\phantom{aaa}
= 
\on{const}S^{\wedge^r}\big(M^{\ell_1,\dots, \ell_r}(z), \tilde M^{m_1,\dots,m_r}(z)\big) = 0.
\eea
\end{proof}

\section{$p$-curvature for $\ka=\pm 2$}
\label{sec 8}

\subsection{Definition of $p$-curvature} 

(\cite[Section 5]{K1}, also see, for example, \cite[Section 2.4]{EV}.) 
Let $\K$ be a field of characteristic $p$. Let $\mathcal V$ be a $\K$-vector bundle over base $X$. 
Let $\nabla$ be a flat connection on  $\mathcal V$.
 If $x_1,...,x_n$ are local coordinates on $X$, then we have commuting operators $\nabla_1,...,\nabla_n$ of covariant partial derivatives with respect to these coordinates, which act on the space $\Gamma_{\rm rat}(\mathcal V)$ of rational sections of 
 $\mathcal V$ and determine $\nabla$.  
The operators
$$
C_i=C_i(\nabla):=\nabla_i^p\ :\  \Gamma_{\rm rat}(\mathcal V)\to \Gamma_{\rm rat}(\mathcal V),
\qquad i=1,\dots,n,
$$ 
are $\K(X)$-linear and called  the  $p$-curvature operators of $\nabla$. We have $[\nabla_i, C_j] =[C_i,C_j]=0$ for all $i,j$.  

\vsk.2>
An application of $p$-curvature is the computation of the space of local flat sections
 of $\nabla$.
It is known that the intersection of kernels $\cap _{i=1}^n \on{ker} C_i$ of the $p$-curvature operators
coincides with the space generated by flat sections of $\nabla$,  see the Cartier descent theorem
\cite[Theorem 5.1]{K1}.

\medskip

In this section we shall discuss the $p$-curvature of the \KZ/ equations with $\ka=2$ on 
the spaces $\Sing L^{\ox(2g+1)}[2g-1]$, $\wedge^r \Sing L^{\ox(2g+1)}[2g-1]$, and 
$\Sing L^{\ox(2g+1)}[2g+1-2r]$. The $p$-curvature operators on 
$\Sing L^{\ox(2g+1)}[2g-1]$ are described 
in \cite{VV1}. Using this information we describe  the $p$-curvature on
$\wedge^r \Sing L^{\ox(2g+1)}[2g-1]$.  Finally, using the isomorphism $T(z)$, we 
obtain information on the $p$-curvature on $\Sing L^{\ox(2g+1)}[2g+1-2r]$.

\subsection{$p$-curvature on $\Sing L^{\ox(2g+1)}[2g-1]$}
\label{sec 8.2}

Assume that $p>2g+1$. Let $\K$ be a field of characteristic $p$. 
 Denote  $D_a(z) = \prod_{j\ne a} (z_a-z_j).$
Recall  the  $p$-hypergeometric solutions 
$N^{\ell}(z)$, $\ell=1,\dots,g$, 
of the \KZ/ equations on 
$\Sing L^{\ox(2g+1)}[2g-1]$ with $\ka=2$
and the  $p$-hypergeometric solutions 
$\bar N^{\ell}(z)$, $\ell=1,\dots,g$, 
of the \KZ/ equations on $\Sing L^{\ox(2g+1)}[2g-1]$ with $\ka=-2$.

We have coordinates $z=(z_1,\dots,z_{2g+1})$ on the base of the \KZ/ connection and hence 
we have
the corresponding $p$-curvature operators $C_a(z)$, $a=1,\dots, 2g+1$.

\begin{thm}
[{\cite[Formula (5.14)]{VV1}}] 
\label{thm cu r=1}
For $a=1,\dots, 2g+1$, the $p$-curvature operators of the \KZ/ equations on
$\Sing L^{\ox(2g+1)}[2g-1]$ with $\ka=2$ are given by the formula
\bean
\label{cu r=1 +2}
C_a(v)\, =\, \frac{-1}{2 D_a(z)^p}
\,S\!\left(\sum_{m=1}^g z_a^{p(m-1)} \bar N^m(z), v\right) 
\sum_{\ell=1}^g z_a^{p(\ell-1)} N^\ell(z)\,,
\eean
where $v\in\Sing L^{\ox(2g+1)}[2g-1]$ and $S$ is the Shapovalov form.
\end{thm}

This formula shows that $C_a(z)$ is an operator of rank 1 whose image is generated by
the vector $\sum_{\ell=1}^g z_a^{p(\ell-1)} N^\ell(z)$ and whose kernel is the orthogonal hyperplane in 
$\Sing L^{\ox(2g+1)}[2g-1]$  to the vector 
$\sum_{m=1}^g z_a^{p(m-1)} \bar N^m(z)$ under the Shapovalov form.

To simplify calculations we define the normalized $p$-curvature operators by the formula
\bean
\label{red p-cu}
\tilde C_a(z)(v) \,=\,
S\!\left(\sum_{m=1}^g z_a^{p(m-1)} \bar N^m(z), v\right) 
\sum_{\ell=1}^g z_a^{p(\ell-1)} N^\ell(z)\,.
\eean
This operator has the same image and kernel as $C_a(z)$. The operator $\tilde C_a(z)$ is a homogeneous polynomial in $z$ of degree  $(2g-1)p$.

Denote 
\bean
\label{V}
V=\Sing L^{\ox(2g+1)}[2g-1].
\eean
This is a $2g$-dimensional vector space over $\K$.  To study the $p$-curvature operators
we fix $z=(z_1,\dots,z_n)$ with distinct coordinates and introduce on $V$ a basis depending on $z$, the basis consisting of vectors
$v_1,\dots, v_g$, $w_1,\dots,w_g$ and associated coordinates $x_1,\dots,x_g$, $y_1,\dots,y_g$ such that
\bean
\label{pgb}
x_i (v_j)=\delta_{ij}, \quad x_i(y_j)=0, \quad y_i(v_j)=0,\quad y_i(w_j)=\delta_{ij}\,.
\eean
Namely, we define 
\bea
w_{i}= N^i(z),  \qquad x_i= S(\cdot, \bar N_i(z)), \qquad i=1,\dots, g.
\eea
Then we choose vectors $v_1, \dots, v_g$  with the property
\bea
x_i(v_j)=\dl_{ij}\,,\qquad i,j=1,\dots,g.
\eea
The vectors $v_1,\dots, v_g$ are defined uniquely up to addition of linear combinations of $w_{1}\dots, w_{g}.$
The vectors $v_1,\dots,v_{g},$  $w_1,\dots,w_g$ form a basis of $V$. 
Finally, using conditions \eqref{pgb} we uniquely determine  the associated coordinates 
$y_{1},\dots,  y_{g}$.
We call this basis   preliminary good.

\vsk.2>

\vsk.2>
Denote $V_{0,1} = \langle w_{i}\mid i=1,\dots,g\rangle\subset V$.

\begin{lem}
\label{lem v-flat}
The subspace $V_{0,1}$ is  \KZ/-invariant.
The projections of the vectors $v_1,\dots, v_g$  to  $V/V_{0,1}$ do not depend 
on the choice of a preliminary good basis of $V$.
 The vectors $v_1,\dots, v_g$ project to a \KZ/-flat basis of $V/V_{0,1}$.

\end{lem}

\begin{proof} The first two statements are trivial since $V_{0,1}$ is generated by solutions $w_1,\dots,w_g$. 
To prove the second statement we differentiate the identity 
$\delta_{ij}= x_j(v_i) = S(v_i, \bar N^j)$, where $i,j=1,\dots,g$. We obtain
\bea
0=\der_a S(v_i, \bar N^j) = S(\nabla_a^{\on{KZ},\ka= 2} v_i, \bar N^j)
+ S(v_i, \nabla_a^{\on{KZ}, \ka=-2}  \bar N^j) =
S(\nabla_a^{\on{KZ}, \ka=2} v_i, \bar N^j).
\eea
By the orthogonality relations \eqref{ort}, the subspace $V_{0,1}$ 
is the annihilator of the subspace $\langle \bar N^j(z)\mid j=1,\dots,g\rangle$ under the Shapovalov form.
Hence $\nabla_a^{\on{KZ},\ka= 2} v_i \in V^{0,1}$ and the  lemma follows.
\end{proof}

For $a=1,\dots,2g+1$, define $2g+1$ vectors of $V$  and $2g+1$ linear functions on $V$:
\bea
K_a=w_{1}+z_a^p w_{2} + \dots + z_a^{p(g-1)}w_{g}\,,
\qquad  
L_a=x_1+z_a^p x_2+ \dots + z_a^{p(g-1)} x_g\,.
\eea
We have
\bean
\label{KL span}
\langle K_a\mid a
=1,\dots,2g+1\rangle 
&=&
 \langle w_1,\dots,w_g\rangle,
\\
\notag
\langle L_a\mid a=1,\dots,2g+1\rangle
& =&
 \langle x_1,\dots,x_g\rangle.
\eean
Formula \eqref{red p-cu} takes the following form,
\bean
\label{atv}
\tilde C_a(z)(v) = L_a(v) K_a\,.
\eean
It follows from \eqref{atv}, that $\cap_{a=1}^{2g+1} \on{ker} \tilde C_a(z) = \langle w_1,\dots,w_g\rangle$ has dimension $g$. 
Hence we obtain the following corollary.

\begin{cor}
[{\cite[Theorem 1.8]{VV1}}]
\label{cor cu r=1}

All solutions in characteristic $p$ of the \KZ/ equations on 
$\Sing L^{\ox (2g+1)}[2g-1]$ 
with $\ka=2$ are generated by the $p$-hypergeometric solutions  $N^\ell(z)$, $\ell=1,\dots,g$.

\end{cor}

\subsection{$p$-curvature operators on $\wedge^r\Sing L^{\ox(2g+1)}[2g-1]$}

Fix $z=(z_1,\dots,z_{2g+1})$ such that $z_1^p,\dots,z_{2g+1}^p$ are distinct.  
 Let $v_1,\dots,v_g$, $w_1,\dots, w_g$ be a preliminary good basis of $V$.
For any $r$, the normalized $p$-curvature operator $\tilde C_a(z)$ is defined on $\wedge^rV$ by the formula
\bea
\tilde C_a (u_1\wedge u_2\wedge \dots \wedge u_r) = \tilde C_a(u_1)\wedge \dots \wedge u_r
+ \dots + u_1\wedge \dots \wedge u_{r-1}\wedge \tilde C_a(u_r).
\eea
Clearly we have
\bean
\label{CC=0}
\tilde C_a(z) \tilde C_a(z) =0, \qquad a=1,\dots,2g+1,
\eean
 for any $r$.

Let $r\leq g$.  It is convenient to introduce the decomposition
\bea
\wedge^rV = \oplus_{k=0}^r V_{k, r-k},
\eea
where $V_{r,k}$ is generated by vectors of the form
$v_{i_1}\wedge\dots\wedge v_{i_k} \wedge
w_{j_1}\wedge\dots\wedge w_{j_{r-k}}$. We have 
$\tilde C_a(z)(V_{k,r-k})\subset V_{k-1,r-k+1}$. 
In particular this means that if $v= \sum_{k=0}^r v_{k,r-k}$ is such that $ v_{k,r-k} \in  V_{k,r-k}$ for all $k$ and
$\tilde C_a(z)(v)=0$, then $\tilde C_a(z)(v^{k,r-k})=0$ for every $k$.

We also introduce a filtration on $\wedge^r V$,
\bea
 V_0 \subset V_1 \subset \dots\subset V_r= \wedge^r V,
\eea
where $V_k = V_{k,r-k} \oplus V_{k-1,r-k+1}\oplus \dots\oplus V_{0,k}$.
By Lemma \ref{lem v-flat}, this filtration does not depend on the choice of a preliminary good basis of $V$.
The projections to $V_k/V_{k-1}$ of the vectors
$v_{i_1}\wedge\dots\wedge v_{i_k} \wedge
w_{j_1}\wedge\dots\wedge w_{j_{r-k}}$ also do not depend on the choice of a preliminary good basis.
The filtration is invariant with respect to the \KZ/ connection on $\wedge ^rV$ with $\ka=2$.
The projections to $V_k/V_{k-1}$ of the vectors 
$v_{i_1}\wedge\dots\wedge v_{i_k} \wedge
w_{j_1}\wedge\dots\wedge w_{j_{r-k}}$ are flat sections of the bundle with fiber $V_k/V_{k-1}$.

\subsection{Kernel of $p$-curvature on $\wedge^2 V$}

Define
\bea
\mc K_2:=\cap_{a=1}^{2g+1} \on{ker}(\tilde C_a : \wedge^2 V\to \wedge^2 V) .
\eea
Let $v_1,\dots,v_{g},$  $w_1,\dots,w_g$ be a preliminary good basis on $V$ with  associated coordinates 
$x_1,\dots,x_{g},$  $y_1,\dots,y_g$.

\begin{thm}
\label{thm int c}

We have
\bean
\label{int ok}
\mc K_2= \langle \delta, \ w_i\wedge w_j \mid 1\leq i<j \leq g\rangle
\eean
where $\delta =v_{1}\wedge w_1+\dots + v_g\wedge w_g$\,.

\end{thm}

\begin{proof} Recall  
$\wedge^2V = V_{2,0}\oplus V_{1,1}\oplus V_{0,2}$.
It is clear that  $V_{0,2} \subset \mc K_2$.

\begin{lem}
\label{lem 6.5}

We have
$\mc K_2\cap V_{1,1} =  \langle \delta\rangle $ and 
$\mc K_2\cap V_{2,0} = 0$.

\end{lem}

The lemma implies the theorem.

\medskip
\noindent
{\it Proof of lemma.}
Let $x\in V_{1,1}$,
$x = \sum_{i=1}^gv_i\wedge \sum_{j=1}^{g} c_{ij}w_j$\,.
Then  $\tilde C_a x = K_a \wedge \sum_{i,j=1}^g z_a^{(i-1)p} c_{ij}w_j$\,.
We have $x\in \mc K_2$ if and only if the vector
$x_a:= \sum_{i,j=1}^g z_a^{(i-1)p} c_{ij}w_j$
is proportional to the vector $K_a=w_{1}+z_a^p w_{2}+\dots + z_a^{p(g-1)}w_{g}$ for $a=1,\dots,2g+1$.
That is, for $i=1,\dots, g-1$, we  have the equations
\bea
z_a^p (c_{1,i} + z_a^p c_{2, i}+\dots +z_a^{p(g-1)}c_{g, i} )
&=&
c_{1, i+1} + z_a^p c_{2, i+1}+\dots + z_a^{p(g-1)}c_{g, i+1}.
\eea
This equation can be written as 
\bea
-c_{1, i+1} + z_a^p(c_{1,i} - c_{2, i+1}) + \dots 
+ 
z_a^{p(g-1)}(c_{g-1, i} - c_{g, i+1})+ z_a^{pg} c_{g, i}=0.
\eea
Given $i$, the system of $g+1$ equations,
\bea
-c_{1, i+1} + z_a^p(c_{1,i} - c_{2, i+1}) + \dots 
+ 
z_a^{p(g-1)}(c_{g-1, i} - c_{g, i+1})+ z_a^{pg} c_{g, i}=0, \quad a=1,\dots, g+1,
\eea
 implies that
\bea
c_{1, i+1} =0,\ \ c_{1,i} = c_{2, i+1}, \   \dots , \ 
c_{g-1, i} = c_{g, i+1}, \ \ c_{g, i}=0.
\eea
It is easy to see that the union of these systems of equations for $i=1,\dots,g$ implies that
\bea
x = \epsilon\, (v_1\wedge w_1 + v_2\wedge w_2 +\dots +  v_g\wedge w_g)
\eea
where $\epsilon$ is a scalar parameter. This proves the first equality in Lemma \ref{lem 6.5}. 

Let  $y = \sum_{1\leq i<j\leq g} c_{ij} v_i\wedge v_j \in V_{2,0}$.
Then  $\tilde C_a(y)=K_a\wedge 
 \sum_{1\leq i<j\leq g} c_{ij} (z_a^{p(i-1)} v_j- z_a^{p(j-1)}v_i )$.
If $ \tilde C_a(y)=0$, then
\bea
-z_a^{p} c_{1,2} - z_a^{2p} c_{1,3} -\dots- z_a^{p(g-1)} c_{1,g}
&=& 0,
\\
c_{1,2} - z_a^{2p} c_{2,3} -\dots- z_a^{p(g-1)} c_{2,g}
&=& 0,
\\
&\dots &
\\
 c_{1,g} + z_a^{p} c_{2,g} +\dots +z_a^{p(g-2)} c_{g-1,g}
&=& 0.
\eea
The system consisting of the  first equation for $a=1,\dots, g-1$ gives
$c_{1,2} =c_{1,3}=\dots=c_{1,g}=0$. Then the system consisting of the second equation 
for $a=1,\dots, g-1$ gives
$c_{2,3}=\dots= c_{2,g}=0$, and so on.  This proves the second equality in Lemma \ref{lem 6.5}.
\end{proof}

\subsection{Poincare pairing and good bases}
\label{sec 8.5}

Fix $z=(z_1,\dots,z_{2g+1})$ such that $z_1^p,\dots,z_{2g+1}^p$ are distinct.  
Let $v_1,\dots,v_{g},$  $w_1,\dots,w_g$ be a preliminary good basis of $V$.
Let $\mc D(z) \in \wedge^2 V$ be the element corresponding to Poincare pairing.
The element $\mc D(z)$ has the following properties.
\begin{enumerate}

\item[(i)]
The element $\mc D(z)$ defines a non-degenerate skew-symmetric form on the space $V^*$ dual to $V$.

\item[(ii)]  

 We have $\tilde C_a(\mc D(z))=0$, $a=1,\dots,2g+1$, since $\mc D(z)$ is invariant with respect to
the Gauss-Manin connection.

\item[(iii)]   
The kernel $\mc P_2^{\on{KZ}}(z)$ of the map 
 \bea
 (\mc D(z)\wedge)^{g-1}\ :\  \wedge^2V\to \wedge^{2g}V
  \eea
  of multiplication by $( \mc D(z)\wedge)^{g-1}$ is invariant under each $p$-curvature operator $\tilde C_a$,
  since the kernel of this map is preserved by the Gauss-Manin connection.

\end{enumerate}

\begin{lem}
Up to proportionality, the element $\mc D(z)$ equals 
$ \delta + \sum_{1\leq i<j\leq g} c_{ij}w_i\wedge w_j$ with suitable coefficients $c_{ij}$.

\end{lem}

\begin{proof} 

Property (ii) implies that $\mc D(z) \in  \langle \delta, w_i\wedge w_j \mid 1\leq i<j \leq g\rangle$. Property
(i) implies that $\mc D(z) $ is not proportional to a linear combination 
$\sum_{1\leq i<j\leq g} c_{ij}w_{i}\wedge w_{j}$\,.
\end{proof}

Let $\mc D(z)$ be proportional to
$\delta + \sum_{1\leq i<j\leq g} c_{ij}w_i\wedge w_j$ for some coefficients $c_{ij}$. Define a new preliminary good 
basis $\bar v_1,\dots, \bar v_{g},$  $\bar w_1,\dots, \bar w_g$ of $V$ by the formulas
$v_i= \bar v_i+\sum_{j=i+1}^{g} c_{ij} w_j$,  $w_i=\bar w_i$\,, $ i=1,\dots,g$.
Then
\bean
\label{nd}
\delta + \sum_{1\leq i<j\leq g} c_{ij}w_i\wedge w_j = \sum_{i=1}^g \bar v_i\wedge \bar w_i\,.
\eean
We say that a basis $v_1,\dots,v_{g},$  $w_1,\dots,w_g$ of $V$
 is good if it is preliminary good
and $\mc D(z)$ is proportional to $\sum_{i=1}^g v_i\wedge w_{i}$\,.
Good bases exist since the above basis $\bar v_1,\dots,\bar v_{g},$  $\bar w_1,\dots,\bar w_g$  is good.
A good basis is not unique.

Let $v_1,\dots,v_{g},$  $w_1,\dots,w_g$ be a good basis of  $V$.
Define the decomposition $\mc P_2^{\on{KZ}} = \mc P_{2,0} \oplus  \mc P_{1,1}\oplus \mc P_{0,2}$,
$\mc P_{i,j} = \mc P_{2}^{\on{KZ}}(z) \cap V_{i,j}$.

\begin{lem}
\label{lem g3k}
We have $\mc P_{2,0} = V_{2,0}$,
$\mc P_{0,2} = V_{0,2}$. The space $\mc P_{1,1}$ is generated by the elements
$v_i\wedge w_j$ with $i\ne j$,    and by the elements
$\sum_{i=1}^g c_{i} v_i\wedge w_i$ such that 
$\sum_{i=1}^g c_{i} = 0$.
\qed

\end{lem}

\begin{thm}
\label{thm 6.9}

We have 
\bean
\label{2.0}
\sum_{a=1}^{2g+1}\tilde C_a(\mc P_{2,0})=\mc P_{1,1}, \qquad
\sum_{a=1}^{2g+1} \tilde C_a(\mc P_{1,1})=\mc P_{0,2}.
\eean

\end{thm}

\begin{proof}

 We have
\bea
&&
\tilde C_a(v_1\wedge v_2)= (w_1 +z_a^{p} w_2+ \dots + z_a^{p(g-1)}w_g)\wedge (-z_a^{p}v_{1} + v_2) =
\\
&&
- v_2\wedge w_1+ z_a^{p}(v_1\wedge w_1 - v_2\wedge w_{2})
+ z_a^{2p}(v_1\wedge w_2 - v_2\wedge w_3)
+ \dots +z_a^{pg} v_1\wedge w_g\,.
\eea
Clearly, these vectors for $a=1,\dots, g+1$ span the subspace
\bea
\langle v_2\wedge w_1,\ v_1\wedge w_1 - v_2\wedge w_2, \
v_1\wedge w_2 - v_2\wedge w_3, \ \dots, \ v_1\wedge w_g \rangle.
\eea
Similarly, we show that the vectors $\tilde C_a(v_i\wedge v_{i+1})$ with $i=2,\dots, g-1$ span the subspace
\bea
\langle v_{i+1}\wedge w_1,\ v_i\wedge w_1 - v_{i+1}\wedge w_2, \
v_i\wedge w_2 - v_{i+1}\wedge w_3, \ \dots, \ v_i\wedge w_g\rangle.
\eea
Clearly these subspaces span $\mc P_{1,1}$.

To show the second  equality in \eqref{2.0} we prove a stronger statement,
\bean
\label{2002}
\sum_{a\ne b}  \tilde C_{a} \tilde C_{b} (V_{2,0}) = V_{0,2}.
\eean
Indeed, for $1\leq i<j\leq g$, we have
\bea
\tilde C_a \tilde C_b (v_i\wedge v_j) = (z_b^{p(i-1)}z_a^{p(j-1)}
-z_a^{p(i-1)}z_b^{p(j-1)})\, K_b\wedge K_a\,.
\eea
It is easy to see that the elements $ K_b\wedge K_a$ span $\mc V_{0,2}$.
\end{proof}

\begin{thm}
\label{thm r=2 all GM sol}
Let $p> 2g+1$. Then all solutions of the \KZ/ equations in characteristic 
$p$ with values in $\mc P_{2}^{\on{KZ}}(z)\subset \wedge ^2 V$ and $\ka=2$
are generated by
the $p$-hypergeometric solutions $M^{\ell_1,\ell_2}(z)$, $1\leq \ell_1< \ell_2\leq g$.
\end{thm}

\begin{proof}

On the one hand, these $p$-hypergeometric solutions are linearly
 independent over the field $\K(z)$ by Theorem \ref{thm lin ind} for $r=2$,
hence they  span the space $V_{0,2}$
of dimension $\binom{g}{2}$. On the other hand,  $\mc K_2 \cap 
\mc P_2^{\on{KZ}}(z) = V_{0,2}$ is the same space by Theorem \ref{thm int c}. This proves the theorem.
\end{proof}

\begin{thm}
\label{thm r=2 all sol}

For all prime $p$ with finitely many exceptions, 
 all solutions of the \KZ/ equations in characteristic $p$ with values in $\Sing L^{\ox(2g+1)}[2g-3]$ and $\ka=2$
are generated by
the $p$-hypergeometric solutions $N^{\ell_1,\ell_2}(z)$, $1\leq \ell_1< \ell_2\leq g$.

\end{thm}

\begin{proof}
On the one hand, under these assumptions on $p$, the map
$ T(z)\vert_{ \mc P_2^{\on{KZ}}(z)} :  \mc P_2^{\on{KZ}}(z)   \to 
 \Sing L^{\ox(2g+1)}[2g-3]$ is an isomorphism in characteristic $p$ between the \KZ/ equations  with $\ka=2$ and
  values in
$\mc P_2^{\on{KZ}}(z)$ and  the KZ equations  with $\ka=2$ and values in
$\Sing L^{\ox(2g+1)}[2g-3]$. Hence these two systems of equations have 
the same dimension $\binom{g}{2}$ of the space of solutions.
On the other hand, the $p$-hypergeometric solutions $N^{\ell_1,\ell_2}(z)$, $1\leq \ell_1< \ell_2\leq g$, 
span the space  of dimension $\binom{g}{2}$.
This proves the theorem.
\end{proof}

\subsection{Solutions in $\Sing L^{\ox(2g+1)}[2g-3]$ in characteristic $p$ for $\ka=-2$}
\label{sec 8.6}

Let $C_a^-(z)$, $a=1,\dots,2g+1,$ denote the $p$-curvature operators of the \KZ/ equations with values in 
\\
$\Sing L^{\ox(2g+1)}[2g-3]$ and $\ka=-2$.  
Fix $z$ with distinct coordinates.  Recall the $p$-curvature operators 
$C_a(z)\vert_{\mc P_2^{\on{KZ}}(z)} : \mc P_2^{\on{KZ}}(z)\to \mc P_2^{\on{KZ}}(z)$ 
for $\ka=2$ of Section \ref{sec 8.2}.

For all prime $p$ with finitely many exceptions, the space
$\Sing L^{\ox(2g+1)}[2g-3]$ with operators $C_a^-(z)$ is isomorphic to the space $(\mc P_2^{\on{KZ}}(z))^*$ 
(dual to the space $\mc P_2^{\on{KZ}}(z)$) 
with  operators $- \left(C_a(z)\vert_{\mc P_2^{\on{KZ}}(z)}\right)^*$, dual to the operators 
$ C_a(z)\vert_{\mc P_2^{\on{KZ}}(z)}$ with coefficient $-1$.
 This follows from formula \eqref{do} and  the isomorphism
$ T(z)\vert_{\mc P_2^{\on{KZ}}(z) } : \mc P_2^{\on{KZ}}(z)
  \to  \Sing L^{\ox(2g+1)}[2g-3]$ for $\ka=2$ in characteristic $p$. 

\begin{lem}

The space 
\bea
\mc K^{-}:=\cap_{a=1}^{2g+1} \on{ker}\left(-\left(C_a(z)\vert_{\mc P_2^{\on{KZ}}(z)}\right)^*
: (\mc P_2^{\on{KZ}}(z))^*  \to  (\mc P_2^{\on{KZ}}(z))^*\right)
\eea
 has dimension $\binom{g}{2}$.

\end{lem}

\begin{proof}

The space $\mc K^-$ is the annihilator of the subspace 
$\sum_{a=1}^{2g+1} \on{Im}(C_a(z) : \mc P_2^{\on{KZ}}(z)\to \mc P_2^{\on{KZ}}(z))$.
By Theorem \ref{thm 6.9}, the annihilator has dimension of the space $\mc P_{2,0}$ which is $\binom{g}{2}$.
\end{proof}

\begin{cor}
\label{cor dim 2}

For all prime $p$ with finitely many exceptions, the space of solutions in characteristic $p$ of the \KZ/ equations with values 
in $\Sing L^{\ox(2g+1)}[2g-3]$ and $\ka=-2$ is of dimension $\binom{g}{2}$.

\end{cor}

\begin{thm}
\label{conj -2}
In characteristic $p$, all solutions of the \KZ/ equations with values 
in 
\\
$\Sing L^{\ox(2g+1)}[2g-3]$ and $\ka=-2$ are generated by the $p$-hypergeometric solutions
$\bar N^{\ell_1,\ell_2}(z) $, see formula \eqref{bar N}.

\end{thm}

\begin{proof} 
By Corollary \ref{cor t vs b N}, the $p$-hypergeometric solutions generate a space of dimension
$\binom{g}{2}$. Now the theorem follows from Corollary \ref{cor dim 2}.
\end{proof}

\section{$p$-curvature on $\wedge^3V$}
\label{sec 9}

Choose a good basis of $V$.

\subsection{Zeros of $p$-curvature on $V_{r,0}$ and $V_{0,r}$}

\begin{lem}
\label{lem CP}

Let $\mc P_r^{\on{KZ}}(z) \subset \wedge^rV $ be the primitive subspace. 
 Then $\tilde C_a(\mc P_r^{\on{KZ}}(z)) \subset \mc P_r^{\on{KZ}}(z)$
for every $a$.
\qed

\end{lem}

We have a decomposition
$\mc P_r^{\on{KZ}}(z) = \oplus_{k=0}^r \mc P_{k,r-k}$, $\mc P_{k,r-k} = \mc P_r^{\on{KZ}}(z) 
\cap V_{k,r-k}$. Denote 
\bea
\mc K_{k,r-k} = \cap_{a=1}^{2g+1} \on{Ker}(\tilde C_a:\mc P_{k,r-k} \to \mc P_{k-1,r-k+1}).
\eea
Clearly $\mc{K}_{0,r}=\mc P_{0,r}= V_{0,r}$.

\begin{prop}
\label{thm Kro}

For any $r$, $1\leq r\leq g$, we have $\mc
K_{r,0}=0$.

\end{prop}

\begin{proof}

We need the following  linear algebra lemma. Let $y_1,\dots, y_n$ be a basis of a vector space $Y$ and 
$y_1^*,\dots,y_n^*$ the dual basis of the dual space $Y^*$.  For $\om \in \wedge ^kY$
and  $y^*\in Y^*$,
let $i_{y^*}\om$ denote the interior product of $y^*$ and $\om$.

\begin{lem}
\label{lem aux}
Let $\om\in \wedge^k Y$ and $i_{y_a^*}\om =0$ for $a=1,\dots,n$. Then $\om=0$.
\qed

\end{lem}

Let  $\om = \sum_{1\leq i_1<\dots<i_r\leq g} c_{i_1,\dots,i_g}v_{i_1}\wedge\dots\wedge v_{i_r}$
be an element of $V_{r,0}$. Then
\bea
\tilde C_a(\om) 
&=&
 K_a\wedge \sum_{1\leq i_1<\dots<i_r\leq g} c_{i_1,\dots,i_g}\sum_{j=1}^r z_a^{p(i_j-1)}
(-1)^{j-1}v_{i_1}\wedge \dots \wedge \widehat{v_{i_j}}\wedge \dots\wedge v_{i_r}\,
\\
&=&
K_a \wedge i_{L_a} \om.
\eea
We have $\tilde C_a(\om) =0$ for all $a$ if and only if $ i_{L_a} \om =0$ for all $a$. 
By Lemma \ref{lem aux}. this implies that $\om=0$ since the  vectors
$L_a$, $a = 1, \dots, g$, are linearly independent. 
\end{proof}

\subsection{Intersection of kernels in $V_{2,1}$}

\begin{lem}
\label{lem a.1}

For  $q\geq 3$, the intersection of kernels
of the reduced $p$-curvature operators $\tilde C_m : V_{2,1} \to V_{1,2}$,
$m=1,\dots, 2g+1$, equals zero.
\end{lem}

\begin{proof}

A vector $\al \in V_{2,1}$ has the form
\bea
\al= \sum_{1\leq a<b\leq g} \sum_{c=1}^g  d_{a,b;c} \, v_a \wedge v_b \wedge w_c\,
\eea
We set  $d_{a,b;c}=0$ if $ b\leq a$ or if any of the indices $a,b,c$ is not in the set 
$\{1,\dots,g\}$. 
Similarly to previous reasonings, we observe that 
$\al$ lies in the intersection of kernels if its coefficients
$d_{a,b:c}$  satisfy the system of equations
\bean
\label{eq 21}
d_{s-k,i;j}-d_{i,s-k;j}+d_{i,s-j;k}-d_{s-j,i;k}=0  
\eean
where the equations are labeled by the parameters $s,k,i,j$
with the only condition  $j<k$.  Equation \eqref{eq 21}  
comes as the coefficient of $z_m^{p(s-2)} v_i\wedge w_j\wedge w_k$
 in the vector $\tilde C_m\al$. 
 
 This system of equations implies that all coefficients $d_{a,b;c}$ are equal to zero.
Indeed, if $c>a$ and $b>a$, we choose $i=b,\,k=c,\,j=c-a, \,s=c$, and equation
\eqref{eq 21} says that $d_{a,b;c}=0$. If  $c<b$ and $b>a$,
we choose
 $i=a,\, j=c,\, s=c+g+1,\, k=c+g-b+1$, and equation \eqref{eq 21} says that $d_{a,b;c}=0$.
\end{proof}

\subsection{Intersection of kernels in $V_{1,2}$}

Introduce  the following 
$2g-2$ elements of $V_{1,2}$:
\bean
\label{alb}
\al_k= \sum_{a=1}^k\sum_{b=k+1}^g v_{a+b-k-1}\wedge w_a\wedge w_b\,,
\qquad
\beta_k= \sum_{a=1}^k\sum_{b=k+1}^g v_{a+b-k}\wedge w_a\wedge w_b\,,
\eean
where  $k=1,\dots, g-1$.  Denote by $X_{1,2}$ the span of these elements.

Let $Y_{1,2}$ be the intersection of kernels 
of the reduced $p$-curvature operators $\tilde C_m : V_{1,2} \to V_{0,3}$, 
$m=1,\dots, 2g+1$. 

\begin{lem}
\label{lem A.5} 

The elements $\al_k, \beta _k$, \, $k=1,\dots, g-1$,
are linearly independent and $X_{1,2}\subset Y_{1,2}$.

\end{lem}

\begin{proof}

Consider the lexicographical ordering on the elements $v_a \wedge w_b\wedge w_c$. Then
the lexicographically minimal summand of $\al_k$ is 
$v_1 \wedge w_1\wedge w_k$ and 
the lexicographically minimal summand of $\beta_k$ is 
$v_2 \wedge w_1\wedge w_k$. These minimal summands are distinct. Hence
$\al_k, \beta _k$, \, $k=1,\dots, g-1$, are linearly independent.

 We have 
\bea
\tilde C_m \al_k  
&=&
 \left(\sum_{i=1}^g  z_i^{p(i-1)} w_i\right) \wedge
\sum_{a=1}^k\sum_{b=k+1}^g z_m^{p(a+b-2)} 
 w_a\wedge w_b\,
\\
&=&
 \left(\sum_{a=1}^k  z_m^{p(a-1)} w_a
 +\sum_{b=k+1}^g z_m^{p(b-1)} w_b\right) \wedge
\left(\sum_{a=1}^k z_m^{p(a-1)} w_a\right)
\wedge
\left(\sum_{b=k+1}^g z_m^{p(b-1)} w_b\right) =0.
\eea
Similarly we show that  $\tilde C_m \beta_k  =0$.

\end{proof}

\begin{lem}
\label{lem K 12}

 The intersection $X_{1,2}\cap \mc P_{1,2}$ has dimension $g-2$.

\end{lem}

\begin{proof}

Recall that $\mc P_{1,2}$ is the kernel of the map
$f: V_{1,2} \to V_{g-1,g}$ of multiplication by the $(g-2)$-th wedge-power of
the element $\sum_{i=1}^g v_i\wedge w_i$. 
The elements
\bea
\ga_i = w_i \wedge \left( \wedge_{j\ne i} (v_j\wedge w_j)\right), \qquad i=1,\dots, g, 
\eea
form a basis of $V_{g-1,g}$.  

We  have $f(X_{1,2}) = V_{g-1,g}$ because of 
$f(\beta_i) = - (g-i)! \ga_i$, $i=1,\dots, g-1$, and $f(\beta_{g-1})=(g-1)! \ga_g$.
Since  $f(X_{1,2}) = V_{g-1,g}$, we obtain that $\dim 
X_{1,2}\cap \mc P_{1,2} = 2g-2-g = g-2$.
\end{proof}

\begin{thm}
\label{thm ker 12}

We have $X_{1,2}=Y_{1,2}$.

\end{thm}

\begin{proof}

Define the set 
\bea
S = \{v_1\wedge w_1 \wedge w_{k+1}, \, 
v_2\wedge w_1 \wedge w_{k+1} \mid
k=1,\dots, g-1\} .
\eea
The elements of $S$ are
the lexicographically minimal summands of $\al_k$ and $\beta_k$ with $k=1,\dots, g-1$.
The theorem clearly follows from the following lemma.

\begin{lem}
\label{lem 9.8}

If $\al$ is a nonzero element of $Y_{1,2}$, then the lexicographically minimal summand of
$\al$ belongs to the set $S$. 

\end{lem}

\smallskip
\noindent
{\it Proof of lemma.}
Let  
\bea
\al = \sum_{1\leq b<c\leq g}\sum_{a=1}^g  d_{a;b,c} \,v_a \wedge w_b \wedge w_c.
\eea
We set  $d_{a;b,c}=0$ if $ c\leq b$ or if any of the indices $a,b,c$ is not in the set 
$\{1,\dots,g\}$. 
Similarly to previous reasonings, we observe that 
$\al$ lies in the intersection of kernels if its coefficients
$d_{a;b,c}$  satisfy the system of equations
\bean
\label{eq 12}
d_{s-i;j,k}-d_{s-j;i,k}+d_{s-k;i,j}=0  
\eean
where the equations are labeled by the parameters $s,i,j,k$
with the only condition  $i< j<k$.  Equation \eqref{eq 12} comes as 
the coefficient of $z_m^{p(s-2)} w_i\wedge w_j\wedge w_k$ in 
$\tilde C_m \al$. 

\medskip

\noindent
{\it Claim 1.}
If $a>c>b$, then $d_{a;b,c}=0$. 

\smallskip

\noindent
Indeed, choose
$i=b, \,j=c, \,s=c+g+1,\, k=c+g-a+1$. Then \eqref{eq 12}
implies that $d_{a;b,c}=0$.

\medskip

\noindent
{\it Claim 2.}
If $1<b<c$, then $v_a\wedge w_b\wedge w_c$ cannot be the lexicographically minimal summand 
of $\al$.

\smallskip

\noindent
Indeed, choose $ j=b,\,i=1,\, s=a+1,\,k=c$. Then equation
\eqref{eq 12} takes the form 
\bea
d_{a;b,c}-d_{a+1-b;1,c}+d_{a+1-c;1,b}=0.
\eea
Claim 2 follows since the first term is the lexicographically largest of these three terms.

\medskip

\noindent
{\it Claim 3.}
If $b=1$ and $2<a<c+1$, 
then $v_a\wedge w_1\wedge w_c$ cannot be the lexicographically minimal summand 
of $\al$.

\smallskip

\noindent
Indeed, choose 
$s=c+2, \,j=c+2-a,\, i=1,\,k=c$. Then  \eqref{eq 12} takes the form
\bea
  d_{c+1;c+2-a,c}  -d_{a;1,c}+d_{2;1,c+2-a}=0.
\eea
In this equation, the first coefficient  is zero by Claim 1, and the element 
$v_2\wedge w_1\wedge w_{c+2-a}$
is lexicographically smaller that $v_a\wedge w_1\wedge w_c$. Claim 3 follows.

Claims 1-3 imply the lemma.
\qed

\end{proof}

\begin{cor}
The space of solutions modulo $p$ with values in $\mc P_3^{\on{KZ}}(z)$
 and $\ka=2$ is of dimension $\binom{g}{3} + g-2$.  
The space of solutions has a $\binom{g}{3}$-dimensional subspace generated the
 $p$-hypergeometric solutions $w_{i_1}\wedge w_{i_2}\wedge w_{i_3}$
 with $1\leq i_1<i_2<i_3\leq g$. 
  The projection
of any solution to $\mc P_3^{\on{KZ}}(z)/ \mc P_{0,3}$ lies in $X_{1,2}\cap \mc P_{1,2}$.
 
\end{cor}

\begin{cor}
\label{cor a6}

For all prime $p$ with finitely many exceptions, the space of solutions of the \KZ/ equations with 
values in  $\Sing L^{\ox (2g+1)}[2g-5]$ and $\ka=2$ in characteristic $p$ has dimension 
$\binom{g}{3} + g-2$. That  space has 
 a $\binom{g}{3}$-dimensional subspace generated the
 $p$-hypergeometric solutions $N^{\ell_1,\ell_2,\ell_3}$
 with $1\leq \ell_1<\ell_2<\ell_3\leq g$.

\end{cor}

\subsection{Intersection of kernels in $V_{1,r-1}$} 

We expect that the intersection  in $V_{1,r-1}$ of the kernels of the $p$-curvature operators is of
 dimension
$(r-1)\binom{g-1}{r-2}$ and has a basis consisting of elements
 analogous to the elements $\al_k,\beta_k$, $k=1,\dots, g-1$, in \eqref{alb}.

For simplicity we define these elements in $V_{1,3}$ for $g>3$.  For
$i=0, 1,2$ and  $1\leq k_1<k_2\leq g-1$, let
\bea
\al_{i;k_1,k_2} =
\sum_{a=1}^{k_1}\sum_{b=k_1+1}^{k_2}\sum_{c=k_2+1}^g v_{a+b+c-k_1-k_2-2+i}
\wedge w_a\wedge w_b\wedge w_c,
\eea
cf. \eqref{alb}. It is easy to see that these elements lie in the intersection of the kernels.

\vsk.2>

A general problem is to describe the dimension of the intersection of kernels 
in $\mc P_r^{\on{KZ}}(z)$ for arbitrary $r$ and determine the dimension of the space of solutions of the \KZ/ equations
 in $\Sing L^{\ox(2g+1)}[2g+1-2r]$ for $\ka=\pm 2$.

\appendix

\section{Kodaira-Spencer maps and $p$-curvature}
\label{sec 10}

\subsection{Cohomology classes on $C(z)$}

Fix $z=(z_1,\dots, z_{2g+1})$ with distinct coordinates. 
Recall $\Psi(t,z)=\prod_{a=1}^{2g+1}(t-z_a)$.
Consider the following differential 1-forms on the hyperelliptic 
curve $C(z)$:
\bea
\om_i 
&=&
\frac{\Psi(t,z)^{-1/2}}{t-z_i} dt, \qquad i=1,\dots,2g+1,
\\
\mu_k
&=&
\Psi(t,z)^{-1/2} t^{k-1} dt, \qquad k=1,\dots,g.
\eea
The forms $\om_i$ have zero residues at all points of $C(z)$ and define elements $\left[\om_i\right] \in H^1(C(z))$. The forms
$\left[\om_i\right]$, $i=1,\dots, 2g+1$, span $H^1(C(z))$ and satisfy the relation
$\sum_{i=1}^{2g+1}\left[\om_i\right] = 0$, see \cite[Theorem 4.2]{VV1}.
The forms $\mu_k$, $k=1,\dots,g,$ are regular on $C(z)$ and give a basis of the Hodge subspace $F^1\subset
H^1(C(z))$, \cite[Lemma 4.4]{VV1}.

The derivatives of the elements $\left[\mu_k\right]$ with respect to the Gauss-Manin connection are given by the formula
\bean
\label{dmg}
\nabla_{\!\!\frac{\partial}{\partial z_i}}
[\mu_k]= \frac 12\big([\mu_{k-1}] + z_i[\mu_{k-2}]+\dots + z_i^{k-2}[\mu_1] + z_i^{k-1}[\om_i]\big),
\eean
\cite[Lemma 4.5]{VV1}.

\vsk.2>

Consider the linear functions $(\cdot, \mu_k)$, $k=1,\dots,g$, on $H^1(C(z))$ where $(\cdot,\cdot)$ is the Poincar\'e
pairing. These linear functions are zero on $F^1$ and define a coordinate system on $H^1(C(z))/F^1$. 
Let $\nu_k$, $k=1,\dots, g$, be the dual basis of $H^1(C(z))/F^1$, $(\nu_i,\om_j)=\delta_{ij}$.
For any element $\ga\in H^1(C(z))$, its projection $\bar\ga$ to $H^1(C(z))/F^1$ equals the linear combination
\bean
\label{gac}
(\ga,\mu_1)\nu_1+ (\ga,\mu_2)\nu_2+\dots +(\ga,\mu_g) \nu_g\,.
\eean
\vsk.2>

 Consider the direct sum  $W=F^1\oplus H^1(C(z))/F^1$ with basis 
$\mu_k, \nu_k$, $k=1,\dots,g$.  The element
 $\sum_{k=1}^g \nu_k\wedge \mu_k$ in $\wedge^2W$ corresponds
 to the Poincare pairing element.

\vsk.2>

The Poincar\'e pairing of $[\om_i]$ and $[\mu_k]$ is given by the formula 
\bean
\label{ommu}
([\om_i], [\mu_k]) \,=\, 
-\,\frac{2z_i^{k-1}}{D_i(z)}\,,
\eean
where $D_i(z) = (z_i-z_1)\dots (z_i-z_{i-1})(z_i-z_{i+1}) \dots (z_i-z_{n})$, \cite[Theorem 4.7]{VV1}.

\subsection{Kodaira-Spencer maps}

Recall the Kodaira-Spencer maps
\bean
\label{ks}
 \text{KS}_a(z) : F^1 \to H^1(C(z))/F^1, \qquad  a=1,\dots,z_{2g+1}. 
 \eean
 By definition, the map $\on{KS}_a(z)$ sends the class $[\mu_k]$ of a regular 1-form to the
projection  $\overline{\nabla_{\!\!\frac{\der}{\der z_a}} [\mu_k]}$ in
$ H^1(C(z))/F^1$ of the class   $\nabla_{\!\!\frac{\der}{\der z_a}} [\mu_k]$.  

We extend the Kodaira-Spencer map to a map $W\to W$ by setting $\on{KS}_a(z)\vert_{H^1(C(z))/F^1} = 0$.
 
 \begin{lem}
 For $a=1,\dots, 2g+1$, the Kodaira-Spencer map $\on{KS}_a(z) :W\to W$ is defined by the formula:
 \bean
 \label{KSe}
 &&
 \on{KS}_a(z) \ :\  \sum_{k=1}^g (b_k\mu_k+c_k \nu_k) \
 \mapsto\
   \frac{-1}{D_a(z)}\left(\sum_{k=1}^g b_k z_a^{k-1}\right)\left(\sum_{\ell=1}^gz_a^{\ell-1}\nu_k\right).
 \eean

 \end{lem}

\begin{proof}
This formula follows from \eqref{dmg}, \eqref{gac}, and \eqref{ommu}.
\end{proof}

Define the reduced Kodaira-Spencer map by the formula
 \bean
 \label{KSm}
\widetilde{\on{KS}}_a(z) 
\  :\  \sum_{k=1}^g (b_k\mu_k+c_k \nu_k) \
 \to\
\left(\sum_{k=1}^g b_k z_a^{k-1}\right)\left(\sum_{\ell=1}^gz_a^{\ell-1}\nu_k\right).
 \eean
The reduced map has the same kernel and image as the map $\on{KS}_a(z)$.

\begin{cor}
\label{cor KS-p}  

Formula \eqref{KSm} for the reduced Kodaira-Spencer maps 
$\widetilde{\on{KS}}_a :W\to W$, $a=1,\dots, 2g+1$, can be identified with the formula \eqref{atv}
for the reduced $p$-curvature operators $\tilde C_a(z): V \to  V$, $a=1,\dots,g$, if
 $\mu_k, \nu_k, z_a$ in formula \eqref{KSm} are identified with $v_k, w_k, z_a^p$
 in Section \ref{sec 8.2}.

\end{cor}

Notice that the $p$-curvature operators are defined in characteristic $p$,
 while the Kodaira–Spen\-cer maps are objects of deformation theory and do not require characteristic restrictions to be defined.

\begin{rem}
\label{rem Katz}  

The Katz $p$-curvature formula, \cite[Theorem 3.2]{K2} presents a $p$-curvature operator as a composition
of three maps where the middle map is the Kodaira-Spencer map and the first and third maps 
are suitable Cartier operators, see also \cite[Fromula (4.36)]{VV1}. The identification in Corollary \ref{cor KS-p}
hides the first and third operators of that composition by  choosing a good basis in Section \ref{sec 8}.

\end{rem}

The identification in Corollary \ref{cor KS-p} enables us to extend the results from Sections \ref{sec 8} and \ref{sec 9} to describe 
the intersection of the kernels of Kodaira–Spencer maps acting on \(\wedge^r W\). For instance, the results 
in Section \ref{sec 9} imply the following corollary.

\begin{cor} The intersection of kernels of the Kodaira-Spencer maps
$\on{KS}_a(z) :\wedge^3 W \to \wedge^3 W$ is of dimension $\binom{g}{3} + 2g-2$, and
the intersection of kernels with the primitive part $\mc P^3 \subset \wedge^3 W$
is of dimension
$\binom{g}{3} + g-2$.

\end{cor}

\end{document}